%% file: submission.tex
\documentclass{article}

\usepackage{amsmath,amssymb,amsthm,amsfonts,mathabx,mathtools}
\usepackage[all]{xy}

\usepackage[margin=1.25in]{geometry}
\usepackage{titling}
\usepackage{fancyhdr,indentfirst,color,hyperref}

\newtheorem{lemma}{Lemma}
\newtheorem*{question}{Question}
\newtheorem{criterion}{Criterion}
\newtheorem{theorem}{Theorem}
\newtheorem{proposition}{Proposition}
\newtheorem{remark}{Remark}
\newtheorem{corollary}{Corollary}

\def\Hom{\operatorname{Hom}}
\def\Eu{\operatorname{Eu}}
\def\Res{\operatorname{Res}}
\def\det{\operatorname{det}}
\def\ev{\operatorname{ev}}
\def\ft{\operatorname{ft}}
\def\coeff{\operatorname{Coeff}}
\def\rank{\operatorname{rank}}
\def\mod{\operatorname{mod}}

\def\Span{\operatorname{span}}
\def\virt{\operatorname{virt}}
\def\Id{\operatorname{Id}}
\def\C{\mathbb{C}}
\def\J{\mathcal{J}}
\def\O{\mathcal{O}}
\def\II{\mathfrak{I}}
\def\JJ{\mathfrak{J}}
\def\KK{\mathfrak{K}}
\def\L{\mathcal{L}}
\def\wT{\widetilde{T}}
\def\M{\widebar{\mathcal{M}}}
\def\t{\mathbf{t}}
\def\f{\mathbf{f}}
\def\g{\mathbf{g}}
\def\e{\mathbf{e}}
\def\one{\mathbf{1}}

\title{Quantum K-theory of flag varieties via non-abelian localization}
\date{}
\author{Xiaohan YAN}

\begin{document}

\maketitle

\begin{abstract}
In this paper, we reconstruct explicitly the generating function of genus-zero K-theoretic permutation-invariant Gromov-Witten invariants, known as the big $\mathcal{J}$-function, for any partial flag variety. The reconstruction may start with any Weyl-group-invariant value of the well-understood big $\mathcal{J}$-function of an associated toric variety. We generalize the recursive method \cite{Givental:perm2}, based on torus fixed point localization, to deal with non-isolated one-dimensional toric orbits, through incorporating ``balanced broken orbits’’ into consideration and subsequently proving a vanishing result of their contribution. Furthermore, we extend the study to twisted generating functions of the flag variety, and demonstrate properties including a non-abelian quantum Lefschetz theorem and a duality between level structures. In the end, we construct a K-theoretic mirror in terms of Jackson-type integrals.
\end{abstract}



\ \\

\section{Introduction} 
\par The study of quantum K-theory was initiated as an analogue of quantum cohomology, or Gromov-Witten (GW) theory, at the beginning of this century by the foundational papers of Givental \cite{Givental:WDVV} and of Lee \cite{Lee}. Later, as the first step toward understanding the quantum K-theory of complete flag varieties, the two authors in their joint paper \cite{Givental-Lee} employed an \emph{ad hoc} method to compute a generating function of genus-zero \textbf{K-theoretic GW invariants}, and related the function to solutions of certain finite-difference equations. Generally speaking, the K-theoretic GW invariants are realized as virtual holomorphic Euler characteristics of vector bundles over $\M_{g,m}(X,d)$, the moduli space of stable maps to the target variety $X$. 

\par Various progress has been made in this field since then. In genus-zero case, the interests were re-ignited when the enumerative geometry of the K-theoretic GW invariants was related to integrable systems and representation theory. Generating functions of a variant of such invariants (defined using quasi-map-related compactifications) were found to appear as R-matrices which arise, on moduli of vacua such as Nakajima quiver varieties, in the stable envelope construction \cite{MO,Okounkov,AO}. They play a part in the 3D $\mathcal{N}=4$ mirror symmetry \cite{RSVZ,RSZV,S-Z,Dinkins} as well. These generating functions, called the vertex functions, are explicitly computed by Pushkar et al. \cite{PSZ} for grassmannians and Koroteev et al. \cite{KPSZ} for flag varieties, and are shown to appear naturally in the context of quantum XXZ spin chains and tRS models in theoretical physics. Connections between quantum K-theory and theoretical physics, especially the study of various super-symmetric gauge theories, have also been discovered in, for instance, the papers of \cite{NS1,NS2} and more recently of \cite{JMNT,UY}. Moreover, inspired by the Chern-Simons terms, Ruan and Zhang \cite{R-Z} introduced the level structures to quantum K-theory and discovered resemblances of $J$-functions to the mock theta functions, prompting new progress in the field such as \cite{RWZ}.

\par Genus-zero K-theoretic GW invariants in the Givental-Lee theory are recorded in generating functions called the K-theoretic big $\J$-functions, which are the main subject of this paper. While such functions are well-understood in the case of toric varieties, not much is known for the more general ``non-abelian’’ GIT quotients of vector spaces like quiver varieties. 
\begin{question}
How can we describe (the image of) the big $\J$-function of the GIT quotient $R//G$, where $R$ is a vector space and $G$ is a complex reductive group?
\end{question}
A complete answer to the Question seems too remote, but a first step in this direction has been made in literature using the quasi-map methods. In \cite{Wen} for instance, Wen computed the quasi-map small $I$-function, which by the wall-crossing results \cite{Zhang-Zhou} by Zhang and Zhou should represent one point on the image of the big $\J$-function. 

\par In order to recover the entire big $\J$-function, however, this is not sufficient. Indeed, for target spaces whose K-rings are generated by line bundles, the entire image can indeed be recovered from any one point of itself with the help of its intrinsic $\mathcal{D}_q$-module structure. Yet, it is not the case for other target spaces, like grassmannians, partial flag varieties, general quiver varieties or GIT quotients of vector spaces. In this paper, we provide, in permutation-invariant settings, a complete answer to the Question for partial flag varieties, using a different method purely in terms of stable map compactification. The method relies on a deeper understanding of the torus-equivariant geometry of the target spaces as well as certain associated toric varieties which we call the abelian quotients.

\subsection{Main theorems} \label{subsec11}
\par In this part, we state the main theorems and provide necessary definitions and constructions along the way. The more detailed ones will be postponed to later sections. 

\par Let $X$ be the flag variety $\text{Flag}(v_1,\cdots, v_n; N)$ parametrizing sequences of subspaces $V_1\subset V_2\subset\ldots\subset V_n\subset\mathbb{C}^N$ of $\mathbb{C}^N$ with $\dim V_i = v_i$. It can be regarded as a GIT quotient
\[
X = R//G = \Hom(\mathbb{C}^{v_1},\mathbb{C}^{v_2})\oplus\cdots\oplus\Hom(\mathbb{C}^{v_n},\mathbb{C}^N) // GL(v_1)\times\cdots\times GL(v_n)
\]
where the stability condition demands that each linear transformation in $\Hom(\mathbb{C}^{v_i},\mathbb{C}^{v_{i+1}})$ is injective. Let $\varphi = (\varphi_1, \cdots, \varphi_n)\in R$ and $g = (g_1,\cdots,g_n)\in G$, then the (left) action is defined so that $g\cdot\varphi = (g_2\varphi_1g_1^{-1}, \cdots, \varphi_ng_n^{-1})$. Abusing notations, we denote by $V_i \ (i=1,\cdots,n)$ the $i$-th tautological bundle of $X$. $H^2(X;\mathbb{Z})$ is generated by $\{c_1(V_i)\}_{i=1}^n$, and we denote by $\{Q_i\}_{i=1}^n$ the corresponding Novikov variables. That is to say, a curve class $d\in H_2(X;\mathbb{Z})$ is marked by the monomial $\prod_{i}Q_i^{d_i}$ with $d_i = \int_d c_1(V_i)$.

\par The associated abelian quotient of $X$, which is the GIT quotient of the same vector space by maximal torus $S\subset G$, is denoted by $Y$:
\[
Y = R//S = \Hom(\mathbb{C}^{v_1},\mathbb{C}^{v_2})\oplus\cdots\oplus\Hom(\mathbb{C}^{v_n},\mathbb{C}^N)//(\mathbb{C}^\times)^{v_1}\times\cdots\times(\mathbb{C}^\times)^{v_n}.
\]
Note that the stability conditions of $X$ and $Y$ are different. As convention, we take $v_{n+1} = N$. $Y$ may be pictured as a tower of fiber bundles 
\begin{equation*}
\xymatrix{
	 & Y\ar[d]\\
	 & \cdots\ar[d]\\
	(\mathbb{C}P^{v_{n-1}-1})^{v_{n-2}}\ar@{^(->}[r] & F_{n-1}\ar[d]\\
	(\mathbb{C}P^{v_{n}-1})^{v_{n-1}}\ar@{^(->}[r] & F_{n}\ar[d]\\
	 & (\mathbb{C}P^{N-1})^{v_{n}}.
}
\end{equation*}
We denote by $P_{is}$ the tautological bundle $\mathcal{O}(-1)$ on the $s$-th copy of $\mathbb{C}P^{v_{i+1}-1}$ in the $i$-th level ($1\leq i\leq n, 1\leq s\leq v_i$). These bundles pull back to $Y$ and generate its K-ring. The corresponding Novikov variables of $Y$ are denoted by $\{Q_{is}\}_{i=1,s=1}^{n\ \ \ v_i}$.

\par There is a natural torus action of $T = (\mathbb{C}^\times)^{N}$ on $X$ and $Y$ induced from the standard $T$-action on $\mathbb{C}^N$. We denote the equivariant parameters by $(\Lambda_1, \cdots, \Lambda_N)$. The rest of this Section is discussed all under $T$-equivariant settings unless otherwise stated. Parallel non-equivariant results may be obtained simply by taking $\Lambda_s = 1, \forall s$.

\par The K-theoretic big $\mathcal{J}$-function of $X$ is defined as
\[
\mathcal{J}^X(\t;q) = 1 - q + \t(q) + \sum_{m,d_i,\alpha} \prod_{i}Q_i^{d_i}\phi^\alpha\langle\frac{\phi_\alpha}{1-qL_0}, \t(L_1), \cdots, \t(L_n)\rangle^{S_m}_{0,m+1,d}, 
\]
where $\langle\cdot\rangle^{S_m}_{0,m+1,d}$ are the genus-0 K-theoretic permutation-invariant GW invariants (correlators) as appearing in \cite{Givental:perm1}. They are the $S_m$-invariant part of the virtual Euler characteristics of certain bundles over the moduli stack of stable maps $\widebar{\mathcal{M}}_{0,m+1}(X,d)$, over which we denote by $L_k\ (k=0,1,\cdots,m)$ the universal cotangent bundles at the marked points. The input $\t = \t(q) = \sum_k \t_kq^k$ is any Laurent polynomial in the formal variable $q$ with coefficients in the K-ring of $X$, and the output $\J^X(\t;q)$ is a rational function in $q$. The image of $\mathcal{J}^X$, denoted by $\mathcal{L}^X$, is an overruled cone in the so-called ``loop space'' of rational functions in $q$. The same definitions apply not only to $X$ but to any smooth projective varieties as well, in particular the abelian quotient $Y$.

\par From toric theory, it is not hard to prove that the $q$-rational function 
\[
\widetilde{J}^{tw,Y} = (1-q)\sum_{d\in \mathcal{D}} \prod_{i,s} Q_{is}^{d_{is}}\frac{\prod_{i=1}^n\prod_{r\neq s}^{1\leq r,s\leq v_i}\prod_{l=1}^{d_{is}-d_{ir}}(1-y\frac{P_{is}}{P_{ir}}q^l)}{\prod_{i=1}^{n-1}\prod_{1\leq r\leq v_{i+1}}^{1\leq s\leq v_i}\prod_{l=1}^{d_{is}-d_{i+1,r}}(1-\frac{P_{is}}{P_{i+1,r}}q^l)\cdot\prod_{1\leq r\leq N}^{1\leq s\leq v_n}\prod_{l=1}^{d_{ns}}(1-\frac{P_{ns}}{\Lambda_{r}}q^l)},
\]
represents a value of a twisted big $\mathcal{J}$-function of $Y$, with twisting of the type $(\text{Eu}, y^{-1}\bigoplus_{i=1}^n\bigoplus_{r\neq s}\frac{P_{ir}}{P_{is}})$ (see Section \ref{Jfunc}). The index set
\[
\mathcal{D}=\{(d_{is})_{i=1,s=1}^{n\ \ \ v_i}|d_{is}\geq 0\}.
\]
The parameter $y$ may be regarded as the equivariant parameter of an auxiliary fiber-wise $\C^\times$-scaling action on the twisting bundle, and is a common ingredient of twisted quantum K-theory for convergence reasons. Yet as we will see later, this specific parameter plays a crucial role when we relate the theory of $Y$ to that of $X$: redundant terms vanish only when $y=1$.

\par Denoting by $\J^{tw,Y}$ the big $\J$-function of $Y$ with such twisting, and by $\L^{tw,Y}$ its image cone, we have $\widetilde{J}^{tw,Y}\in\L^{tw,Y}$. $\L^{tw,Y}$ is invariant under certain operators. To start with, the ruling spaces of the K-theoretic image cone are invariant under operators of the form
\[
e^{D(Pq^{Q\partial_Q}, Q, q)}.
\]
for $D$ any Laurent polynomial, $P$ any line bundle over $Y$, and $Q$ the corresponding Novikov variable. The precise meaning and general theory of $Pq^{Q\partial_Q}$ will be further explained in Section \ref{subsecPFD}. Over the abelian quotient $Y$, $P$ may simply be taken as any monomial of the tautological bundles $P_{is}$. Moreover, for any Laurent expression $D$ as above, the image cone is also preserved by the operators of the following form, though not necessarily each ruling space \cite{Givental:perm8}:
\[
e^{\sum_{k>0}\frac{\Psi^k(D(Pq^{kQ\partial_Q}, Q, q))}{k(1-q^k)}}.
\]
Here the Adams operation $\Psi^k$ acts by $\Psi^k(Q) = Q^{|k|}$ and $\Psi^k(q) = q^k$. We denote by $\mathcal{P}$ the group generated by operators of the above two forms. Then $\L^{tw,Y}$ is invariant under $\mathcal{P}$.

\par The Weyl group $W = S_{v_1}\times\cdots\times S_{v_n}$ of $G$ (as in $X = R // G$) acts on the operators by permuting the subscripts $s$ in $P_{is}$ and $Q_{is}$. It is the $W$-invariant part $\mathcal{P}^W\subset\mathcal{P}$ that helps us reconstruct the big $\mathcal{J}$-function of $X$. In fact, only $W$-invariant expressions are well-defined over $X$. The central theorem of the paper is stated as follows, answering the Question at the beginning.
\begin{theorem} \label{0}
Elements in the orbit of $\widetilde{J}^{tw,Y}$ under the group $\mathcal{P}^W$ cover the entire cone $\mathcal{L}^X$, under the specialization of Novikov variables $Q_{is} = Q_i (\forall i,s)$ and the parameter $y = 1$.
\end{theorem}
The theorem effectively reconstructs the big $\mathcal{J}$-function of $X$ from one single point.
\begin{remark}
It will be seen from the proof of this paper that the reconstruction in Theorem \ref{0} does not necessarily start with $\widetilde{J}^{tw,Y}$, but any $W$-invariant point on $\L^{tw,Y}$ as well. Yet $\widetilde{J}^{tw,Y}$ is somehow a canonical choice, as the easiest point on $\L^{tw,Y}$ that we can write out.
\end{remark}
\begin{remark}
The reconstruction cannot be done directly on $\L^X$, essentially because $X$ does not have as many line bundles as $Y$ does to generate a big enough group $\mathcal{P}$.
\end{remark}
Using the identity element from the group $\mathcal{P}^W$, we recover the following result implied by Wen \cite{Wen} through quasi-map methods.
\begin{corollary} \label{MainCor}
\[
J^X = (1-q)\sum_{d\in \mathcal{D}} \prod_{i=1}^n Q_i^{\sum_{s}d_{is}}\frac{\prod_{i=1}^n\prod_{r\neq s}^{1\leq r,s\leq v_i}\prod_{l=1}^{d_{is}-d_{ir}}(1-\frac{P_{is}}{P_{ir}}q^l)}{\prod_{i=1}^{n-1}\prod_{1\leq r\leq v_{i+1}}^{1\leq s\leq v_i}\prod_{l=1}^{d_{is}-d_{i+1,r}}(1-\frac{P_{is}}{P_{i+1,r}}q^l)\cdot\prod_{1\leq r\leq N}^{1\leq s\leq v_n}\prod_{l=1}^{d_{ns}}(1-\frac{P_{ns}}{\Lambda_{r}}q^l)}
\]
is the small $J$-function of the flag variety $X$.
\end{corollary}
The small $J$-function is the distinguished value of $\mathcal{J}^X$ at $\t = 0$. To clarify the notations, $P_{is}$ in Corollary \ref{MainCor} denote now the K-theoretic Chern roots of the tautological bundles $V_i$ on $X$, for $i=1,2,\ldots,n$. We obtain the small $J$-function of $X$ without $T$-action by taking $\Lambda_s = 1$ for all $s$.

\par Theorem \ref{0} generalizes the result of \cite{Givental:perm8} where the K-ring of the target variety is required to be generated by line bundles (examples include toric varieties and complete flag varieties). Reconstruction of a different flavor is provided under similar assumptions of the K-ring in \cite{IMT}, which recovers the big quantum K-ring from the small $J$-function through analysis of certain $q$-shift operators.

\par Our strategy of proving Theorem \ref{0} is mainly divided into two parts:
\begin{itemize}
	\item to show that elements in the $\mathcal{P}^W$-orbit of $\widetilde{J}^{tw,Y}$ all lie on $\L^X$ under the given specialization;
	\item to show that such orbit covers the entire $\L^X$.
\end{itemize}
The first part of the proof is achieved primarily by a recursive characterization of the big $\J$-function. This part is where the difficulty of this paper mainly lies. We explain further our method in Section \ref{subsec12}, but leave all details to later of the paper. The outcome of it is Theorem \ref{2}, which is actually a $\widetilde{T}$-equivariant version of what we need ($\widetilde{T}$ being an enlargement of the torus $T$), but reduces to the latter as we specialize all redundant equivariant parameters to 1.  Then in the second part, we prove a reconstruction theorem (Theorem \ref{Reconstruction}), stronger than the one needed by the main theorem, which constructs an explicit family parametrizing all values of $\J^X$ starting from \emph{any} $W$-invariant point on $\L^{tw,Y}$. This part involves essentially a surjectivity argument, which is achieved by a $W$-equivariant formal Implicit Function Theorem.

\par Twisted big $\mathcal{J}$-functions of $X$ may be reconstructed using the same method, as we will see in various examples considered in this paper. Among them we list a few major results here. In the case of Euler-type twistings, consider $E$ a vector bundle over $X$. It admits a formal decomposition $E = \mu^{-1} \sum_{k=1}^r L_k$ into K-theoretic Chern roots $L_k$, where each summand $L_k$ is a monomial in $P_{is}$ and $\mu$ is the equivariant parameter for fiber-wise $\C^\times$-scaling on $E$. We have the following quantum-Lefschetz-type theorem concerning the image cone $\L^{X,(\Eu,E)}$ of the $(\Eu,E)$-twisted big $\J$-function (see Section \ref{Jfunc} for precise meaning).
\begin{theorem} \label{MainTheoremLefschetz}
Suppose
\[
J = \sum_{d\in\mathcal{D}} \prod_{i,s} Q_{is}^{d_{is}} \cdot J_d(P,\Lambda, y;q) 
\]
is any $W$-invariant point on $\L^{tw,Y}$ (like $\widetilde{J}^{tw,Y}$). Then,
\[
J^{(\Eu,E)} = \sum_{d\in\mathcal{D}} \prod_{i} Q_{i}^{\sum_s d_{is}} \cdot \prod_{k=1}^r \prod_{l=1}^{\langle c_1(L_k),d \rangle} (1- \mu L_k(P)\ q^l) \cdot J_d(P,\Lambda,1;q)
\]
locates on the image cone $\L^{X,(\Eu,E)}$, under the specialization $Q_{is} = Q_i \ (\forall i,s)$ and $y=1$. 
\end{theorem}
Here $P$ refers to the collection of all $P_{is}$ and $\Lambda$ refers to the collection of all $\Lambda_i$. For simplicity of notations, we follow similar convention for the rest of this paper as long as no confusion is caused. 

\par In fact, the entire $\L^{X,(\Eu,E)}$ is recovered by applying $\mathcal{P}^W$ to $J$, just like in Theorem \ref{0}. We omit this part of the conclusion to avoid repetition. 

\par With proper choice of $E$, we may use such twistings to express big $\J$-functions of hypersurfaces or even complete intersections in $X$ such as generalized flag varieties. Moreover, a more general form will be proved in Section \ref{twist}, which enables us to re-interpret the quasi-map vertex functions in terms of stable-map big $\J$-functions.

\par Aside from the Euler-type twistings above, we consider also K-theoretic GW invariants twisted by the level structures \cite{R-Z}. Given any vector bundle $E$ over $X$ and integer $l$, the level-$l$ structure of $E$, usually abbreviated as $(E,l)$ in later sections, modifies the virtual structure sheaf by
\[
\mathcal{O}^{\text{virt}}_{0,m,d}\longmapsto \mathcal{O}^{\text{virt}}_{0,m,d}\otimes \det^{-l}(\ft_*\ev^*E).
\]
Taking $E$ as one of the tautological bundles $V_j$ where $1\leq j\leq n$, we have
\begin{theorem} \label{MainTheoremLevel}
\[
J^{X,(V_j,l)} = (1-q)\sum_{d\in \mathcal{D}} \prod_{i=1}^n Q_i^{\sum_s d_{is}} \cdot \left[\prod_{s=1}^{v_j} P_{js}^{d_{js}}q^{\frac{d_{js}(d_{js}-1)}{2}}\right]^l \cdot J^X_d
\]
represents a point on $\L^{X,(V_j,l)}$, the image cone of the big $\J$-function with level structure $(V_j,l)$. 
\end{theorem}
Here for simplicity of notation, we denote by $J^X_d$ the coefficient of $\prod_i Q_i^{\sum_s d_{is}}$ in the summand of $J^X$ associated to $d\in\mathcal{D}$, as appearing in Corollary \ref{MainCor}. The above theorem generalizes the results in \cite{R-Z} and in \cite{Wen} where $E$ is required to be generated by actual line bundles on $X$. Moreover, it turns out that the level structures are not entirely independent on flag varieties, and we will prove the level correspondence phenomenon (Theorem \ref{levelcorrespondence}), generalizing the observation in \cite{Dong-Wen} for grassmannians.

\par Finally, we construct K-theoretic mirrors of the small $J$-functions of flag varieties as Jackson-type integrals ($q$-integrals) following \cite{Givental-Yan}. Define index sets
\begin{align*}
\mathfrak{I}_X & := \{(i,s,r)| 1\leq i\leq n, 1\leq s\leq v_i, 1\leq r\leq v_{i+1} \},\\
\mathfrak{I}_Y & := \{(i,s,s')| 1\leq i\leq n, 1\leq s,s'\leq v_i, s\neq s'\}.
\end{align*}
Then, the K-theoretic mirrors are given in the theorem below.
\begin{theorem} \label{thm2}
The $q$-integral
\begin{align*}
\mathcal{I}^{T} = \int_{\Gamma\subset\mathcal{X}_{\{Q_{i}\}}} & \prod_{s=1}^{v_n}\prod_{r=1}^{N} X_{nsr}^{\ln\Lambda_{r}/\ln q} \cdot \prod_{i,s,s'} Y_{iss'} \cdot e^{\sum_{k>0}\frac{\sum_{i,s,r} X_{isr}^k - q^k\sum_{i,s,s'} Y_{iss'}^k}{k(1-q^k)}}\\
& \cdot \frac{(\bigwedge_{i,s,r}d_q\ln X_{isr})\wedge(\bigwedge_{i,s,s'}d_q\ln Y_{iss'})}{\bigwedge_{1\leq i\leq n, 1\leq s\leq v_i} (\sum_{r=1}^{v_{i+1}}d_q\ln X_{isr} - \sum_{r=1}^{v_{i-1}}d_q\ln X_{i-1,rs} - \sum_{1\leq s'\leq v_i}^{s'\neq s}d_q\ln\frac{Y_{iss'}}{Y_{is's}})},
\end{align*}
where the products and sums with omitted index sets are taken either over $(i,s,r)\in\mathfrak{I}_X$ or $(i,s,s')\in\mathfrak{I}_Y$, over suitable choices of (linear combinations of) $q$-lattices $\Gamma$ on the torus
\[
\mathcal{X}_{\{Q_{i}\}} = \left\{(X,Y)\in\mathcal{X}|\frac{\prod_{r=1}^{v_{i+1}}X_{isr}}{\prod_{r=1}^{v_{i-1}}X_{i-1,rs}} = Q_{i}\frac{\prod_{1\leq s'\leq v_i}^{s'\neq s}Y_{iss'}}{\prod_{1\leq s'\leq v_i}^{s'\neq s}Y_{is's}},\ \ \forall i,s\right\},
\]
represents components of the $T$-equivariant small $J$-function $J^X$.
\end{theorem}
One example of the $q$-lattices $\Gamma$ considered in the theorem is the lifting of 
\[
\prod_{(i,s,r)\in \mathfrak{I}_X}^{r\neq s}\{X_{isr}=q^{-d} | d\in\mathbb{Z}\} \ \times \ \prod_{(i,s,s')\in \mathfrak{I}_Y}\{Y_{iss'}=-q^d | d\in\mathbb{Z}\}
\]
onto the torus $\mathcal{X}_{\{Q_{i}\}}$, and the integral over this specific ``cycle'' is exactly (a multiple of) the localization of $J^X$ at the $T$-fixed point $E$ representing the flag
\[
\Span\{\e_1,\cdots,\e_{v_1}\}\subset \Span\{\e_1,\cdots,\e_{v_2}\} \subset\cdots\subset \Span\{\e_1,\cdots,\e_{v_n}\} \subset \mathbb{C}^N
\]
of $X$. In fact, the theorem amounts to proving that $\mathcal{I}^{T}$ and $J^X$ satisfy similar $q$-difference equations.

\subsection{Fixed point localization} \label{subsec12}

\par Our major idea of proof in this paper, which may properly be called non-abelian localization as in \cite{Givental-Yan}, involves practicing torus fixed point localization in the context of abelian/non-abelian correspondence. 

\par Abelian/non-abelian correspondence dates back to the papers \cite{BCK1} and \cite{BCK2} in which the authors studied the cohomological $J$-function for grassmannians and flag manifolds via their associated abelian quotients which are toric varieties. In this way, they built up a quantum version of the classical theory on (equivariant) cohomology of GIT quotients by Ellingsrud-Str{\o}mme \cite{E-S} and Martin \cite{Martin}. The correspondence was then completed to the level of Frobenius structures in \cite{CKS}. Such correspondence has been generalized in various settings. In \cite{Webb}, Webb proved in cohomological settings the $I$-function correspondence using the quasimap compactification for quotients $Z//G$ with $Z$ being any l.c.i affine variety. In \cite{CLS}, Coates et al. derived a correspondence with bundles and used it to compute GW invariants of certain blow-ups. In \cite{GW1} and \cite{GW2}, in cohomological and K-theoretic settings respectively, Gonz\'alez and Woodward proved another variant of the correspondence for quotients $Z//G$ with $Z$ being smooth projective to the level of potential functions and quantum products, using the theory of gauged maps. 

\par On the other hand, applying fixed point localization on the moduli space of stable maps gives rise to a characterization of the big $\J$-functions through a system of recursive relations in the torus-equivariant settings. It holds true under the condition that the target space is equipped with a torus action with isolated fixed points and isolated 1-dim invariant orbits connecting the fixed points. More precisely, given two fixed points $a$ and $b$ and their connecting orbit $\varphi$, the recursive relations take roughly the form
\[
\Res_{q = \lambda^{1/m}}\J|_a(q)\frac{dq}{q} = \coeff_m(\varphi) \cdot \J|_b(\lambda^{1/m}),
\]
where $\lambda$ is the tangent character of $\varphi$ at $a$, $m$ is any positive integer, and $\coeff_m(\varphi)$ is a rather explicit recursion coefficient. In other words, the residues at poles of the specialization of the big $\J$-function at $a$ are expressed in terms of its specialization at other fixed points (like $b$) multiplied with recursion coefficients. 

\par The recursion coefficients allow for building up a connection between the quantum K-theory of non-abelian GIT quotients like flag varieties with that of their associated abelian quotients. The special case of grassmannians is considered in \cite{Givental-Yan}. Both the grassmannian and its associated abelian quotient (a product of projective spaces, in this case) satisfies the premises of applying the recursive characterization. Moreover, one may see that the recursion coefficients of a certain Euler-type twisted theory on the abelian quotient descend exactly to those of the standard theory on the grassmannian. Consequently, the image cone of the big $\J$-function of the grassmannian, roughly speaking, is proved to be the Weyl-group-invariant part of that of its abelian quotient.

\par One may expect the generalization from grassmannians to flag varieties to be straight-forward, but this is not true, as a vital new problem arises in our case of flag varieties: there is no longer a torus acting naturally on flag variety $X$ as well as abelian quotient $Y$ and giving rise to isolated fixed points on both targets at the same time. In fact, it will be seen that the most natural torus action on $Y$ has non-isolated fixed points which appear essentially because of the discrepancy between the stability conditions of $X$ and $Y$. The same problem arises in general GIT quotients as well. 

\par To solve this problem, one is motivated to develop an analogue of the recursive characterization that works in the non-isolated case. The idea we employ here is, to first enlarge the torus action on the abelian quotient, which reduces the problem to the isolated case, and then to study the limiting behavior as we restrict the torus back to the initial one and non-isolated fixed-points arise. In this process, a careful study of the non-isolated fixed point loci as well as the strata of non-isolated orbits connecting them is needed. The most important observation here is the necessity of taking into consideration the ``balanced broken orbits'' in between the fixed points when we compare the recursive formulae of $X$ and of $Y$.
\begin{remark}
The role played by flag varieties in the derivation is not special. The entire process may be generalized to quiver varieties or other GIT quotients of vector spaces of interest as well, without introducing extra technicalities. The central problem in such generalization reduces to finding a hands-on (e.g. combinatorial) description of the fixed points and one-dimensional invariant orbits with respect to the torus action, on both $X$ and $Y$.
\end{remark}

\subsection{Structure}

\par The organization of the paper follows approximately the same logic as in Section \ref{subsec11}. In Section \ref{prelim}, we introduce some notations that will be used to describe the geometry of $X$ and $Y$ throughout the paper, and recall in more detail the definitions related to $\mathcal{J}$-functions. We also describe in more detail the basics of the recursive characterization mentioned above. Then in Section \ref{NonAbelLoc} and \ref{NonisoRec}, we will prove the first half of the main theorem, which is enough to justify Corollary \ref{MainCor}. In Section \ref{recon}, we deduce the reconstruction theorem and finish (the second half of) the proof of Theorem \ref{0}. As further illustration of the newly developed strategy, we consider twisted quantum K-theories of the flag varieties in Section \ref{twist}, and analyze their properties. Theorem \ref{MainTheoremLefschetz} and Theorem \ref{MainTheoremLevel} are demonstrated here among others. In the end, in Section \ref{chap5}, we construct the K-theoretic mirrors of the flag varieties and prove Theorem \ref{thm2}.

\section{Preliminaries} \label{prelim}
\par Recall that we denote by $X = \text{Flag}(v_1,\cdots, v_n; N)$ the partial flag variety regarded as a GIT quotient
\[
X = R//G = \Hom(\mathbb{C}^{v_1},\mathbb{C}^{v_2})\oplus\cdots\oplus\Hom(\mathbb{C}^{v_n},\mathbb{C}^N) // GL(v_1)\times\cdots\times GL(v_n),
\]
and $Y$ its associated abelian quotient
\[
Y = R//S = \Hom(\mathbb{C}^{v_1},\mathbb{C}^{v_2})\oplus\cdots\oplus\Hom(\mathbb{C}^{v_n},\mathbb{C}^N)//(\mathbb{C}^\times)^{v_1}\times\cdots\times(\mathbb{C}^\times)^{v_n},
\]
where $S$ is the maximal torus in $G$ consisting of diagonal matrices. The stability condition of $X$ requires that each linear transformation in $\Hom(\mathbb{C}^{v_i},\mathbb{C}^{v_{i+1}})$ is injective, while the stability condition of $Y$ requires that the matrices have no zero columns. Note that due to the discrepancy in the stability conditions, $X$ is not a further quotient of $Y$. Both spaces are endowed with natural action by torus $T = (\mathbb{C}^\times)^N$, induced from the standard action of $T$ on $\mathbb{C}^N$ appearing in $\Hom(\mathbb{C}^{v_n},\mathbb{C}^N)$.

\subsection{K-rings} \label{Kring}
\par $K_T(X)$, the $T$-equivariant K-ring of $X$, is generated over the coefficient ring of equivariant parameters $\mathbb{C}[\Lambda_1^{\pm 1}, \cdots, \Lambda_n^{\pm 1}]$ by the tautological bundles $\{V_i\}_{=1}^n$ and their exterior powers. By the splitting principle, there exists a variety $\widetilde{X}$ such that $K_T(X)\rightarrow K_T(\widetilde{X})$ is injective and all $V_i$ can be decomposed into direct sum of line bundles in the latter. Assume $V_i = \sum_{s=1}^{v_i} P_{is}$, then $\Lambda^k V_i$ is simply the $k$-th elementary symmetric polynomial of $\{P_{is}\}_{s=1}^{v_i}$. Abusing the notations, we replace $V_i$ with $\sum_{s=1}^{v_i} P_{is}$ in $K_T(X)$ when convenient, although none of $P_{is}$ standing alone is a well-defined line bundle over $X$. These notations are, however, convenient as they are compatible with the tautological bundles $\{P_{is}\}$ on the abelian quotient $Y$ under the diagram
\begin{equation*}
\xymatrix{
	R^s(G)/S\ \ar@{^(->}^{\iota\quad\ }[r]\ar^q[d] & \ Y = R^s(S)/S\\
	X = R^s(G)/G &
}
\end{equation*}
where $R^s(G)$ and $R^s(S)$ stands for the stable locus of the $G$- and $S$-action respectively. In other words, we have
\begin{proposition} \label{descend}
The tautological bundles of $X$ and $Y$ are related by
\[
\iota^*\bigoplus_{k=1}^{v_i} P_{ik} = q^*V_i.
\]
\end{proposition}
In order to simplify the language that we use, we say a certain bundle over $Y$ \textbf{descends} to one over $X$ if they pull back to the same bundle over $R^s(G)/S$. Algebraically, it means that they may be represented by the same (symmetric) expression of $\{P_{is}\}$. In fact, the equivariant K-ring $K_T(Y)$ bears a \textbf{Weyl group} action by $W = S_{v_1}\times\cdots\times S_{v_n}$ which effectively permutes the subscript $s$ in $P_{is}$. Only the $W$-invariant expressions in $K_T(Y)$ descend to well-defined elements of $K_T(X)$.

\par Recall that we take the Novikov variables $\{Q_i\}_{i=1}^n$ of $X$ as corresponding to (the first Chern classes of) the determinant line bundles $\{\det V_i\}_{i=1}^n$, and $\{Q_{is}\}_{i=1,s=1}^{n\ \ \ v_i}$ of $Y$ as corresponding to $\{P_{is}\}_{i=1,s=1}^{n\ \ \ v_i}$. By Proposition \ref{descend}, the descending map between the K\"{a}hler cones of $X$ and $Y$ is given by $Q_{is}\mapsto Q_i (\forall i,s)$.

\par The Poincar\'e pairing on $K_T(X)$ is given by $\langle \mathcal{F},\mathcal{G} \rangle = \chi_T(X, \mathcal{F}\otimes\mathcal{G})$. Up to localization of the coefficient ring, we have the following residue formula of the pairing by fixed point localization.
\begin{proposition} \label{FlagPairing}
\[
\chi_T(X, \mathcal{F}\otimes\mathcal{G}) = \frac{(-1)^{\sum_{i=1}^n v_i}}{\prod_{i=1}^n v_i!}\Res_{P\neq 0,\infty} \frac{\mathcal{F}(P,\Lambda)\mathcal{G}(P,\Lambda) \cdot \prod_{i=1}^n\prod_{r\neq s}^{1\leq r,s\leq v_i}(1-\frac{P_{is}}{P_{ir}})}{\prod_{i=1}^{n-1}\prod_{1\leq r\leq v_{i+1}}^{1\leq s\leq v_i}(1-\frac{P_{is}}{P_{i+1,r}})\cdot\prod_{1\leq r\leq N}^{1\leq s\leq v_n}(1-\frac{P_{ns}}{\Lambda_{r}})}\cdot \frac{dP}{P},
\]
where $\frac{dP}{P} = \wedge_{i,s}\frac{dP_{is}}{P_{is}}$ is a scaling-invariant volume form.
\end{proposition}

\subsection{Fixed points and enlarged torus action} \label{Torus}
\par The $T$-fixed points of $X$ are easily seen to be the standard flags. They can be represented by those sequences of matrices in $R = M_{v_2\times v_1}(\mathbb{C})\times\cdots\times M_{N\times v_n}(\mathbb{C})$ where each matrix is of full rank and has columns chosen from the standard basis vectors $\{\e_1,\cdots,\e_N\}$ of $\mathbb{C}^N$. For instance, when $X = R//G = \text{Flag}(1,2;3)$, the standard flag 
\[
\Span\{\e_1\}\subset \Span\{\e_1,\e_2\} \subset \mathbb{C}^3
\]
has representative
\[
\left(\begin{bmatrix}1 \\ 0\end{bmatrix}, \begin{bmatrix} 1 & 0\\ 0 & 1\\ 0 & 0\end{bmatrix}\right) \in R = M_{2\times 1}(\mathbb{C})\times M_{3\times 2}(\mathbb{C}).
\]
\par On the other hand, $T$-fixed points of $Y$ are usually non-isolated (see Section \ref{specialcasen2}). The stability condition of $Y = R//S$ requires that points in $Y$ have representatives in $R$ being sequences of matrices without all-zero columns, and the quotient by torus $S$ means that two such sequences of matrices represent the same point of $Y$ if they differ by some transformations of the following type:
\begin{itemize}
	\item first scale the $k$-th row of the $v_i\times v_{i-1}$ matrix by $z\in\mathbb{C}^\times$,
	\item then at the same time scale the $k$-th column of the $v_{i+1}\times v_i$ matrix by $z^{-1}$.
\end{itemize}
\par For brevity, in the rest of the paper, we use always the representatives in $R$ (i.e. sequences of matrices) to refer to points in $X$ and $Y$ when no confusion is caused. 

\subsubsection*{The enlarged torus $\wT$}
\par  As preparation for later sections of the paper, we enlarge the action of $T$ on $Y$ to that of a bigger torus
\[
\wT = \prod_{i=2}^{n+1} \C^{v_i}
\]
with equivariant parameters $(\Lambda_{is})_{i=2, s=1}^{n+1\ \ v_i}$. $\widetilde{T}$ acts on $R$ by scaling the rows of the $v_{i}\times v_{i-1}$ matrix with characters $\Lambda_{i,1}, \ldots, \Lambda_{i, v_i}$, and thus on $Y$ as it commutes with that of $S$. We identify $T$ with the last component of $\widetilde{T}$, and thus $\Lambda_s$ of $T$-action with $\Lambda_{n+1,s}$ of $\widetilde{T}$-action, for $s=1,2,\cdots,N$. 

\par Such $\wT$-action does not always descend onto $X$ as the Weyl group $W=S_{v_1}\times\cdots\times S_{v_n}$ does not always commute with $\wT$. Yet, $W$ lies in the normalizer of $\wT$. Its induced action on the characters of $\wT$ is given, once again, exactly by permuting the subscript $s$ in $\Lambda_{is}$ for all $i\neq n+1$. The last component in $\wT$ (which we have identified with $T$) remains untouched, while the first component $S_{v_1}$ in $W$ commutes with $\wT$. In other words, $W/S_{v_1}$ acts effectively by permuting the $\wT/T$-characters.

\par $\widetilde{T}$-fixed points of $Y$ form an isolated subset of the $T$-fixed points. They are represented in $R$ by sequences of matrices where there is exactly one 1 in each column while all other entries vanish. Not all of such points reside in the stable locus for $X$: only those represented by sequences of full-rank matrices do. They are exactly the ones which are also isolated in the $T$-fixed loci of $Y$. When $n=2$ and $v_i = i (i=1,2)$, the following sequence in $R$ gives such a point on $Y$
\[
A = \left(\begin{bmatrix}1 \\ 0\end{bmatrix}, \begin{bmatrix} 1 & 0\\ 0 & 1\\ 0 & 0\end{bmatrix},\begin{bmatrix}1 & 0 & 0 \\ 0 & 1 & 0 \\ 0 & 0 & 1 \\ 0 & 0 & 0 \end{bmatrix}\right).
\]
Geometrically, it represents the standard flag $\Span\{\e_1\}\subset\Span\{ \e_1,\e_{2}\}\subset\C^3$ of $X = \text{Flag}(1,2;3)$. We call such $T$- and $\wT$-fixed points of $Y$ \textbf{non-degenerate}.

\subsubsection*{Non-degenerate fixed points}
\par We may easily generalize our example above to any $X=\text{Flag}(v_1,\cdots,v_n;N)$ and find a ``model'' non-degenerate fixed point $A$, represented in $R$ by the sequence whose $i$-th matrix is a $v_i\times v_i$ identity matrix followed by $v_{i+1}-v_i$ rows of zeros. All other non-degenerate fixed points of $Y$ behave exactly like $A$ up to the Weyl group symmetry, so we focus our attention to $A$ for the rest of the paper. 

\par The tangent space $T_AY$ has distinct $\widetilde{T}$-characters
\[
\frac{\lambda_{ir}}{\lambda_{is}}:= \prod_{a=i+1}^{n+1}\frac{\Lambda_{ar}}{\Lambda_{as}}, \text{   where } 1\leq i\leq n, 1\leq r\leq v_{i+1}, 1\leq s\leq v_{i}, r\neq s. 
\]
Note that for any $i$, $\frac{\lambda_{ir}}{\lambda_{is}}$ descends to the $T$-character $\frac{\Lambda_r}{\Lambda_s}$ as we take $\Lambda_{as}\rightarrow 1$ for all $a\leq n$.

\par We use $\widetilde{T}\rightarrow T$ to indicate the process of setting $\Lambda_{is} = 1 \ (\forall i\leq n)$ (and $\Lambda_{n+1,s}=\Lambda_s$).

\subsection{$\J$-functions} \label{Jfunc}
\par In this section, we recall the general theory of big $\J$-function of a smooth projective variety $M$ endowed with action by torus $T$. For simplicity of notations, we write $K(M)$ for the $T$-equivariant K-ring of $M$ when no ambiguity is caused.

\par Let $\widebar{\mathcal{M}}_{0,m}(M,d)$ be the moduli space of genus-$0$, degree-$d$ stable maps to $M$ with $m$ marked points. There are natural \textbf{evaluation maps}
\[
\ev_i: \M_{0,m}(M,d) \longrightarrow M, \quad\quad (i = 1,2,\cdots,m)
\]
sending each stable map $f: (C;p_1,\cdots,p_m)\rightarrow M$ to $f(p_i)$. There are also line bundles $L_i\ (i = 1,2,\cdots,m)$ over $\M_{0,m}(M,d)$, which are called the \textbf{universal cotangent bundles}, whose fibers over $f$ are exactly the cotangent lines to $C$ at $p_i$. Moreover, there is a natural $S_m$-action on the moduli space by permuting the marked points. Under these notations, fixing a K-theoretic class $a\in K(M)$ and an integer $k$, the genus-0 $T$-equivariant (permutation-invariant) \textbf{K-theoretic GW invariant} (or \textbf{correlator}) is defined as follows:
\[
\langle aL_1^{k}, \cdots, aL_m^{k}\rangle^{S_m}_{0,m,d} := \chi^{S_m}_T\left(\widebar{\mathcal{M}}_{0,m}(M,d); \mathcal{O}^{\virt}_{0,m,d} \otimes \bigotimes_{l=1}^m\ev_l^*(a)L_l^{k}\right),
\]
where $\mathcal{O}^{\virt}_{0,m,d}$ denotes the virtual structure sheaf on $\M_{0,m}(M,d)$ introduced by Lee \cite{Lee} and $\chi^{S_m}_T$ denotes the $S_m$-invariant part of the $T$-equivariant (virtual) holomorphic Euler characteristic for each $l$. By linear extension, we may define correlators for any inputs being Laurent polynomials of the universal cotangent bundles instead of merely the monomial $aL_l^{k}$. Note that we insert the same $a$ and $k$ to all marked points to ensure that the Euler characteristic on RHS is a well-defined $S_m$-representation.

\par The permutation group $S_m$ in the definition above may be replaced by any of its subgroup of the form $H = S_{m_1}\times\cdots\times S_{m_s}$ with $m_1+\cdots+m_s = m$. In such cases, we may loosen our requirements on the inputs $aL_l^{k}$: the insertations $a$ and $k$ need only to be the same for each group of marked points permuted by $H$. In other words, we may define correlators of the form 
\[
\langle a_1L_1^{k_1}, \cdots, a_1L_{m_1}^{k_1}, a_2L_{m_1+1}^{k_2}, \cdots, a_2L_{m_1+m_2}^{k_2},\cdots, a_{s}L_m^{k_s}\rangle_{0,m,d}^H.
\]

\par Let $\gamma_1,\cdots,\gamma_r$ be an effective basis of $H_2(M)$, then any curve class $d\in H_2(M)$ can be represented by an $r$-tuple $d=(d_1,\cdots,d_r)$ where $d = \sum_i d_i\gamma_i, d_i\in\mathbb{Z}$. Let $Q_1,\cdots,Q_r$ be the corresponding Novikov variables, then the curve class $d$ can also be marked with the monomial $Q^d := \prod_{i}Q_i^{d_i}$.

\par Let $\{\phi_\alpha\}$ be an additive basis of $K(M)$ and $\{\phi^\alpha\}$ be its dual basis under the Poincar\'e pairing. Then, as we have briefly mentioned in Introduction, the $T$-equivariant (permutation-invariant) \textbf{K-theoretic big $\mathcal{J}$-function}, or simply the $\mathcal{J}$-function, of $M$ is defined as
\[
\mathcal{J}^M_T(\t;q) = 1 - q + \t(q) + \sum_{m,d,\alpha} Q^d\phi^\alpha\langle\frac{\phi_\alpha}{1-qL_0}, \t(L_1), \cdots, \t(L_n)\rangle^{S_m}_{0,m+1,d},
\]
where $\t = \t(q) = \sum_k \t_kq^k$ is a Laurent polynomial of $q$ with coefficients $\t_k\in K(M)[[Q_1,\cdots,Q_r]]$. We often omit the subscript $T$ when no confusion is caused. 

\par As a matter of fact, the invariants here are with respect to $H = S_m = S_1\times S_m\subset S_{m+1}$ under our notations introduced above, which gives rise to the name ``permutation-invariant''. Taking smaller subgroups of $H$ will lead to other variants of quantum K-theory. The advantage of considering the permutation-invariant theory is two-fold. 
\begin{itemize}
	\item On one hand, the big $\mathcal{J}$-function of the permutation-equivariant correlators satisfies a rather explicit recursive characterization in terms of its residues at certain poles, which facilitates proofs in many cases. In particular, this leads eventually to an explicit reconstruction of its image cone. 
	\item On the other hand, by representation theory of the permutation groups, the ordinary, i.e. non-permutation-equivariant, correlators may be recovered from the permutation-equivariant ones by enlarging the coefficient ring of $K(M)$ \cite{Givental:perm1}.
\end{itemize}

\subsubsection*{Loop space formalism}
\par Denote by
\[
\mathcal{K}_+ = K(M)[[Q_1,\cdots,Q_r]][q,q^{-1}]
\]
the space of Laurent polynomials of $q$ with coefficients in $K(M)[[Q_1,\cdots,Q_r]]$, and by
\[
\mathcal{K} = K(M)[[Q_1,\cdots,Q_r]](q^{\pm 1})
\]
the space of rational functions of $q$ with such coefficients. Then, the big $\mathcal{J}$-function may be regarded as a mapping from the former to the latter. There is a subspace 
\[
\mathcal{K}_- = \{\f\in\mathcal{K}|\f(0)\neq\infty, \f(\infty) = 0\}
\] 
complementary to $\mathcal{K}_+\subset\mathcal{K}$. In fact, it is not hard to check that $\mathcal{K}_{\pm}$ gives a Lagrangian polarization of $\mathcal{K}$ under the symplectic pairing
\[
\Omega(\f,\g) = \Res_{q\neq 0,\infty}\langle \f(q^{-1}), \g(q)\rangle \frac{dq}{q},
\]
where $\langle\cdot,\cdot\rangle$ is the K-theoretic Poincar\'e pairing defined above. Under this polarization, $1-q+\t(q)$ belongs to $\mathcal{K}_+$ while the correlators appearing in $\J(\t)$ lie entirely in $\mathcal{K}_-$. Therefore, the input $\t$ can always be recovered by projecting $\mathcal{J}(\t)$ to $\mathcal{K}_+$.

\par The distinguished value of $\mathcal{J}^M$ at $\t = 0$ is called the \textbf{small $J$-function}. More explicitly,
\[
J^X = \mathcal{J}^M(0;q) = 1 - q + \sum_{d,\alpha} Q^d\phi^\alpha\langle\frac{\phi_\alpha}{1-qL_0}\rangle_{0,1,d}.
\]
Note that the $S_m$-action no longer persists here, since its effect is entirely on permuting the inputs $\t$ in the correlators which now vanish completely. In other words, the small $J$-function defined above lives also in the ordinary (i.e. non-permutation-equivariant) quantum K-theory of $M$. 

\par The image $\mathcal{L}^M$ of $\mathcal{J}^M$ resides in the loop space $\mathcal{K}$. According to the axiomatic properties of the correlators \cite{Givental:perm7}, $\mathcal{L}^M$ enjoys nice geometric properties, as described in the proposition below.
\begin{proposition} \label{cone}
$\mathcal{L}^M\subset(\mathcal{K},\Omega)$ is an overruled Lagrangian cone. Moreover, the ruling space $L$ passing through the point $p\in\L$ takes exactly the form
\[
L = (1-q)T_p\L.
\]
In other words, the tangent space $T_p\L$ is tangent to $\L$ exactly along $(1-q)T_p\L$.
\end{proposition}
For convergence reasons, we are interested only in the formal germ of $\mathcal{L}^M$ at $(1-q)$. That is to say, we require the inputs $\t$ of the big $\J$-function to be small with respect to a certain adic topology defined on the coefficient ring $\Lambda = K(M)[[Q_1,\cdots,Q_r]]$. Alternatively, we require that $\t\in\Lambda_+$ for a certain ideal $\Lambda_+\subset\Lambda$ preserved by the Adams operations $\Psi^k$ on $\Lambda$, i.e. satisfying $\Psi^k(\Lambda_+)\subset\Lambda_+$. This will be discussed in more detail later.
\subsubsection*{Euler-type twisted theories}
\par Various twisted big $\J$-functions are extensively used in this paper. The twisted big $\J$-functions are defined as generating functions of correlators defined with respect to certain modified virtual structure sheaves
\[
\mathcal{O}^{\text{virt}}_{0,m,d}\longmapsto \mathcal{O}^{\text{virt}}_{0,m,d} \ \otimes \ \mathcal{E}_{0,m,d}.
\]
One subtlety that is worth mentioning is that $\phi^\alpha$ as appearing in the definition of $\mathcal{J}^M$ would gain an extra factor $\mathcal{E}_{0,3,0}$, as it is dual to $\phi_\alpha$ by definition under the twisted Poincar\'e pairing.

\par For now, we introduce only the \textbf{Euler-type} twisting, examples of which we have seen in Introduction, and save other types until necessary. Let $E$ be any vector bundle over $M$. For simplicity of notations, we define 
\[
M_{0,m,d} = \M_{0,m}(M,d),\quad\quad E_{0,m,d} := \ft_*\ev^* E,
\] 
where $\ft$ and $\ev$ refer respectively to the forgetting and evaluation map of the last marked point from the universal curve $\M_{0,m+1}(M,d)$ as below.
\begin{equation*}
\xymatrix{
	\overline{\mathcal{M}}_{0,m+1}(M,d)\ \ar^{\quad\quad \ev}[r]\ar^{\ft}[d] & \ M \\
	\overline{\mathcal{M}}_{0,m}(M,d)\ & 
}
\end{equation*}
Then, the \textbf{$(\Eu,E)$-twisted big $\J$-function} of $M$ is defined by taking $\mathcal{E}_{0,m,d} = E_{0,m,d}$ and is denoted by $\J^{M,(\Eu,E)}$. The image $\L^{M,(\Eu,E)}$ of $\J^{M,(\Eu,E)}$ is still an overruled cone.

\par The special case in Introduction where $M$ is taken as the abelian quotient $Y$ and $E$ is taken as the bundle
\[
y^{-1}\mathfrak{g}/\mathfrak{s} = y^{-1}\bigoplus_{i}\bigoplus_{r\neq s}\frac{P_{ir}}{P_{is}}
\]
plays an important role in our strategy of abelian/non-abelian correspondence. We spend a few words to explain the notation employed here. $\mathfrak{g}$ denotes the Lie algebra of the group $G$ appearing in the GIT construction of $X = \text{Flag}(v_1,\cdots, v_n; N)$, and $\mathfrak{s}$ the Lie algebra of the maximal torus $S$. $\mathfrak{g}/\mathfrak{s}$ is then a natural representation of $S$ under the adjoint action, and thus gives rise to a vector bundle over $Y$, which we denote still by $\mathfrak{g}/\mathfrak{s}$, as is shown below.
\begin{equation*}
\xymatrix{
	\mathfrak{g}/\mathfrak{s} = & (\mathfrak{g}/\mathfrak{s}\times R)//S \ar[d] \\
	Y = & R//S
}
\end{equation*}
The direct sum is over all $i$ and $r\neq s$, which means more precisely over
\[
\{(i,r,s)| 1\leq i\leq n, 1\leq r,s\leq v_i, r\neq s\}.
\]
This specific index set will appear repeatedly later, and we use the above notation as convention. 

\par Recall $y$ is the auxiliary equivariant parameter of fiber-wise $\C^\times$-action as explained earlier. More precisely, let $\mathbb{C}_{y}$ be the trivial bundle over $Y$ endowed with the standard $\mathbb{C}^\times$-action on each fiber. Then $y^{-1}\mathfrak{g}/\mathfrak{s} = (\mathbb{C}_{y})^\vee \otimes \mathfrak{g}/\mathfrak{s}$. Eventually we will take $y=1$.

\par For brevity, we denote by $\mathcal{J}^{tw,Y} = \J^{Y,(\Eu, y^{-1}\mathfrak{g}/\mathfrak{s})}$ the big $\mathcal{J}$-function of $Y$ twisted \textit{in this specific way}, and by $\mathcal{L}^{tw,Y}$ its image cone. As we will see later, $\L^X$ is recovered essentially from the Weyl-group-invariant part of $\L^{tw,Y}$, by specializing $y$ and Novikov variables.

\subsection{Fixed point localization and recursive characterization} \label{subsecRecursive}
\par The idea of using recursive relations to characterize big $\mathcal{J}$-functions goes back to the paper \cite{Jeff_Brown} for quantum cohomology. Its K-theoretic version, developed in \cite{Givental:perm2}, provides us with a direct way of checking whether a $q$-rational function in $\mathcal{K}$ lies on the image cone of a certain big $\J$-function based on fixed point localization. The criterion works for target spaces equipped with a torus action with isolated fixed points and connecting orbits. 

\par Let $M$ be such a target space under the action of $T$ and $Q_1,\cdots,Q_r$ be the Novikov variables as above. Denote by $\mathcal{F}$ its fixed point set. The fixed point classes $\{\phi^a\}_{a\in\mathcal{F}}$ form an additive basis of $K(M)$, so any $q$-rational function $\f$ with coefficients in $K(M)[[Q_1,\cdots,Q_r]]$ may be expanded as 
\[
\f = \sum_{a\in\mathcal{F}}\f_a\phi^a.
\] 
The criterion applies to these specialization coefficients $\f_a$, which are the images of $\f$ under restriction $K(M)\rightarrow K(a)$. For example, in our case of flag variety $X$ (resp. abelian quotient $Y$), $\f_a$ is obtained from $\f$ by replacing all tautological bundles $P_i$ (resp. $P_{is}$) by $P_{i}|_a$ (resp. $P_{is}|_a$).
\begin{criterion} \label{criterion}
$\f\in\mathcal{K}$ represents a value of $\J^M$ if and only if $\forall a\in\mathcal{F}$, the two conditions below hold
\begin{itemize}
	\item[(i)] $\f_a$, when expanded at its poles at roots of unity, lies in $\mathcal{L}^{pt}$, the image cone of the (permutation-invariant) big $\J$-function of a point, but with coefficient ring $K_T(a)[[Q_1,\cdots,Q_r]]$.
	\item[(ii)] Outside $0,\infty$ and roots of unity, $\f_a$ has poles only at values of the form $\lambda^{1/m}$, for $\lambda$ a $T$-character of the tangent space $T_aM$ and $m$ a positive integer. Moreover, the residues satisfy
	\[
	\Res_{q = \lambda^{1/m}}\f_a(q)\frac{dq}{q} = - \frac{Q^{mD}}{m}\frac{\Eu(T_aM)}{\Eu(T_\phi M_{0,2,mD})} \f_b(\lambda^{1/m}).
	\]
    Here $b$ is (only) other fixed point on $M$ in the closure of the (only) 1-dim $T$-orbit with tangent character $\lambda$ at $a$, $D\in H_2(M)$ is the degree of this orbit $ab\simeq\mathbb{C}P^1$, and $\phi$ is $T$-fixed element in $M_{0,2,mD}$ given by the $m$-sheet covering map $\mathbb{C}P^1\rightarrow ab \hookrightarrow M$ ramified at the two marked points which are sent exactly to $a$ and $b$ respectively.
\end{itemize}
\end{criterion}
Intuitively, Condition (i) controls the behavior of $\f_a$ near the poles at roots of unity, which are the only legitimate poles arising in the fixed-point localization computations of correlators of point-target space, and it usually not hard to check; Condition (ii) controls the behavior near all other poles. Note that Condition (i) does not involve any additional choices depending on $M$, so the image cone of the big $\J$-function is, in some sense, determined completely by the recursion coefficients in Condition (ii), which in turn reply only on the $T$-equivariant geometry of $M$, i.e. the torus fixed points on $M$ and their connecting 1-dimensional orbits.

\par The Criterion above is a major tool of the paper.

\par It is called \textbf{recursive} characterization because $Q^{mD}$ appears in the coefficient on RHS of Condition (ii). It means that we compare always (the residues of) the $Q^d$-term in $\f|_a$ with the $Q^{d-mD}$-term in $\f|_b$. In other words, the recursion is on the homological degrees.
\begin{remark} \label{twistedrecursion}
If we are to characterize twisted theories, the recursion coefficients in Condition (ii) will need to ``twist'' accordingly. For example, if the twisting is given by $(\Eu,E)$, i.e. of the Euler-type that we introduced before, we will need $\f_a$ to satisfy
\[
\Res_{q = \lambda^{1/m}}\f_a(q)\frac{dq}{q} = -\frac{Q^{mD}}{m}\frac{\Eu(T_aM)}{\Eu(E)|_a}\frac{\Eu(E_{0,2,mD})|_\phi}{\Eu(T_\phi M_{0,2,mD})}\f_b(\lambda^{1/m}).
\]
\end{remark}

\subsection{Invariance under pseudo-finite-difference (PFD) operators} \label{subsecPFD}

\par We assume further that $H^2(M;\mathbb{Z})$ admits a basis $\{c_1(P_i)\}_{i=1}^r$ for line bundles $P_i$ over $M$ (this is true for the flag varieties and the abelian quotients), and that $\{Q_i\}_{i=1}^r$ are dual exactly to this basis. Then, $\L^M$ has the following $\mathcal{D}_q$-module property \cite{Givental:perm8}, mentioned already in the Introduction. It will be used repeatedly in various cases in this paper.
\begin{lemma} \label{pfdpreserve}
The image cone $\L^M$ of $\J^M$ is preserved by the group generated by operators of the form
\[
O_1 = e^{D(Pq^{Q\partial_Q}, Q, q)}
\]
and the form
\[
O_2 = e^{\sum_{k>0}\frac{\Psi^k(D(Pq^{kQ\partial_Q}, Q, q))}{k(1-q^k)}}.
\]
Here $D$ is any Laurent polynomial (with coefficients in $\mathbb{C}$), and 
\[
Pq^{Q\partial_Q} = (P_1q^{Q_1\partial_{Q_1}}, \cdots, P_rq^{Q_r\partial_{Q_r}}),\ Q = (Q_1,\cdots,Q_r).
\]
\end{lemma}
By definition,
\[
P_iq^{Q_i\partial_{Q_i}}\cdot Q^d = P_iq^{d_i}\cdot Q^d;
\]
$\Psi^k$, the Adams operation, acts by $\Psi^k(Q) = Q^{|k|}$, $\Psi^k(q) = q^k$ and $\Psi^k(P_iq^{kQ_i\partial_{Q_i}}) = P_i^kq^{kQ_i\partial_{Q_i}}$. We denote by $\mathcal{P}$ the group generated by operators of the form $O_1$ and $O_2$, and call elements of $\mathcal{P}$ the \textbf{pseudo-finite-difference (PFD)} operators.

\par The lemma is originally proved using the so-called adelic characterization which relates K-theoretic big $\J$-functions to cohomological ones. Due to the divisor axiom of (cohomological) GW invariants, the ruling spaces of the image cone of the cohomological big $\mathcal{J}$-function are modules over the algebra of differential operators of the form $D(zQ\partial_Q - p, Q, z)$, with $z = \log q$, $p$ divisor classes and $Q$ Novikov variables. Such operators become the PFD operators as above under the adelic characterization, so Lemma \ref{pfdpreserve} may in fact be regarded as the (genus-zero) reincarnation of the divisor axiom under K-theoretic settings. For more details one is referred to \cite{Givental-Tonita} for the ordinary quantum K-theory and to \cite{Givental:perm8}\cite{Givental:perm10} for the permutation-equivariant quantum K-theory. 

\par An alternative proof may be given using the recursive characterization in Section \ref{subsecRecursive}. It suffices to show that for any $F\in\mathcal{P}$, $F\cdot\f$ will satisfy Criterion \ref{criterion} as long as $\f$ does. We provide only the details of verification of Condition (ii) here, in order to avoid unnecessary digression into the structure of $\L^{pt}$ involved in Condition (i), which is not truly related to the main topic of this paper. 

\par Indeed, it suffices to consider the case of $F = P_i^k q^{kQ_i\partial_{Q_i}}$ for some $1\leq i\leq r$ and $k\in\mathbb{Z}$, as these are the only ingredients of $O_1$ and $O_2$ that may possibly have different effect on LHS and RHS of Condition (ii). Suppose $\f = \sum_{d}Q^d \cdot\f_d$ satisfies already the recursion. We have
\[
\Res_{q = \lambda^{1/m}}Q^{d+mD}\cdot \f_{d+mD}|_a(q)\frac{dq}{q} = - \frac{Q^{mD}}{m}\frac{\Eu(T_aM)}{\Eu(T_\phi M_{0,2,mD})} Q^d\cdot \f_d|_b(\lambda^{1/m}),
\]
for any $d = (d_1,\cdots,d_r)$. Now $F\cdot\f = \sum_{d}Q^d\cdot P_iq^{d_i} \f_d$. Since $P_iq^{d_i}$ has no pole at $q = \lambda^{1/m}$, compared to $\f$, $F\cdot\f$ obtains an extra factor of $P_i|_{a} (\lambda^{1/m})^{d_i+mD_i}$ on LHS of the recursion, while an extra factor of $P_i|_{b} (\lambda^{1/m})^{d_i}$ on RHS. Recall that the 1-dim $T$-orbit $ab$ has homological degree $D = (D_1,\cdots,D_r)$ under the basis $\{c_1(P_i)\}_{i=1}^r$, which means $P_i$ restricts to $\mathcal{O}(-D_i)$ on $ab$. Since $\lambda$ is the tangent character of $ab$ at $a$, we have $P_i|_a = P_i|_b\cdot \lambda^{-D_i}$. Therefore, the two factors on LHS and RHS are the same for $F\cdot\f$, which proves it satisfies the same recursive equation as $\f$.
\begin{remark} \label{remark2}
It is not hard to see that the invariance with respect to PFD operators holds true for twisted big $\J$-functions of $M$ as well, by the same proof.
\end{remark}

\section{Non-abelian localization} \label{NonAbelLoc}
\par We restrict our consideration back to the flag variety $X$ and its associated abelian quotient $Y$. In this section and Section \ref{NonisoRec}, we will prove Theorem \ref{2} below, which gives us exactly the first half of the central Theorem \ref{0} (see the strategy of proving Theorem \ref{0} in Section \ref{subsec11}) as we take $\wT\rightarrow T$. The second half will be carried out in Section \ref{recon}. 
\begin{theorem} \label{2}
Elements in the $\mathcal{P}^W$-orbit of the function
\[
J^{tw,Y} = (1-q)\sum_{d\in \mathcal{D}} \prod Q_{is}^{d_{is}}\frac{\prod_{i=1}^n\prod_{r\neq s}^{1\leq r,s\leq v_i}\prod_{l=1}^{d_{is}-d_{ir}}(1-y\frac{P_{is}}{P_{ir}}q^l)}{\prod_{i=1}^n\prod_{1\leq r\leq v_{i+1}}^{1\leq s\leq v_i}\prod_{l=1}^{d_{is}-d_{i+1,r}}(1-\frac{P_{is}}{P_{i+1,r}\Lambda_{i+1,r}}q^l)},
\]
all locate on the image cone $\mathcal{L}^{tw,Y}$ (as described in Section \ref{Jfunc}) of the ($\widetilde{T}$-equivariant) big $\mathcal{J}$-function of $Y$ twisted by $(\Eu, y^{-1}\mathfrak{g}/\mathfrak{s})$. Moreover, they descend to points on the image cone $\mathcal{L}^X$ of the ($T$-equivariant) big $\mathcal{J}$-function of $X$, under the specialization $\widetilde{T}\rightarrow T, Q_{is}\rightarrow Q_i, y\rightarrow 1$.
\end{theorem}
Here as convention we define $d_{n+1,r} = 0$ and $P_{n+1,r}=1$ for any $r$. Recall that the action of the Weyl group $W$ on $\mathcal{P}$ is given by permuting $s$ in not only the tautological bundles $P_{is}$ of $Y$ and their associated Novikov variables $Q_{is}$, but the $\widetilde{T}$-characters $\Lambda_{is}$ (or more precisely $\wT/T$-characters, as explained in Section \ref{Torus}) as well. Yet it suffices to consider operators in $P$ that are independent of $\Lambda_{is}$ for $i\leq n$ as eventually they reduce to $1$ as $\wT\rightarrow T$.

\subsection{Structure of proof} \label{StructureProofThm2}

\par Below we briefly summarize the structure of our proof of Theorem \ref{2} and motivate the non-isolated version of recursive characterization promised in Introduction.

\subsubsection*{Reduction to $J^{tw,Y}$}

\par To prove that the $\mathcal{P}^W$-orbit of $J^{tw,Y}$ lies entirely in $\mathcal{L}^{tw,Y}$, we need only to verify that $J^{tw,Y}$ satisfies the two conditions associated to $\mathcal{L}^{tw,Y}$ as in Criterion \ref{criterion}. Indeed, this will imply that $J^{tw,Y}$ itself lies in $\mathcal{L}^{tw,Y}$, and thus so does the rest of its $\mathcal{P}^W$-orbit by Lemma \ref{pfdpreserve} (and Remark \ref{remark2}).

\subsubsection*{Condition (i) of $\mathcal{L}^{tw,Y}$}
\par Since this condition is not really the focus of this paper, in order to save ourselves from getting lost in details, we adopt a ``shortcut method'' and start directly from the known result \cite{Givental:perm5} from quantum K-theory of toric varieties, which tells us the ($\wT$-equivariant) small $J$-function of the toric variety $Y$, though not yet twisted by $(\Eu,y^{-1}\mathfrak{g}/\mathfrak{s})$,
\[
J^Y = (1-q)\sum_{d\in \mathcal{D}} \prod_{i=1}^n\prod_{s=1}^{v_i} Q_{is}^{d_{is}} \cdot \frac{1}{\prod_{i=1}^n\prod_{1\leq r\leq v_{i+1}}^{1\leq s\leq v_i}\prod_{l=1}^{d_{is}-d_{i+1,r}}(1-\frac{P_{is}}{P_{i+1,r}\Lambda_{i+1,r}}q^l)}.
\]
Our $J^{tw,Y}$ differs from $J^Y$ only by the numerators, which can be produced by certain $q$-gamma operators. More precisely,
\[
J^{tw,Y} = \prod_{i=1}^n\prod_{r\neq s}(\Gamma^i_{r,s})^{-1}\cdot J^Y, 
\quad \text{  for }\ 
\Gamma^i_{r,s} = \exp{\left(\sum_{k>0}\frac{q^k}{k(1-q^k)}\left(y\frac{P_{is}}{P_{ir}}\right)^k\left(1-q^{kQ_{is}\partial_{Q_{is}} - kQ_{ir}\partial_{Q_{ir}}}\right)\right)}.
\]
Here $(\Gamma^i_{r,s})^{-1}$ is in fact the asymptotic expansion (near the unit circle on the $q$-plane) of the ratio
\[
\frac{\prod_{l=-\infty}^0(1-y\frac{P_{is}}{P_{ir}}\cdot q^{Q_{is}\partial_{Q_{is}} - Q_{ir}\partial_{Q_{ir}}}\cdot q^l)}{\prod_{l=-\infty}^0(1-y\frac{P_{is}}{P_{ir}}\cdot q^l)},
\]
and thus acts by
\[
(\Gamma^i_{r,s})^{-1} \cdot Q^d 
= \frac{\prod_{l=-\infty}^0(1-y\frac{P_{is}}{P_{ir}}\cdot q^{d_{is}-d_{ir}}\cdot q^l)}{\prod_{l=-\infty}^0(1-y\frac{P_{is}}{P_{ir}}\cdot q^l)} \cdot Q^d 
= \prod_{l=1}^{d_{is}-d_{ir}}(1-y\frac{P_{is}}{P_{ir}}\cdot q^l) \cdot Q^d
\]
These $q$-gamma operators are not the PFD operators that we introduced earlier, and thus does not preserve $\L^{tw,Y}$, but they are known to preserve $\L^{pt}$ \cite{Givental:perm4}, which proves $J^{tw,Y}\in \L^{pt}$ from $J^{Y}\in\L^{pt}$.

\subsubsection*{Condition (ii) of $\mathcal{L}^{tw,Y}$}
\par For simplicity, we still start with $J^Y$ and assume that it satisfies the recursive relations demanded by the untwisted $\L^Y$, which may be derived from either toric theory \cite{Givental:perm5} or brute-force computation. It can be proven then, that the effect on the recursive relations of those $q$-gamma operators producing $J^{tw,Y}$ fills exactly in the gap between $\L^Y$ and $\L^{tw,Y}$. We work out the details in Section \ref{JonJ}.

\subsubsection*{Passing from $\mathcal{L}^{tw,Y}$ to $\L^X$}
\par Having known that the $\mathcal{P}^W$-orbit of $J^{tw,Y}$ lies entirely on $\mathcal{L}^{tw,Y}$, we are now ready to show that they descend to points on (the $T$-equivariant) $\L^X$ as we specialize $\widetilde{T}\rightarrow T, Q_{is}\rightarrow Q_i, y\rightarrow 1$. 

\par We start by clarifying what we mean by ``descend''. Recall that the $\L^X$ is a Lagrangian cone in the loop space $\mathcal{K}^X$ of $q$-rational functions with coefficients in $K_T(X)[[\{Q_i\}_{i=1}^n]]$, while $J^{tw,Y}$ and the elements in its $\mathcal{P}^W$-orbit are \emph{a priori} $q$-rational functions with coefficients in $K_{\wT\times \C_{y}}(Y)[[\{Q_{is}\}_{i=1, s=1}^{n \quad v_i}]]$. Now, since $J^{tw,Y}$ and thus all elements in its $\mathcal{P}^W$ are $W$-invariant, under the specialization $\widetilde{T}\rightarrow T, Q_{is}\rightarrow Q_i, y\rightarrow 1$, they become $q$-rational functions with coefficients in $K_T(Y)^W[[\{Q_i\}_{i=1}^n]]$, where $K_T(Y)^W$ refers to the $W$-invariant part of $K_T(Y)$. Therefore, they may be regarded as points in $\mathcal{K}^X$ by identifying $K_T(Y)^W$ with $K_T(X)$ in the sense elucidated in the diagram of Section \ref{Kring}. That is to say, we regard $P_{is}$ as K-theoretic Chern roots of the tautological bundles $V_i$ over $X$.

\par Then we prove that the $\mathcal{P}^W$-orbit of $J^{tw,Y}$ descend not simply to $\mathcal{K}^X$ but actually to the cone $\L^X$. Once again, by invariance of the cone under PFD operators (Lemma \ref{pfdpreserve}) and the recursive characterization (Criterion \ref{criterion}), it boils down to verifying that $J^{tw,Y}$ satisfies the Conditions (i) and (ii) demanded by $\L^X$.

\subsubsection*{Condition (i) of $\mathcal{L}^{X}$}
\par There is not much work to do for this part. In fact, Condition (i) for $\L^X$ differs from that of $\L^{tw,Y}$ only by the coefficient ring of the involved $\L^{pt}$, which poses no problem for $W$-invariant expressions like $J^{tw,Y}$ as explained above.

\subsubsection*{Condition (ii) of $\mathcal{L}^{X}$}
\par Finally, this part is where all the complicated work comes in. We prove that for any $T$-fixed point $a$ in $X$, any of its tangent character $\lambda$, and any positive integer $m$, $J^{tw,Y}|_a$ under the specialization $\widetilde{T}\rightarrow T, Q_{is}\rightarrow Q_i, y\rightarrow 1$ has only simple poles at $\lambda^{1/m}$, and it satisfies 
\[
\Res_{q = \lambda^{1/m}}\f_a(q)\frac{dq}{q} = \coeff^X_{ab}(m) \cdot \f_b(\lambda^{1/m}) := - \frac{Q^{mD}}{m}\frac{\Eu(T_aX)}{\Eu(T_\phi X_{0,2,mD})} \cdot \f_b(\lambda^{1/m}),
\]
where $b$ is the (only) $T$-fixed point of $X$ in the closure of the (only) 1-dim $T$-orbit with tangent character $\lambda$ at $a$, $D = (D_1,\cdots,D_n)$ is the homological degree of the orbit closure $ab\simeq \C P^1$, and $\phi$ is the $m$-sheet ramified covering to $ab$ with two markings.
\begin{remark}
One may realize that under the specialization $\widetilde{T}\rightarrow T, Q_{is}\rightarrow Q_i, y\rightarrow 1$, $J^{tw,Y}$ becomes exactly the expression $J^X$ in Corollary \ref{MainCor}. However, it is hard to verify directly that it satisfies the above Condition (ii) of $\L^X$, because individual summands of $J^X|_a$ by themselves, as we will see later, may contain higher poles at $\lambda^{1/m}$ for certain tangent characters $\lambda$, and may even contain poles at $(\lambda')^{1/m}$ for characters $\lambda'$ non-existent in $T_aX$. As a by-product of our derivation in Section \ref{NonisoRec}, we will see that these seemingly redundant poles come from distinct simple poles of $J^{tw,Y}$, and eventually cancel out when summed together under the specialization $\widetilde{T}\rightarrow T, Q_{is}\rightarrow Q_i, y\rightarrow 1$. Only ``legitimate'' simple poles at $\lambda^{1/m}$ persist.
\end{remark}

\par Instead of direct verification, since $J^X$ is reduced from $J^{tw,Y}$, an alternative way is to prove that the recursive formulae of $\L^X$ are also reduced from those of $\L^{tw,Y}$. Given the $T$-fixed point $a\in X$, we fix a lift of it in $Y$ which is a non-degenerate fixed point (see Section \ref{Torus}). This lift, which we still denote by $a$, is highly non-unique, but the non-uniqueness does not affect the entire derivation that follows.

\par Among all the possible different lifts of $b\in X$ adjacent to $a$ (connected to $a$ through a 1-dim orbit $ab\simeq \C P^1$), there is only one adjacent to the lift of $a$ above, and we call this specific lift $b$ as well. We denote by $\widetilde{\lambda}$ the tangent character at $a$ along the lifted 1-dim $T$-orbit in $Y$, then it is not hard to see $\widetilde{\lambda}\rightarrow\lambda$ as $\wT\rightarrow T$. The recursive formula associated to $\L^{tw,Y}$ along $ab$ in $Y$ characterizes the residues of at poles $\widetilde{\lambda}^{1/m}$
\[
\Res_{q = \widetilde{\lambda}^{1/m}}\f_a(q)\frac{dq}{q} = \coeff^{tw,Y}_{ab}(m)\cdot\f_b(\widetilde{\lambda}^{1/m}) := -\frac{Q^{mD}}{m}\frac{\Eu(T_aY)}{\Eu(y^{-1}\mathfrak{g}/\mathfrak{s})|_a}\frac{\Eu((y^{-1}\mathfrak{g}/\mathfrak{s})_{0,2,mD})|_\phi}{\Eu(T_\phi Y_{0,2,mD})} \f_b(\widetilde{\lambda}^{1/m}).
\]
Here $D = (D_{is})_{i=1,s=1}^{n \quad v_i}$ is the homological degree of $ab$ in $Y$ and thus has a slightly different meaning from the $D$ appearing in the recursive formulae of $\L^X$, but we will see that they equal as we take $Q_{is}=Q_i$. Similarly, $\phi$ here means the ramified $m$-sheet covering to $ab$ in $Y$. We will see that the recursive coefficient here reduces to that of $X$ under the specialization $\widetilde{T}\rightarrow T, Q_{is}\rightarrow Q_i, y\rightarrow 1$. The details are carried out in Section \ref{IsoRec}. 

\par The above deduction seems to prove what we need for $J^X$, but in fact it is not enough! In most cases, $\widetilde{\lambda}$ is not the only $\widetilde{T}$-character of $T_aY$ that is reduces to $\lambda$ as $\wT\rightarrow T$. In other words, $J^{tw,Y}_a$ may have simple poles at many other values which would collapse to the pole at $\lambda^{1/m}$ of $J^X_a$. When we check that the residue of $J^X_a$ at $q = \lambda^{1/m}$ satisfies the above recursive relation, we should take all these contributions into consideration. It is also the reason why individual summands of $J^X_a$ may have higher poles at certain values $q = \lambda^{1/m}$. Our proof will imply that these higher poles must cancel and that $J^X_a$ has only simple poles at $\lambda^{1/m}$. 

\par To facilitate the proof that follows, without loss of generality, we focus starting from now only on the case where $a$ is taken as the ``model'' non-degenerate fixed point $A\in X = R//G$ (introduced in Section \ref{Torus}), which is the standard flag whose $i$-th space is $\Span\{\e_1,\cdots,\e_{v_i}\}$. Following our convention of notation in Section \ref{Torus}, $A$ is also used to refer to its standard lifts in $Y = R//S$ and in $R$. By the Weyl-group-symmetry, the cases of other fixed points in $X$ are equivalent to that of $A$. 

\par Recall that the tangent space $T_AX$ has distinct $T$-characters $\frac{\Lambda_r}{\Lambda_s}$, for $(i,r,s)\in \mathfrak{T}_X:= \{(i,r,s) | 1\leq i\leq n, v_i+1\leq r\leq v_{i+1}, 1\leq s\leq v_i\}$. On the other hand, the tangent space $T_AY$ has distinct $\widetilde{T}$-characters
\[
\frac{\lambda_{ir}}{\lambda_{is}} := \prod_{a=i+1}^{n+1}\frac{\Lambda_{ar}}{\Lambda_{as}}, 
\]
for $(i,r,s)\in \mathfrak{T}_Y := \{(i,r,s) | 1\leq i\leq n, 1\leq r\leq v_{i+1}, 1\leq s\leq v_{i}, r\neq s\}$, and $\frac{\lambda_{ir}}{\lambda_{is}}$ descends to $\frac{\Lambda_r}{\Lambda_s}$ as $\widetilde{T}\rightarrow T$.
\begin{itemize}
    \item When $(i,r,s)\in\mathfrak{T}_Y$ and $r<s$, $\frac{\lambda_{ir}}{\lambda_{is}}$ reduces to a $T$-character non-existent in $T_AX$, which suggests that the sum of residues computed by recursive formulae along all such directions should eventually vanish, under the specialization $\widetilde{T}\rightarrow T, Q_{is}\rightarrow Q_i, y\rightarrow 1$. In fact, such directions always connect $A$ to degenerate fixed points on $Y$.
    \item When $(i,r,s)\in\mathfrak{T}_Y$ and $r>s$, $\frac{\lambda_{ir}}{\lambda_{is}}$ descends to the same $T$-character $\frac{\Lambda_r}{\Lambda_s}$ of $T_AX$ for any $i$. Among all 1-dim $\wT$-orbits from $A$ along such directions, only one connects $A$ to a non-degenerated $\wT$-fixed point in $Y$. (As we will see, it happens when $i$ takes the smallest possible value.) By results to be proven in Section \ref{IsoRec}, the recursive relation along this orbit contributes already the correct residue of $J^X$, which suggests that the sum of contribution from all others, i.e. those connecting $A$ to degenerate fixed points in $Y$, should vanish as $\widetilde{T}\rightarrow T, Q_{is}\rightarrow Q_i, y\rightarrow 1$. 
\end{itemize}
Recall that degenerate $\wT$-fixed points in $Y$ are those whose representatives in $R$ are unstable with respect to the $G$-action. They are exactly the ones which are non-isolated as $T$-fixed points.

\par Our proof is not complete until these two vanishing statements above are verified. In order to solve this, we introduce a ``moduli space'' $\overline{\mathcal{M}}$ of 1-dim $T$-orbits from $A\in Y$ to non-isolated $T$-fixed points in $Y$ with fixed tangent $T$-character $\frac{\Lambda_r}{\Lambda_s}$ at $A$, compactified by certain broken $T$-orbits which we call balanced. The 1-dim $\wT$-orbits from $A$ to the degenerate $\wT$-fixed points involved are exactly $\wT$-fixed points in $\overline{\mathcal{M}}$, and their contributions to the residue at $q = (\frac{\Lambda_r}{\Lambda_s})^{1/m}$ can be regarded as the $\wT$-fixed point localization terms of a holomorphic Euler characteristic of the form 
\[
\chi (\overline{\mathcal{M}}; (1-y) \mathcal{F}),
\]
where $\mathcal{F}$ is a vector bundle over $\overline{\mathcal{M}}$. The sum of all such contributions thus vanishes as $y\rightarrow 1$.  We carry out the details in Section \ref{NonisoRec}. 
\begin{remark}
At this point, $J^{tw,Y}$ is no longer the central object of consideration: we focus instead on how the recursive coefficients of $\L^{tw,Y}$ behave. The Euler characteristic over $\overline{\mathcal{M}}$ as above, which may morally be regarded as an recursive formulae computing the residue of $J^X|_A$ at $q = (\frac{\Lambda_r}{\Lambda_s})^{1/m}$ through 1-dim $T$-orbits in $Y$ from $A$ to non-trivial connected components of $T$-fixed points, generalizes the recursive formulae as in Criterion \ref{criterion} which works only when $T$-fixed points are isolated, and motivates the title \textbf{non-isolated recursion} of Section \ref{NonisoRec}.
\end{remark}

\subsection{$J^{tw,Y}$ is a value of twisted $\mathcal{J}$-function of $Y$} \label{JonJ}

\par Given $d = (d_{is})$, we denote by $J^Y_{d}$ and $J^{tw,Y}_{d}$ the coefficients of $Q^d = \prod_{i,s} Q_{is}^{d_{is}}$ in $J^Y$ and $J^{tw,Y}$ respectively. It follows from quantum K-theory of toric varieties that given $\wT$-fixed point $a\in Y$ and tangent $\wT$-character $\widetilde{\lambda}$ in $T_aY$, $J^Y$ satisfies 
\[
\Res_{q = \widetilde{\lambda}^{1/m}} Q^{d+mD} \cdot J^Y_{d+mD}|_a \cdot \frac{dq}{q} = - \frac{Q^{mD}}{m}\frac{\Eu(T_aY)}{\Eu(T_\phi Y_{0,2,mD})} \cdot Q^{d} \cdot J^Y_d|_b(\widetilde{\lambda}^{1/m})
\]
for any positive integer $m$ and degree $d$, where $b$, $D$, $\phi$ are all as before. Our goal is to prove that
\[
\Res_{q = \widetilde{\lambda}^{1/m}} Q^{d+mD} \cdot J^{tw,Y}_{d+mD}|_a \cdot \frac{dq}{q} = - \frac{Q^{mD}}{m} \frac{\Eu(T_aY)}{\Eu(y^{-1}\mathfrak{g}/\mathfrak{s})|_a} \frac{\Eu((y^{-1}\mathfrak{g}/\mathfrak{s})_{0,2,mD})|_\phi}{\Eu(T_\phi Y_{0,2,mD})} \cdot Q^{d} \cdot J^Y_d|_b(\widetilde{\lambda}^{1/m}).
\]
Recall that 
\[
J^{tw,Y} = \prod_{(i,r,s)\in\mathfrak{T}_{\mathfrak{g}/\mathfrak{s}}}(\Gamma^i_{r,s})^{-1}\cdot J^Y,\quad \text{ and } \quad \mathfrak{g}/\mathfrak{s} = \bigoplus_{(i,r,s)\in \mathfrak{T}_{\mathfrak{g}/\mathfrak{s}}} \frac{P_{ir}}{P_{is}}
\]
with $\mathfrak{T}_{\mathfrak{g}/\mathfrak{s}} = \{(i,r,s)|1\leq i\leq n, 1\leq r,s\leq v_i, r\neq s\}$. Hence, it suffices to prove for any $(i,r,s)\in \mathfrak{T}_{\mathfrak{g}/\mathfrak{s}}$ that
\[
\left.\frac{(\Gamma^i_{r,s})^{-1}|_aQ^{d+mD}}{(\Gamma^i_{r,s})^{-1}|_bQ^d}\right|_{q=\widetilde{\lambda}^{1/m}} = \frac{\Eu((y^{-1}\frac{P_{ir}}{P_{is}})_{0,2,mD})|_\phi}{\Eu(y^{-1}\frac{P_{ir}}{P_{is}})|_a} \cdot Q^{mD}. 
\]
By definition, on the domain curve of $\phi$, $T$ acts with character $\widetilde{\lambda}^{1/m}$ on the tangent space of $a$, and $\phi^*P_{ir} = \mathcal{O}(-mD_{ir})$, so $P_{ir}|_b = \widetilde{\lambda}^{mD_{ir}/m} \cdot P_{ir}|_a$, and same for $P_{is}$. Then by direct computation,
\begin{align*}
\left.\frac{(\Gamma^i_{r,s})^{-1}|_aQ^{d+mD}}{(\Gamma^i_{r,s})^{-1}|_bQ^d}\right|_{q=\widetilde{\lambda}^{1/m}} & = \left.\frac{\prod_{l=1}^{d_{is}-d_{ir}+m(D_{is}-D_{ir})}(1 - y\frac{P_{is}|_a}{P_{ir}|_a}\cdot q^l)}{\prod_{l=1}^{d_{is}-d_{ir}} (1 -y\frac{P_{is}|_b}{P_{ir}|_b}\cdot q^l)}\right|_{q=\widetilde{\lambda}^{1/m}} \cdot Q^{mD} \\
& = \prod_{l=1}^{m(D_{is}-D_{ir})}\left(1-y\frac{P_{is}|_a}{P_{ir}|_a}\cdot \widetilde{\lambda}^{l/m}\right)\cdot Q^{mD} \\
& = \frac{\prod_{l=0}^{m(D_{is}-D_{ir})}(1-y\frac{P_{is}|_a}{P_{ir}|_a}\cdot \widetilde{\lambda}^{l/m})}{(1-y\frac{P_{is}|_a}{P_{ir}|_a})}\cdot Q^{mD} = \frac{\Eu((y^{-1}\frac{P_{ir}}{P_{is}})_{0,2,mD})|_\phi}{\Eu(y^{-1}\frac{P_{ir}}{P_{is}})|_a} \cdot Q^{mD}. 
\end{align*}

\subsection{Isolated recursive coefficient of $\L^{tw,Y}$ reduces to that of $\L^X$} \label{IsoRec}

\par By isolated recursive coefficient, we mean the $\wT$-equivariant recursive coefficient along orbits $ab$, where $a$ and $b$ are both non-degenerate $\wT$-fixed points in $Y$ and thus isolated even as $T$-fixed points. As is described in Section \ref{StructureProofThm2}, $a$ and $b$ are lifts of adjacent $T$-fixed points of $X$. We prove that 
\[
\coeff^{tw,Y}_{ab}(m) = \coeff^X_{ab}(m)
\]
as $\Lambda_{is}=1(\forall i\leq n), Q_{is} = Q_i, y=1$. The meaning of both sides are given in Section \ref{StructureProofThm2}.

\par We first identify the degree terms of both sides. To avoid confusion on notation, let $\widetilde{D} = (D_{is})$ be the homological degree of (the lift of) $ab$ in $Y$, and $D = (D_{i})$ be that of $ab$ in $X$, we prove that 
\[
\prod_{i=1}^n\prod_{s=1}^{v_i} Q_{is}^{d_{is}} = \prod_{i=1}^n Q_i^{d_i}
\]
as $Q_{is} = Q_i$. In fact, over $Y$, since the Novikov variables are defined corresponding to the basis $\{c_1(P_{is})\}$ of $H_2(Y)$, in $\wT$-equivariant theory, we have $P_{is}|_b = \widetilde{\lambda}^{D_{is}} P_{is}|_a \in R(\widetilde{T})$ over $Y$. Restricting back to $T$, we have $P_{is}|_b = \lambda^{D_{is}} P_{is}|_a \in R(T)$. Similarly, over $X$, we have $\det V_i|_b = \lambda^{D_i}\det V_i|_a \in R(T)$. In the sense of Proposition \ref{descend}, $\det V_i$ is identified with $\prod_{s=1}^{v_i} P_{is}$, so we have $D_i = \sum_{s=1}^{v_i} D_{is}$, which implies what we need.

\par Now it remains to show that
\[
\frac{\Eu(T_aY)}{\Eu(y^{-1}\mathfrak{g}/\mathfrak{s})|_a} \cdot \frac{\Eu((y^{-1}\mathfrak{g}/\mathfrak{s})_{0,2,m\widetilde{D}})|_{\widetilde{\phi}}}{\Eu(T_{\widetilde{\phi}} Y_{0,2,m\widetilde{D}})}
\rightarrow 
\Eu(T_aX) \cdot \frac{1}{\Eu(T_{\phi} X_{0,2,mD})}
\]
as $\widetilde{T}\rightarrow T$ and $y\rightarrow 1$. Here we use $\widetilde{\phi}$ to refer to the ramified covering to $ab$ in $Y$. Notice that 
\[
T_{\widetilde{\phi}} Y_{0,2,m\widetilde{D}} = (TY)_{0,2,m\widetilde{D}} - 1 = \sum(-1)^i H^i(\mathbb{C}P^1;\widetilde{\phi}^*TY) - 1
\]
since $\widetilde{\phi}$ has two marked points. The $(-1)$ term comes from the one-dimensional automorphism of $\widetilde{\phi}$, which is given by the ``rotation'' on its domain curve fixing the two ramification points and upon which $\wT$ has trivial action. Similarly, we have $T_\phi X_{0,2,mD} = (TX)_{0,2,mD} - 1$. Therefore, what we need boils down to proving that, under the natural identification of $ab$ in both $X$ and $Y$ with $\C P^1$,
\[
(TY - \mathfrak{g}/\mathfrak{s})|_{ab}\in K_{\wT}(\C P^1) \rightarrow (TX)|_{ab} \in K_T(\C P^1)
\]
(and thus $T_aY - \mathfrak{g}/\mathfrak{s}|_a \rightarrow T_aX$) as $\wT\rightarrow T$. This follows directly from the construction $X = R//G$ and $Y = R//S$. More explicitly, we have
\[
TX = \bigoplus_{i=1}^n\Hom(V_i,V_{i+1}-V_i) = \bigoplus_{i=1}^{n-1}\bigoplus_{s=1}^{v_i}\left(\bigoplus_{r=1}^{v_{i+1}}\frac{P_{i+1,r}}{P_{is}}\ominus \bigoplus_{r=1}^{v_i}\frac{P_{ir}}{P_{is}}\right)\oplus \bigoplus_{s=1}^{v_n}\left(\bigoplus_{r=1}^{N}\frac{\Lambda_r}{P_{ns}}\ominus \bigoplus_{r=1}^{v_n}\frac{P_{nr}}{P_{ns}}\right),
\]
\[
TY = \bigoplus_{i=1}^n\bigoplus_{s=1}^{v_i}\Hom\left(P_{is}, \bigoplus_{r=1}^{v_{i+1}}P_{i+1,r}\Lambda_{i+1,r} - P_{is}\right) = \bigoplus_{i=1}^n\bigoplus_{s=1}^{v_i}\left(\bigoplus_{r=1}^{v_{i+1}}\frac{P_{i+1,r}\Lambda_{i+1,r}}{P_{is}}\ominus 1\right),
\]
\[
\mathfrak{g}/\mathfrak{s} = \bigoplus_{i=1}^n\bigoplus_{s=1}^{v_i}\left( 
\bigoplus_{r=1,r\neq s}^{v_i}\frac{P_{ir}}{P_{is}} \right)
\]
As $\widetilde{T}\rightarrow T$, we take $\Lambda_{ir} = 1\ (\forall 1\leq i\leq n, 1\leq r\leq v_i)$ and $\Lambda_{n+1,r} = \Lambda_r\ (\forall 1\leq r\leq N)$.
\begin{remark}
In this last step, we prove essentially
\[
\iota^*(TY - \mathfrak{g}/\mathfrak{s}) = q^*TX,
\]
as $T$-equivariant vector bundles on $R^s(G)\subset R$, where $\iota$ and $q$ are as given in Proposition \ref{descend}.
\end{remark}

\section{Non-isolated recursion} \label{NonisoRec}
\par In this section, we finish the proof of Theorem \ref{2} by verifying the two vanishing statements described in Section \ref{StructureProofThm2}, in regard to the ``non-isolated recursion'' from $A$ to degenerate $\wT$-fixed points in $Y$.  We illustrate the proof of the two statements by the $T$-character $\Lambda_{v_1+1}/\Lambda_1$ of $T_AX$ and the $T$-character $\Lambda_1/\Lambda_{v_1+1}$ non-existent in $T_AX$ respectively. Over the flag variety $X$, the first tautological bundle $V_1$ restricts to $\Lambda_1+\cdots+\Lambda_{v_1}$ over $A$, so $\Lambda_{v_1+1}$ is the first character that appears in all other tautological bundles $V_i$ but not $V_1$.

\par To make our statements more precise, we introduce some further notation. For $v>w>0$, let $I_{v\times w}$ be the $v\times w$ matrix with the first $w$ rows giving the $w\times w$ identity matrix and the last $v-w$ rows being all-zero. In other words, $A\in R$ is represented by the sequence of matrices 
\[
A = (I_{v_2\times v_1}, I_{v_3\times v_2}, \cdots, I_{v_{n+1}\times v_n}).
\]
Moreover, for $v>w>0$ and for $k\leq v$, let $I_{v\times w,k}$ be the matrix that differs from $I_{v\times w}$ only by the first column: instead of $\e_1\in\C^v$ we put $\e_{k}$. For instance, $I_{v\times w,1} = I_{v\times w}$; $I_{v\times w,k}$ is of full rank if and only if $k=1$ or $k>w$. 

\par Now, for a given subset $\JJ = \{j_1,\cdots, j_l\} \subset \{1,2,\cdots,n\}$ (where $j_1<\cdots<j_l$), we define $E_{\JJ} = E_{j_1,\cdots, j_l}\in R$ as the sequence of matrices with the $j$-th matrix:
\begin{itemize}
    \item equal to $I_{v_{j+1}\times v_j}$ if $j\notin\JJ$,
    \item equal to $I_{v_{j+1}\times v_j,v_1+1}$ of $j\in\JJ$.
\end{itemize}
In particular, $E_\phi = A$. For instance, in the case of $X = \text{Flag}(1,2;3)$, i.e. when $n=2$ and $v_i=i$ $(i=1,2)$,
\[
E_1 = \left(\begin{bmatrix} 0 \\ 1 \end{bmatrix}, \begin{bmatrix} 1 & 0\\ 0 & 1 \\ 0 & 0 \end{bmatrix}\right), \quad
E_{2} = \left(\begin{bmatrix} 1 \\ 0 \end{bmatrix}, \begin{bmatrix} 0 & 0\\ 1 & 1\\ 0 & 0 \end{bmatrix}\right), \quad
E_{1,2} = \left(\begin{bmatrix} 0 \\ 1 \end{bmatrix}, \begin{bmatrix} 0 & 0\\ 1 & 1\\ 0 & 0 \end{bmatrix}\right).
\]
Note that all $E_{\JJ}$ represent $\wT$-(and thus $T$-)fixed points in $Y$. Yet, only $E_\phi$ and $E_1$ represent $T$-fixed points in $X$, as all others contain at least one matrix not of full rank and are thus unstable with respect to the $G$-action on $R$. Moreover, $E_1,\cdots, E_n$ are exactly the $\wT$-fixed points in $Y$ to which $A$ is adjacent through a 1-dim $\wT$-orbit $\C P^1$ with tangent $T$-character $\Lambda_{v_1+1}/\Lambda_1$ at $A$, among which $E_1$ is the only non-degenerate. By our deduction in Section \ref{IsoRec}, $\coeff^{tw,Y}_{AE_1}(m)$ reduces to $\coeff^{X}_{AE_1}(m)$ under the specialization $\Lambda_{is}=1(\forall i\leq n), Q_{is} = Q_i, y=1$. We give it the special name $C = E_1$. 
\begin{remark}
There is no 1-dim $T$-orbit that connects $A$ directly to, for example, $E_{12}$. There is, however, a ``broken'' orbit, consisting of two components, going first from $A$ to $E_1$ and then from $E_1$ to $E_{12}$. Such broken orbits will be of great importance to us.
\end{remark}

\par Following a similar spirit, for $v>w>0$ and for $k\leq w$, let $I'_{v\times w,k}$ be the matrix that differs from $I_{v\times w}$ only by the $k$-th column: instead of $\e_k$ we put $\e_1$. Then for instance $I'_{v\times w,1} = I_{v\times w}$, and $I'_{v\times w,k}$ is never of full rank. Moreover, for any subset $\JJ = \{j_1,\cdots, j_l\} \subset \{2,\cdots,n\}$, we define $F_{\JJ} = F_{j_1,\cdots, j_l}\in R$ as the sequence of matrices with the $j$-th matrix:
\begin{itemize}
    \item equal to $I_{v_{j+1}\times v_j}$ if $j\notin\JJ$,
    \item equal to $I'_{v_{j+1}\times v_j,v_1+1}$ of $j\in\JJ$.
\end{itemize}
Note that still all $F_{\JJ}$ represent $\wT$-(and thus $T$-)fixed points in $Y$. $F_2,\cdots, F_n$ are exactly the $\wT$-fixed points in $Y$ to which $A$ is adjacent through a 1-dim $\wT$-orbit $\C P^1$ with tangent $T$-character $\Lambda_{1}/\Lambda_{1+v_1}$ at $A$, and all of them are degenerate as we explained earlier. In the case of $X = \text{Flag}(1,2;3)$,
\[
F_\phi = A = \left(\begin{bmatrix} 1 \\ 0 \end{bmatrix}, \begin{bmatrix} 1 & 0\\ 0 & 1 \\ 0 & 0 \end{bmatrix}\right), \quad
F_{2} = \left(\begin{bmatrix} 1 \\ 0 \end{bmatrix}, \begin{bmatrix} 1 & 1\\ 0 & 0 \\ 0 & 0 \end{bmatrix}\right).
\]

\par For simplicity, we denote $\lambda_k = \frac{\lambda_{k,v_{1}+1}}{\lambda_{k,1}}$ for $1\leq k\leq n$ (for definition see Section \ref{StructureProofThm2}). Then, the tangent $\wT$-character in $T_AY$ along $AE_k$ is exactly $\lambda_k$ for $1\leq k\leq n$, while that along $AF_k$ is exactly $\lambda_k^{-1}$ for $2\leq k\leq n$. The two vanishing statements may be formulated rigorously as follows.
\begin{proposition} \label{vanishing1}
Under the specialization $\Lambda_{is}=1(\forall i\leq n), \Lambda_{n+1,s} = \Lambda_s, Q_{is} = Q_i, y=1$,
\[
\sum_{k=2}^n \frac{-1}{1-q\lambda_{k}^{-1/m}} \cdot \coeff^{tw,Y}_{AE_k}(m) \cdot J^{tw,Y}_{E_k}(\lambda_k^{1/m}) = 0.
\]
\end{proposition}
\begin{proposition} \label{vanishing2}
Under the specialization $\Lambda_{is}=1(\forall i\leq n), \Lambda_{n+1,s} = \Lambda_s, Q_{is} = Q_i, y=1$,
\[
\sum_{k=2}^n \frac{-1}{1-q\lambda_{k}^{1/m}} \cdot \coeff^{tw,Y}_{AF_k}(m) \cdot J^{tw,Y}_{F_k}(\lambda_k^{-1/m}) = 0.
\]
\end{proposition}
In fact, if a function $J$ satisfies $\Res_{q=\lambda} J(q) \frac{dq}{q} = R$, then $J = -R / (1-q\lambda^{-1})$ up to a regular function near $\lambda$. 
\par For expository reasons, we prove the two propositions below, equivalent to the two above.
\begin{proposition} \label{nonisoA}
Under the specialization $\Lambda_{is}=1(\forall i\leq n), \Lambda_{n+1,s} = \Lambda_s, Q_{is} = Q_i, y=1$,
\[
\sum_{k=1}^n \frac{-1}{1-q\lambda_{k}^{-1/m}} \coeff^{tw,Y}_{AE_k}(m) J^{tw,Y}_{E_k}(\lambda_k^{1/m}) 
= \frac{-1}{1-q(\Lambda_{v_1+1}/\Lambda_1)^{-1/m}} \coeff^X_{AE_1}(m) J^X_{E_1}((\Lambda_{v_1+1}/\Lambda_1)^{1/m})
\]
\end{proposition}
\begin{proposition} \label{nonisoC}
Under the specialization $\Lambda_{is}=1(\forall i\leq n), \Lambda_{n+1,s} = \Lambda_s, Q_{is} = Q_i, y=1$,
\[
\sum_{k=2}^n \frac{-1}{1-q\lambda_{k}^{-1/m}} \cdot \coeff^{tw,Y}_{CE_{1,k}}(m) \cdot J^{tw,Y}_{E_{1,k}}(\lambda_k^{1/m}) = 0.
\]
\end{proposition}
\begin{remark}
In the propositions above, we formally write $m$-th roots of $T$- or $\wT$-characters without specifying the choices, but algebraically there should always be $m$ different possibilities, which differ from each other by $m$-th roots of unity. In fact, as we will see from the proofs, in order that the equations above are correct, one should always take the sum over all $m$ choices of $m$-th roots whenever such an expression appears. However, the extra summation does not weaken our statements as the residues of $q$ at these different $m$-th roots of torus characters are independent. For instance, if we manage to show that
\[
\sum_{\alpha = 1}^m \frac{c_\alpha}{1-q\cdot \zeta^\alpha\lambda^{-1/m}} = 0,
\]
where $\zeta = e^{2\pi i/m}$, then we must have $c_\alpha = 0$ for all $\alpha$. We will omit this implied summation in most part of the proof to avoid unnecessary confusion, and only emphasize it in the special case of $X = \operatorname{Flag}(1,2;3)$ in Section \ref{specialcasen2} below.
\end{remark}
\par Recall that by our earlier special notation, $C=E_1$ is the only non-degenerate $\wT$-fixed point other than $A$ in our story. Proposition \ref{nonisoA} describes the residue of $J^X_A$ at $q = (\Lambda_{v_1+1}/\Lambda_1)^{1/m}$, and Proposition \ref{nonisoC} the residue of $J^X_C$ at $q = (\Lambda_{v_1+1}/\Lambda_1)^{1/m}$ (which should vanish). In fact, we will see that the proof of Proposition \ref{nonisoC} is mostly contained as part of the proof of Proposition \ref{nonisoA}.

\par Proposition \ref{nonisoA} is equivalent to Proposition \ref{vanishing1} by what has been proven in Section \ref{IsoRec}. Proposition \ref{nonisoC} is equivalent to Proposition \ref{vanishing2} by Weyl group symmetry: the story of the $T$-character $\Lambda_{v_1+1}/\Lambda_1$ (non-existent in $T_CX$) at $C$ is exactly the same as that of $\Lambda_{1}/\Lambda_{v_1+1}$ (non-existent in $T_AX$) at $A$.

\par Finally, it is worth noting that it suffices to consider $\Lambda_{v_1+1}/\Lambda_1$ and $\Lambda_{1}/\Lambda_{v_1+1}$ at $A$ as we do in this section. The cases of other $T$-characters of the form $\Lambda_r/\Lambda_s$ and $\Lambda_s/\Lambda_r$ can be proved using exactly the same method. To say the least, assume $r>s$ and $v_{K}+1\leq r\leq v_{K+1}$, the structure of the $T$- and $\wT$-fixed loci involved in $X$ will be the same as that in the ``truncated'' flag variety $X' = \text{Flag}(v_K,v_{K+1}\cdots,v_n;N)$, which will be strictly simpler.

\par Our proof of Proposition \ref{nonisoA} will be separated into two steps.
\begin{itemize}
    \item \textbf{Step I:  } We prove that for each $k\geq 2$ (i.e. when $E_k\neq E_1 = C$), the residue term from $AE_k$
    \[
    \frac{-1}{1-q\lambda_{k}^{-1/m}} \cdot \coeff^{tw,Y}_{AE_k}(m) \cdot J^{tw,Y}_{E_k}(\lambda_k^{1/m}),
    \]
    when supplemented with suitable terms coming from the so-called balanced broken orbits, gives a sum vanishing as $\widetilde{T}\rightarrow T, Q_{is}\rightarrow Q_i, y\rightarrow 1$.
    \item \textbf{Step II:  } We prove that the supplementary terms from balanced broken orbits cancel out when we add together all the contribution from $AE_1$ to $AE_n$.  
\end{itemize}

\par In Section \ref{specialcasen2}, we will first work fully out the case of $X = \text{Flag}(1,2;3)$, that is to say, $n=2$ and $v_i=i\ (i=1,2)$. This is the easiest case where non-isolated recursion appears, and the proof comprises already the aforementioned idea in its entirety. Then we proceed to the proof of the general case in Section \ref{StepI} and \ref{StepII}, where the extra complexity comes only from notation.

\subsection{Special case of $X = \text{Flag}(1,2;3)$} \label{specialcasen2}

\par We briefly recall our notation in this special case. The flag variety $X$ and its associated abelian quotient $Y$ are both GIT quotients
\[
X = R//G = \Hom(\C,\C^2)\oplus\Hom(\C^2,\C^3) // GL(1)\times GL(2),
\]
\[
Y = R//S = \Hom(\C,\C^2)\oplus\Hom(\C^2,\C^3) // \C^\times \times (\C^\times)^2.
\]
Points on $X$ and $Y$ are thus represented by elements in $R$, i.e. pairs $(M_1,M_2)$ with $M_i$ being an $(i+1)\times i$ matrix. The natural $T = (\C^\times)^3$-action on $X$ and $Y$ are given by scaling the three row vectors of $M_2$, and we denote the equivariant parameters by $\Lambda_1, \Lambda_2, \Lambda_3$. 

\par $Y$ is a $\C P^1$-fiber bundle over $\C P^2\times\C P^2$ in this specific case, and $T$-fixed points on $Y$ are not necessarily isolated. For instance, recall that we mentioned four $T$-fixed points $A,C,E_2,E_{1,2}$ earlier in this section, and among them only $A$ and $C$ are isolated. In fact, the connected component of $T$-fixed points that contain $E_2$ and $E_{1,2}$ is isomorphic to $\C P^1$, and can be represented by
\[
\Omega = \left\{ \left. \left(\begin{bmatrix} a \\ b \end{bmatrix}, \begin{bmatrix} 0 & 0\\ 1 & 1 \\ 0 & 0 \end{bmatrix}\right)\ \right| \ a,b\in \C, \text{ not both zero} \right\}.
\]
It is one entire $\C P^1$-fiber in $Y$. Moreover, except $E_{1,2}$, each point on this component has exactly one 1-dim $T$-orbit connecting to $A$, and each such orbit has tangent $T$-character $\Lambda_2/\Lambda_1$ at $A$. More precisely, for fixed $a$ and $b$ with $a\neq 0$, the orbit is represented by
\[
\left\{ \left. \left(\begin{bmatrix} 1 \\ bt \end{bmatrix}, \begin{bmatrix} 1 & 0\\ at & 1 \\ 0 & 0 \end{bmatrix}\right)\ \right| \ t\in \C^\times \ \right\} = \left\{ \left. \left(\begin{bmatrix} a \\ b \end{bmatrix}, \begin{bmatrix} (at)^{-1} & 0\\ 1 & 1 \\ 0 & 0 \end{bmatrix}\right)\ \right| \ t\in \C^\times \ \right\}
\]
In order to connect $E_{1,2}$ to $A$ along a 1-dim $T$-orbit with tangent character $\Lambda_2/\Lambda_1$, however, the orbit has to break at $C = E_1$ (definition also see earlier in Section \ref{NonisoRec}). In other words, we have to consider (the closure of) the union of two irreducible components, one connecting $A$ to $C$ and one connecting $C$ to $E_{1,2}$. More precisely, it is represented by 
\[
\left\{ \left. \left(\begin{bmatrix} 1 \\ t \end{bmatrix}, \begin{bmatrix} 1 & 0\\ 0 & 1 \\ 0 & 0 \end{bmatrix}\right)\ \right| \ t\in \C^\times \ \right\} \bigcup \left\{ \left. \left(\begin{bmatrix} 0 \\ 1 \end{bmatrix}, \begin{bmatrix} 1 & 0\\ t & 1 \\ 0 & 0 \end{bmatrix}\right)\ \right| \ t\in \C^\times \ \right\}
\]
This gives us the first example of the balanced broken orbit mentioned above. 

\par In fact, it is not hard to see the above are all possibilities of 1-dim $T$-orbits with tangent character $\Lambda_2/\Lambda_1$ at $A$. Such orbits are not isolated essentially because there is a two-dimensional subspace of $T_AY$ where $T$ acts by character $\Lambda_2/\Lambda_1$. We enlarge $T$ into $\wT = (\C^\times)^2\times (\C^\times)^3$ to achieve isolatedness. The two extra factors act by scaling the two rows of $M_1$ for each representative $(M_1,M_2)\in R$. Geometrically it gives rise to rotation on the fibers $\C P^1$ in $Y$. Under $\wT$-action, the characters of $T_AY$ are distinct. More precisely, using the notation introduced earlier,
\[
T_AY = \frac{\Lambda_{22}\Lambda_{32}}{\Lambda_{21}\Lambda_{31}} + \frac{\Lambda_{32}}{\Lambda_{31}} + \frac{\Lambda_{33}}{\Lambda_{31}} + \frac{\Lambda_{31}}{\Lambda_{32}} + \frac{\Lambda_{33}}{\Lambda_{32}} \in R(\wT) 
\ \rightarrow \ 
\frac{\Lambda_2}{\Lambda_1} + \frac{\Lambda_2}{\Lambda_1} + \frac{\Lambda_3}{\Lambda_1} + \frac{\Lambda_1}{\Lambda_2} + \frac{\Lambda_3}{\Lambda_2} \in R(T)
\]
$A,C,E_2,E_{1,2}$ are now isolated $\wT$-fixed points. The figure below illustrates briefly the configuration of relevant $T$-fixed points and 1-dim $T$-orbits.
\begin{figure}[hbtp]
    \centering
    \def\svgwidth{3in}
    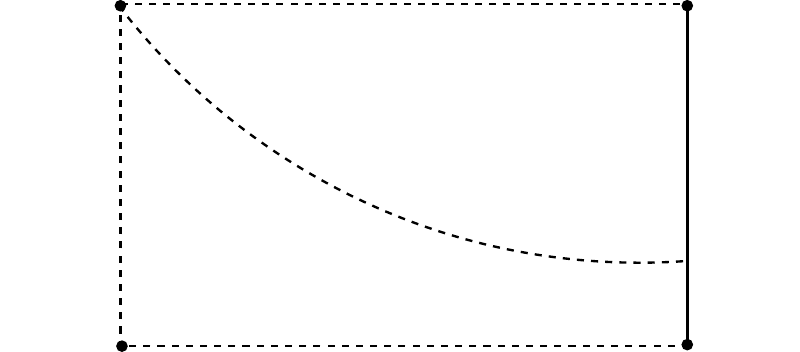
    \caption{$T$-fixed points and connecting orbits}
\end{figure}
\par We start by proving Proposition \ref{nonisoC} for $C$. That is to say, we prove for any $m$,
\[
\frac{-1}{1-q\left(\frac{\Lambda_{32}}{\Lambda_{31}}\right)^{-\frac{1}{m}}} \cdot 
\coeff^{tw,Y}_{CE_{1,2}}(m) \cdot 
J^{tw,Y}_{E_{1,2}}\left(\left(\frac{\Lambda_{32}}{\Lambda_{31}}\right)^{\frac{1}{m}}\right) = 0
\]
under specialization $\wT\rightarrow T, Q_{is}\rightarrow Q_i, y\rightarrow 1$. Here $\Lambda_{32}/\Lambda_{31}$ is the tangent $\wT$-character along $CE_{1,2}$, and is the only one descending to $\Lambda_2/\Lambda_1$ as $\wT\rightarrow T$ in $T_CY$
\[
T_CY = \frac{\Lambda_{21}\Lambda_{31}}{\Lambda_{22}\Lambda_{32}} + \frac{\Lambda_{32}}{\Lambda_{31}} + \frac{\Lambda_{33}}{\Lambda_{31}} + \frac{\Lambda_{31}}{\Lambda_{32}} + \frac{\Lambda_{33}}{\Lambda_{32}} \in R(\wT) 
\ \rightarrow \ 
\frac{\Lambda_1}{\Lambda_2} + \frac{\Lambda_2}{\Lambda_1} + \frac{\Lambda_3}{\Lambda_1} + \frac{\Lambda_1}{\Lambda_2} + \frac{\Lambda_3}{\Lambda_2} \in R(T)
\]
The proposition indicates that $J^X_C$ has no poles at $q = (\Lambda_2/\Lambda_1)^{1/m}$, which is expected as $\Lambda_2/\Lambda_1$ does not appear as a character in $T_CX = \Lambda_1/\Lambda_2 + \Lambda_3/\Lambda_1 + \Lambda_3/\Lambda_2$.

\par It amounts to proving $\coeff^{tw,Y}_{CE_{1,2}}(m) = 0$. Indeed, we have 
\[
T_{E_{1,2}}Y = \frac{\Lambda_{21}}{\Lambda_{22}} + \frac{\Lambda_{31}}{\Lambda_{32}} + \frac{\Lambda_{33}}{\Lambda_{32}} + \frac{\Lambda_{31}}{\Lambda_{32}} + \frac{\Lambda_{33}}{\Lambda_{32}} \in R(\wT)
\ \rightarrow \ 
1 + \frac{\Lambda_1}{\Lambda_2} + \frac{\Lambda_3}{\Lambda_2} + \frac{\Lambda_1}{\Lambda_2} + \frac{\Lambda_3}{\Lambda_2} \in R(T)
\]
so $\Lambda_2/\Lambda_1$ does not appear as a $T$-character in $T_{E_{1,2}}Y$. It implies by recursive characterization of $J^{tw,Y}$ that $J^{tw,Y}_{E_{1,2}}$ does not have any poles that would descend to a pole at $q = (\Lambda_2/\Lambda_1)^{1/m}$, which in turn implies that $J^{tw,Y}_{E_{1,2}}((\Lambda_{32}/\Lambda_{31})^{1/m})$ is a well-defined expression as $\wT\rightarrow T$.

\par That $\coeff^{tw,Y}_{CE_{1,2}}(m) = 0$ may be verified through direct computation. It is an isolated recursive coefficient, and we have done similar computation several times in Section \ref{NonAbelLoc}. By the expression of $TY$ at the end of Section \ref{NonAbelLoc}, we have
\[
TY|_{CE_{1,2}} = \frac{\Lambda_{31}}{\Lambda_{32}}\O + \frac{\Lambda_{33}}{\Lambda_{32}}\O + \frac{\Lambda_{33}}{\Lambda_{31}}\O(1) + \frac{\Lambda_{21}\Lambda_{31}}{\Lambda_{22}\Lambda_{32}}\O(-1) + \frac{\Lambda_{32}}{\Lambda_{31}}\O(2),
\]
where $\O(-1)$ stands for the $\wT$-equivariant tautological bundle over $CE_{1,2}\simeq \C P^1$ with $\wT$ acting trivially on the fiber at $C$ (only one such lifting exists), and $\O(k) = \O(-1)^{\otimes -k}$, $\forall k\in\mathbb{Z}$. Therefore,
\begin{align*}
\coeff^{tw,Y}_{CE_{1,2}}(m) = & -\frac{Q_{21}^m}{m} \cdot
\frac{\prod_{l=1}^m\left(1-y\frac{\Lambda_{31}}{\Lambda_{32}}\left(\frac{\Lambda_{32}}{\Lambda_{31}}\right)^{\frac{l}{m}}\right) }
{\prod_{l=1}^m\left(1-y\left(\frac{\Lambda_{32}}{\Lambda_{31}}\right)^{\frac{l}{m}}\right) } \\
& \cdot \frac{\prod_{l=1}^m\left(1-\frac{\Lambda_{22}}{\Lambda_{21}}\left(\frac{\Lambda_{32}}{\Lambda_{31}}\right)^{\frac{l}{m}}\right)}
{\prod_{l=1}^m\left(1-\frac{\Lambda_{31}}{\Lambda_{33}}\left(\frac{\Lambda_{32}}{\Lambda_{31}}\right)^{\frac{l}{m}}\right) 
\prod_{l=1}^m\left(1-\left(\frac{\Lambda_{32}}{\Lambda_{31}}\right)^{\frac{l}{m}}\right) 
\prod_{l=1}^{m-1}\left(1-\frac{\Lambda_{31}}{\Lambda_{32}}\left(\frac{\Lambda_{32}}{\Lambda_{31}}\right)^{\frac{l}{m}}\right) }.
\end{align*}
All terms on RHS, no matter in numerator or denominator, are well-defined and invertible under the specialization $\Lambda_{2s}=1, \Lambda_{3s} = \Lambda_s, Q_{is} = Q_i, y=1$, except the one in numerator of the second factor with $l=m$, which gives
\[
1-y\frac{\Lambda_{31}}{\Lambda_{32}}\left(\frac{\Lambda_{32}}{\Lambda_{31}}\right)^{\frac{m}{m}} = 1 - y = 0, \quad \text{ as } y=1.
\]
Note that this factor comes from the twisting $(\Eu,y^{-1}\mathfrak{g}/\mathfrak{s})$.
\begin{remark}
The direct proof of Proposition \ref{nonisoC} above is included mostly for expository purposes. In fact, it works in this special case only because the tangent $T$-character $\Lambda_2/\Lambda_1$ along $CE_{1,2}$ appears with multiplicity 1 in $T_CY$, which implies $CE_{1,2}$ is an isolated 1-dim $T$-orbit. In general, $\Lambda_2/\Lambda_1$ may appear with higher multiplicity in $T_CY$ (even though it does not appear in $T_CX$ at all), and we will need a more systematic method, similar to the proof of Proposition \ref{nonisoA} that follows. Nevertheless, the term $(1-y)$ will always persist, which is intuitively the reason why the expected vanishing happens.
\end{remark} 

\par Then we proceed to the proof of Proposition \ref{nonisoA}. On LHS there are only two terms in this case, one from $AE_2$ and one from $AE_1 = AC$. The proof below contains essentially a re-summation argument. In non-rigorous terms, we will see:
\begin{itemize}
    \item The contribution from $AE_1$ is further decomposed into one involving $G_{E_1}(q)$, the principal part of $J^{tw,Y}_{E_1}(q)$ containing poles descending possibly to $(\Lambda_2/\Lambda_1)^{1/m}$, and one involving $\overline{G}(q)$, the remaining regular part.
    \item The former, which comes virtually from recursion along the broken orbit $A\rightarrow E_1\rightarrow E_{1,2}$, pairs with the contribution from $AE_2$ to form the ``non-isolated recursion'' from $A$ to $\Omega$, and altogether they vanish due to a (reversed) fixed-point localization argument. This is \textbf{Step I}.
    \item The latter takes well-defined value at $q = (\Lambda_2/\Lambda_1)^{1/m}$, and descends to RHS of the equation in Proposition \ref{nonisoA}, by the isolated recursion formula of Section \ref{IsoRec} and the vanishing result proven in Proposition \ref{nonisoC}. This is \textbf{Step II}.
\end{itemize}
Denote by $S_1$ the contribution from $AE_2$. Explicit computation shows
\begin{align*}
S_1:= & \frac{-1}{1-q\left(\frac{\Lambda_{32}}{\Lambda_{31}}\right)^{-\frac{1}{m}}} \cdot \coeff^{tw,Y}_{AE_2}(m) \cdot 
J^{tw,Y}_{E_2}\left(\left(\frac{\Lambda_{32}}{\Lambda_{31}}\right)^{\frac{1}{m}}\right) \\
= & \frac{1}{1-q\left(\frac{\Lambda_{32}}{\Lambda_{31}}\right)^{-\frac{1}{m}}} \cdot \frac{Q_{11}^mQ_{21}^m}{m} \cdot J^{tw,Y}_{E_2}\left(\left(\frac{\Lambda_{32}}{\Lambda_{31}}\right)^{\frac{1}{m}}\right) \cdot
\frac{ \left(y;\left(\frac{\Lambda_{31}}{\Lambda_{32}}\right)^{\frac{1}{m}}\right)_{-m} }
{ \left(y\frac{\Lambda_{31}}{\Lambda_{32}};\left(\frac{\Lambda_{31}}{\Lambda_{32}}\right)^{\frac{1}{m}}\right)_{-m} } \cdot \\
& \left(\frac{\Lambda_{31}}{\Lambda_{33}}, 1, \frac{\Lambda_{21}\Lambda_{31}}{\Lambda_{22}\Lambda_{32}};\left(\frac{\Lambda_{31}}{\Lambda_{32}}\right)^{\frac{1}{m}}\right)_{-m} 
\left(\frac{\Lambda_{31}}{\Lambda_{32}};\left(\frac{\Lambda_{31}}{\Lambda_{32}}\right)^{\frac{1}{m}}\right)_{1-m}
\end{align*}
where we use the $q$-Pochhammer symbols for simplicity of notation:
\[
(a;q)_{-n} = \prod_{l=1}^n \frac{1}{1-aq^{-l}}, \quad (a_1,\cdots,a_K;q)_{-n} = \prod_{k=1}^K (a_k;q)_{-n}
\]
for integer $n$. All except two $q$-Pochhammer symbols in the expression of $S_1$ are well-defined and invertible under the specialization $\Lambda_{2s}=1, \Lambda_{3s} = \Lambda_s, Q_{is} = Q_i, y=1$:
\begin{itemize}
    \item The last factor of $\left(y\frac{\Lambda_{31}}{\Lambda_{32}};\left(\frac{\Lambda_{31}}{\Lambda_{32}}\right)^{\frac{1}{m}}\right)_{-m}$ gives $(1-y)$, which as is in the proof of Proposition \ref{nonisoC} reduces to zero as $y\rightarrow 1$, in the \emph{numerator} of $S_1$.
    \item The last factor of $\left(\frac{\Lambda_{21}\Lambda_{31}}{\Lambda_{22}\Lambda_{32}};\left(\frac{\Lambda_{31}}{\Lambda_{32}}\right)^{\frac{1}{m}}\right)_{-m}$ gives $(1-\Lambda_{21}/\Lambda_{22})$, which is new for Proposition \ref{nonisoA} and reduces to zero as $\wT\rightarrow T$, in the \emph{denominator} of $S_1$.
\end{itemize}
This is exactly why naive computation does not work as a legitimate proof in this case. Separating out the two special terms and decomposing $1/(1-\Lambda_{21}/\Lambda_{22})$ further into $m$-th roots, we obtain
\begin{align*}
S_1 = & \sum_{\alpha=1}^m \frac{1-y}{1-q\left(\frac{\Lambda_{32}}{\Lambda_{31}}\right)^{-\frac{1}{m}}} \cdot \frac{Q_{11}^mQ_{21}^m}{m^2} \cdot J^{tw,Y}_{E_2}\left(\left(\frac{\Lambda_{32}}{\Lambda_{31}}\right)^{\frac{1}{m}}\right) \cdot
\frac{ \left(y;\left(\frac{\Lambda_{31}}{\Lambda_{32}}\right)^{\frac{1}{m}}\right)_{-m} }
{ \left(y\frac{\Lambda_{31}}{\Lambda_{32}};\left(\frac{\Lambda_{31}}{\Lambda_{32}}\right)^{\frac{1}{m}}\right)_{1-m} } \cdot 
\frac{1}{1-\zeta^\alpha \left(\frac{\Lambda_{21}}{\Lambda_{22}}\right)^{\frac{1}{m}}} \\
& \cdot \left(\frac{\Lambda_{31}}{\Lambda_{33}}, 1;\left(\frac{\Lambda_{31}}{\Lambda_{32}}\right)^{\frac{1}{m}}\right)_{-m} 
\left(\frac{\Lambda_{31}}{\Lambda_{32}};\left(\frac{\Lambda_{31}}{\Lambda_{32}}\right)^{\frac{1}{m}}\right)_{1-m} \cdot 
\left(\frac{\Lambda_{21}\Lambda_{31}}{\Lambda_{22}\Lambda_{32}};\left(\frac{\Lambda_{31}}{\Lambda_{32}}\right)^{\frac{1}{m}}\right)_{1-m} 
\end{align*}
We denote by $S_1^\alpha$ the summand of $S_1$ associated to $\alpha$ above. Here we use
\[
\frac{1}{1-\frac{\Lambda_{21}}{\Lambda_{22}}} = \frac{1}{m}\sum_{\alpha=1}^m \frac{1}{1-\zeta^\alpha \left(\frac{\Lambda_{21}}{\Lambda_{22}}\right)^{\frac{1}{m}}}
\]
where $\zeta = e^{2\pi i/m}$ and $(\Lambda_{21}/\Lambda_{22})^{1/m}$ is one specific choice of $m$-th root. Recall that in order for the statement of Proposition \ref{nonisoA} to be correct, we should eventually take the sum over all $m$ choices of $(\Lambda_{32}/\Lambda_{31})^{1/m}$ as well, but we fix one specific choice for now.

\par Following the idea of \textbf{Step I}, we may supplement $S_1$ with contribution from the broken orbit $A\rightarrow E_1 \rightarrow E_{1,2}$ and expect the sum to vanish. Recall that the tangent character along $AE_1$ is $\Lambda_{22}\Lambda_{32}/\Lambda_{21}\Lambda_{31}$, and along $E_1E_{1,2}$, or $CE_{1,2}$, is $\Lambda_{32}/\Lambda_{31}$. Fixing $m$, let $G_{E_1}(q)$ be the part of $J^{tw,Y}_{E_1}(q)$ as appearing in Proposition \ref{nonisoC} with simple pole at $q = (\Lambda_{32}/\Lambda_{31})^{1/m}$. Then, we define the ``recursive'' contribution of $A\rightarrow E_1 \rightarrow E_{1,2}$ to the pole as
\begin{align*}
S_2 := & \frac{-1}{1-q\left(\frac{\Lambda_{22}\Lambda_{32}}{\Lambda_{21}\Lambda_{31}}\right)^{-\frac{1}{m}}} \cdot 
\coeff^{tw,Y}_{AE_1}(m) \cdot
G_{E_1}\left(\left(\frac{\Lambda_{22}\Lambda_{32}}{\Lambda_{21}\Lambda_{31}}\right)^{\frac{1}{m}}\right) \\
= & \frac{-1}{1-q\left(\frac{\Lambda_{22}\Lambda_{32}}{\Lambda_{21}\Lambda_{31}}\right)^{-\frac{1}{m}}} \cdot 
\coeff^{tw,Y}_{AE_1}(m) \cdot 
\frac{-1}{ 1-\left(\frac{\Lambda_{22}}{\Lambda_{21}}\right)^{\frac{1}{m}} } \cdot
\coeff^{tw,Y}_{E_1E_{1,2}}(m) \cdot 
J^{tw,Y}_{E_{1,2}}\left( \left(
\frac{\Lambda_{32}}{\Lambda_{31}}\right)^{\frac{1}{m}} \right).
\end{align*}
Here the middle term comes from the evaluation
\[
\frac{-1}{ 1-\left(\frac{\Lambda_{22}}{\Lambda_{21}}\right)^{\frac{1}{m}}} = \left.\frac{-1}{1-q\left(\frac{\Lambda_{32}}{\Lambda_{31}}\right)^{-\frac{1}{m}}}\right|_{q = \left(\frac{\Lambda_{22}\Lambda_{32}}{\Lambda_{21}\Lambda_{31}}\right)^{\frac{1}{m}}}.
\]
Geometrically, it may be regarded as coming from smoothing the node of $A\rightarrow E_1 \rightarrow E_{1,2}$ at $E_1$.

\par In this way, for a fixed choice of $m$-th root of $\Lambda_{32}/\Lambda_{31}$, choices of $m$-th roots of $\Lambda_{22}\Lambda_{32}/\Lambda_{21}\Lambda_{31}$ are in one-to-one correspondence with choices of $m$-th roots of $\Lambda_{22}/\Lambda_{21}$, and different such choices give rise to different values of $S_2$. We denote by $S_2^\alpha$ the value when the $m$-th root of $\Lambda_{22}/\Lambda_{21}$ is taken as the inverse of $\zeta^{\alpha} (\Lambda_{21}/\Lambda_{22})^{1/m}$ as appearing in $S_1^\alpha$, for $1\leq\alpha\leq m$. Direct computation gives
\begin{align*}
S_2^\alpha = & \frac{1-y}{1-q\left(\frac{\Lambda_{32}}{\Lambda_{31}}\right)^{-\frac{1}{m}} \cdot \zeta^\alpha \left(\frac{\Lambda_{22}}{\Lambda_{21}}\right)^{-\frac{1}{m}}} \cdot 
\frac{Q_{11}^mQ_{21}^m}{m^2} \cdot J^{tw,Y}_{E_{1,2}}\left(\left(\frac{\Lambda_{32}}{\Lambda_{31}}\right)^{\frac{1}{m}}\right) \cdot
\frac{ \left(y;\left(\frac{\Lambda_{31}}{\Lambda_{32}}\right)^{\frac{1}{m}}\right)_{-m} }
{ \left(y\frac{\Lambda_{31}}{\Lambda_{32}};\left(\frac{\Lambda_{31}}{\Lambda_{32}}\right)^{\frac{1}{m}}\right)_{1-m} } \\ 
& \cdot \frac{1}{1- \zeta^{-\alpha} \left(\frac{\Lambda_{22}}{\Lambda_{21}}\right)^{\frac{1}{m}}} \cdot \left(\frac{\Lambda_{31}}{\Lambda_{33}}, 1;\left(\frac{\Lambda_{31}}{\Lambda_{32}}\right)^{\frac{1}{m}}\right)_{-m} 
\left(\frac{\Lambda_{31}}{\Lambda_{32}};\left(\frac{\Lambda_{31}}{\Lambda_{32}}\right)^{\frac{1}{m}}\right)_{1-m} \\
& \cdot \frac{\left(1; \zeta^\alpha \left(\frac{\Lambda_{21}}{\Lambda_{22}}\right)^{\frac{1}{m}} \left(\frac{\Lambda_{31}}{\Lambda_{32}}\right)^{\frac{1}{m}} \right)_{-m} 
\left(\frac{\Lambda_{21}\Lambda_{31}}{\Lambda_{22}\Lambda_{32}};\zeta^\alpha \left(\frac{\Lambda_{21}}{\Lambda_{22}}\right)^{\frac{1}{m}} \left(\frac{\Lambda_{31}}{\Lambda_{32}}\right)^{\frac{1}{m}} \right)_{1-m} }
{\left(\frac{\Lambda_{22}}{\Lambda_{21}}; \left(\frac{\Lambda_{31}}{\Lambda_{32}}\right)^{\frac{1}{m}}\right)_{-m}}.
\end{align*}
We have re-arranged certain terms in $S_2^\alpha$ to turn it into the above form very similar to $S_1^\alpha$.

\par For \textbf{Step I}, it remains now to prove $S_1^\alpha + S_2^\alpha = 0$ for any $\alpha$, under the specialization of parameters in Proposition \ref{nonisoA}. Recall that there is exactly one $T$-invariant orbit from $A$ to each point in the $T$-fixed point component $\Omega$ except $E_{1,2}$, for which one would need the broken orbit $A\rightarrow E_1\rightarrow E_{1,2}$. We denote such space of orbits by $\M_A(\Lambda_2/\Lambda_1)$. Our idea is to write $S_1^\alpha + S_2^\alpha$ as torus fixed point localization\footnote{Let $M$ be a smooth projective variety where torus $\tau$ acts with isolated fixed points, and $\mathcal{F}$ be a vector bundle over $M$. Then, $$\chi(M;\mathcal{F}) = \sum_{p\in\text{Fix}_\tau(M)}\frac{\mathcal{F}|_p}{\Eu_\tau(N_pM)}.$$ Here $\Eu_\tau$ denotes the $\tau$-equivariant K-theoretic Euler class, defined by $\Eu_\tau(V) = \sum_{i} (-1)^i \Lambda^i V^\vee$.} of $\chi(\M_A(\Lambda_2/\Lambda_1), (1-y)\mathcal{F})$ for certain sheaf $\mathcal{F}$, so that
\begin{itemize}
    \item $\frac{1}{1-\zeta^\alpha \left(\frac{\Lambda_{21}}{\Lambda_{22}}\right)^{\frac{1}{m}}}$ in $S_1^\alpha$ and $\frac{1}{1- \zeta^{-\alpha} \left(\frac{\Lambda_{22}}{\Lambda_{21}}\right)^{\frac{1}{m}}}$ in $S_2^\alpha$ appear as the Euler class of normal bundle at the two fixed points in $\M_A(\Lambda_2/\Lambda_1)$;
    \item all remaining factors of $S_1^\alpha$ and $S_2^\alpha$ go to $\mathcal{F}$, which is well-defined under the specialization $\Lambda_{2s}=1, \Lambda_{3s} = \Lambda_s, Q_{is} = Q_i, y=1$.
\end{itemize}
More precisely, we define $\M_A(\Lambda_2/\Lambda_1)$ as the space of \textbf{maximal balanced broken orbits} from $A$ with tangent character $\Lambda_2/\Lambda_1$. In other words, we consider chains of $\C P^1$'s originating from $A$ that satisfies 
\begin{itemize}
    \item[(i)] the tangent $T$-character at $A$ is $\Lambda_2/\Lambda_1$;
    \item[(ii)] the tangent $T$-characters of the two components at each node of the chain are inverse to each other;
    \item[(iii)] such chain can no longer be elongated.
\end{itemize}
Such chains must terminate at a point in $\Omega$. As a set, $\M_A(\Lambda_2/\Lambda_1)$ is thus in one-to-one correspondence with $\Omega$, and it inherits naturally a reduced scheme structure from $\Omega \simeq \C P^1$ by the ``evaluation map'' sending each broken orbit to its end point.

\par We introduce a fictional $\widetilde{T}$-action on $\M_A(\Lambda_2/\Lambda_1) \simeq \C P^1$ so that $AE_2$ and $AE_1E_{1,2}$ are the two fixed points, and the tangent characters at the two points are $(\zeta^\alpha (\Lambda_{21}/\Lambda_{22})^{1/m})^{\mp 1}$. In particular,
\[
\Eu(N_{AE_2}\M_A(\Lambda_2/\Lambda_1)) = 1-\zeta^\alpha \left(\frac{\Lambda_{21}}{\Lambda_{22}}\right)^{\frac{1}{m}}, \quad \Eu(N_{AE_1E_{1,2}}\M_A(\Lambda_2/\Lambda_1)) = 1- \zeta^{-\alpha} \left(\frac{\Lambda_{22}}{\Lambda_{21}}\right)^{\frac{1}{m}}.
\]
Note that this action is different from the existing one on $\Omega\subset Y$, which would have the same fixed points but tangent characters $(\Lambda_{21}/\Lambda_{22})^{\mp 1}$ instead. 

\par We denote by $P = \mathcal{O}(-1)$ the tautological line bundle over $\M_A(\Lambda_2/\Lambda_1) \simeq \C P^1$ with $\wT$-action on fibers
\[
P|_{AE_2} = 1, \quad P|_{AE_1E_{1,2}} = \zeta^{-\alpha} \left(\frac{\Lambda_{22}}{\Lambda_{21}}\right)^{\frac{1}{m}}.
\]

\par We are now ready to find the proper $\mathcal{F}^\alpha = \mathcal{F}^\alpha(P,\Lambda,y,Q,q)$ such that
\begin{align*}
  S_1^\alpha + S_2^\alpha & = \frac{(1-y)\mathcal{F}^\alpha(1,\Lambda,y,Q,q)}{\Eu(N_{AE_2}\M_A(\Lambda_2/\Lambda_1))} + \frac{(1-y)\mathcal{F}^\alpha(\zeta^{-\alpha} \left(\frac{\Lambda_{22}}{\Lambda_{21}}\right)^{\frac{1}{m}},\Lambda,y,Q,q)}{\Eu(N_{AE_1 E_{1,2}}\M_A(\Lambda_2/\Lambda_1))} \\
  & = \chi(\M_A(\Lambda_2/\Lambda_1));(1-y)\mathcal{F}^\alpha). 
\end{align*}
In fact, through comparing directly the expression of $S_1^\alpha$ and $S_2^\alpha$, it is not hard to see that it suffices to prove $J^{tw,Y}_{E_2}(q)$ and $J^{tw,Y}_{E_{1,2}}(q)$ come from the specialization of a common expression in $P$ at $P=1$ and $P = \zeta^{-\alpha} (\Lambda_{22}/\Lambda_{21})^{1/m}$ respectively. This is indeed true because $J^{tw,Y}(q)$ depends entirely on the tautological bundles $P_{is}$ of $Y$, among which only $P_{11}$ restricts to a non-trivial bundle on $\Omega$. Since $P_{11}|_{E_2} = \Lambda_{21}\Lambda_{32}$ and $P_{11}|_{E_{1,2}} = \Lambda_{22}\Lambda_{32}$, its contribution to $S_1^\alpha$ and $S_2^\alpha$ may be identified with that of $P^{m}\Lambda_{21}\Lambda_{32}$ over $\M_A(\Lambda_2/\Lambda_1)$.

\par Under the specialization $\Lambda_{2s}=1, \Lambda_{3s} = \Lambda_s, Q_{is} = Q_i, y=1$, all factors of $\mathcal{F}^\alpha$ are well-defined (that is to say, no non-zero terms in the denominator), and
\[
S_1^\alpha + S_2^\alpha = \chi(\M_A(\Lambda_2/\Lambda_1));(1-y)\mathcal{F}) = 0.
\]
Summing up the terms for all $\alpha$ as well as all choices of $(\Lambda_{32}/\Lambda_{31})^{1/m}$ made at the beginning, which we now denote by $\zeta^{-\beta} \cdot (\Lambda_{32}/\Lambda_{31})^{1/m}$, we obtain
\begin{align*}
0 = & \sum_{\beta}\frac{-1}{1-q \zeta^{\beta} \left(\frac{\Lambda_{32}}{\Lambda_{31}}\right)^{-\frac{1}{m}}} \cdot \coeff^{tw,Y,\beta}_{AE_2}(m) \cdot J^{tw,Y}_{E_2}\left(\zeta^{-\beta}\left(\frac{\Lambda_{32}}{\Lambda_{31}}\right)^{\frac{1}{m}}\right) + \\
& \sum_{\alpha,\beta} \frac{-1}{1-q \zeta^{\alpha + \beta} \left(\frac{\Lambda_{22}\Lambda_{32}}{\Lambda_{21}\Lambda_{31}}\right)^{-\frac{1}{m}}} \cdot \coeff^{tw,Y,\alpha+\beta}_{AE_1}(m) \cdot
G_{E_1}^\beta\left(\zeta^{ - \alpha - \beta} \left(\frac{\Lambda_{22}\Lambda_{32}}{\Lambda_{21}\Lambda_{31}}\right)^{\frac{1}{m}}\right),
\end{align*}
where $G_{E_1}^\beta(q)$ is the (possible) part in $J^{tw,Y}_{E_1}(q)$ with simple pole at $q = \zeta^{-\beta}(\Lambda_{32}/\Lambda_{31})^{1/m}$, and the superscripts over the recursion coefficients are to indicate their dependence on the choices of $m$-th roots. This finishes \textbf{Step I}. 

\par Combining the above equation with what have been proven in Section \ref{IsoRec}, and then comparing with Proposition \ref{nonisoA} (where the summation over $m$-roots of characters is implied), we see that it remains to show 
\[
\overline{G}\left(\zeta^{\gamma} \left(\frac{\Lambda_{22}\Lambda_{32}}{\Lambda_{21}\Lambda_{31}}\right)^{\frac{1}{m}}\right) := (J^{tw,Y}_{E_1} - \sum_{\beta}  G_{E_1}^\beta)\left(\zeta^{\gamma} \left(\frac{\Lambda_{22}\Lambda_{32}}{\Lambda_{21}\Lambda_{31}}\right)^{\frac{1}{m}}\right)
\]
is well-defined and equals $J^X_{E_1}(\zeta^\gamma(\Lambda_{2}/\Lambda_1)^{1/m})$ for any $1\leq \gamma = - \alpha - \beta \leq m$, under the specialization $\Lambda_{2s}=1, \Lambda_{3s} = \Lambda_s, Q_{is} = Q_i, y=1$. Recall that $G_{E_1}^\gamma(q)$ is exactly the principal part of $J^{tw,Y}_{E_1}(q)$ at $q = \zeta^\gamma(\Lambda_{22}\Lambda_{32}/\Lambda_{21}\Lambda_{31})^{1/m}$ for any $\gamma$. Therefore, $\overline{G}(q)$ is by definition free of any such poles, or any poles that would possibly descend to one at $q = \zeta^\gamma(\Lambda_{2}/\Lambda_1)^{1/m}$, which means the quantity above is well-defined. That it equals $J^X_{E_1}(\zeta^\gamma(\Lambda_{2}/\Lambda_1)^{1/m})$ under the specialization follows directly from Proposition \ref{nonisoC} that $G_{E_1}^\beta(q) = 0$ for any $\beta$. In other words, the part we ``borrowed'' from $J^{tw,Y}_{E_1}(q)$ to fill in the contribution $S_2$ of balanced broken orbits really vanishes. This last paragraph is the content of \textbf{Step II}.

\par We end this section with the following remark.
\begin{remark}
The choices of roots of unity and $m$-th roots of $T$-characters that appear in our construction of $T$-action on $\M_A(\Lambda_2/\Lambda_1)$ indicates that a more intrinsic construction of $\M_A(\Lambda_2/\Lambda_1)$ should exist. Morally the space of orbits is not the isomorphic to $\Omega$, but rather with each point an isotropy, and the isotropy group grows as the chain of $\C P^1$ elongates. We refrain ourselves from going too far in this direction since the formal algebraic deduction above suffices for our purposes. Yet it is indeed an interesting topic to consider, of finding a correct and rigorous definition of the space of orbits in between fixed point components.
\end{remark}

\subsection{General case: Step I} \label{StepI}
\par The idea is still to supplement the recursive contribution from each $AE_k\ (k\geq 2)$ with terms coming from balanced broken orbits on the abelian quotient $Y$, so that the sum vanishes under the specialization described in the Proposition. More precisely, for each $AE_k$, we consider chains of $\C P^1$'s starting from $A$ such that
\begin{itemize}
    \item each node is a $T$-fixed point and each irreducible component is the closure of a 1-dim $T$-orbit,
    \item the sum of homological degrees of the irreducible components is equal to that of $AE_k$,
    \item the tangent $T$-character at $A$ is $\Lambda_{v_1+1}/\Lambda_1$, and the tangent $T$-characters of the two components at each node are inverse to each other. 
\end{itemize}
In this section, we only consider the case of $AE_n$ as an example as all other cases are parallel. 

\par For $AE_n$, we denote the space of such balanced broken orbits by $\M_A(\Lambda_{v_1+1}/\Lambda_1)$, following the notation before. Recall that $H^2(Y)$ is generated by the first Chern classes of the tautological bundles $P_{is}$. We denote by $\mathbf{1}_{is}$ the (negative) dual basis of $H_2(Y)$. Then, $AE_n$ has homological degree $\sum_{i=1}^n \mathbf{1}_{i1}$.

\par Orbits in $\M_A(\Lambda_{v_1+1}/\Lambda_1)$ must end at a $T$-fixed point $E$ (regarded as an element in $R$, i.e. a sequence of matrices) which satisfies 
\begin{itemize}
    \item the last matrix of $E$ is $I_{N\times v_n,v_1+1}$.
    \item and for $i\leq n-1$, the $i$-th matrix of $E$ is $a_i I_{v_{i+1}\times v_i} + b_i I_{v_{i+1}\times v_i, v_1+1}$ for $a_i,b_i\in\C$ not both zero.
\end{itemize}
The definition of $I_{v\times w}$ and $I_{v\times w,k}$ was given at the beginning of Section \ref{NonisoRec}. $E_n$ itself obviously satisfies the two conditions. Moreover, every $E$ that satisfies the above conditions determines a unique orbit in $\M_A(\Lambda_{v_1+1}/\Lambda_1)$. In fact, the number of irreducible components of the orbit determined by $E$ depends on how many of those $a_i$ vanish.

\par $\M_A(\Lambda_{v_1+1}/\Lambda_1)$ is thus identified with the locus of possible endpoints $E\in Y$, which we denote still by $\Omega$, as a set, and inherits from the latter a scheme structure. Geometrically it is a tower of $\C P^1$-fiber bundles. The only tautological bundles of $Y$ that restrict non-trivially to $\Omega$ are $P_{i1}$ for $1\leq i\leq n-1$, and they generate the K-ring of $\M_A(\Lambda_{v_1+1}/\Lambda_1)\simeq \Omega$. We denote these bundles by $P_i\ (1\leq i\leq n-1)$.

\par $\Omega$ is invariant under the $\wT$-(and thus $T-$)action on $Y$, and the fixed points are exactly those $E_{\JJ}$ (notation see the beginning of Section \ref{NonisoRec}) for $\{n\}\subset \JJ \subset \{1,2,\cdots,n\}$. 

\par As we have seen in Section \ref{specialcasen2}, however, such $\wT$-action is not the most convenient on $\M_A(\Lambda_{v_1+1}/\Lambda_1)$ for our intended formal computation of fixed-point localization. Following the intuition of the special case of $X = \text{Flag}(1,2;3)$, we consider the (formal) $\wT$-action on $\M_A(\Lambda_{v_1+1}/\Lambda_1)\simeq\Omega$ where the $s$-th row of the $i$-th matrix is scaled by $\Lambda_{is}^{1/m}$ instead of $\Lambda_{is}$. Under this new action, the fixed points are still (balanced broken orbits from $A$ to) $E_{\JJ}$. 

\par For simplicity of notation, given any set $\JJ$ containing elements of $j_1<\cdots<j_{K} = n$, we define for $1\leq k\leq K$ the truncation $\JJ[k] := \{j_1,\cdots,j_k\}$. In this way, $\JJ[K] = \JJ$. Moreover, for any $1\leq i\leq n$, we define $k_{\JJ}[i]$ as the index of the smallest element in $\JJ$ that is not smaller than $i$, i.e. such that $j_{k_{\JJ}[i]-1}<i\leq j_{k_{\JJ}[i]}$. When there is no confusion with $\JJ$, we simply write $k[i] = k_{\JJ}[i]$.

\par We have the following fact regarding the tautological bundles of $\M_A(\Lambda_{v_1+1}/\Lambda_1)$.
\begin{lemma} \label{tautolem}
\[
P_i|_{\JJ} = \left(\prod_{a=i+1}^{j_{k_{\JJ}[i]}} \Lambda_{a,1} \cdot \prod_{a=j_{k_{\JJ}[i]}+1}^{n+1} \Lambda_{a,v_1+1}\right)^{\frac{1}{m}}.
\] 
\end{lemma}
\begin{corollary} \label{tautocor}
For any $1\leq i\leq n$, there exists (formal) $\wT$-equivariant line bundle over $\M_A(\Lambda_{v_1+1}/\Lambda_1)$ whose fiber at $E_{\JJ}$ has $\wT$-character $\lambda_{j_{k_{\JJ}[i]}}$.
\end{corollary}
The bundle is simply $P_i^m/\lambda_{i,1}$. For definition of $\lambda_{i,1}$ and $\lambda_{j_{k_{\JJ}[i]}}$, see the paragraph before Proposition \ref{vanishing1}. Another ingredient that we would need later is the tangent $\wT$-characters.
\begin{lemma} \label{tangentlem}
The tangent space $T_{E_{\JJ}}\M_A(\Lambda_{v_1+1}/\Lambda_1)$ has dimension $n-1$. Moreover, 
\[
T_{E_{\JJ}}\M_A(\Lambda_{v_1+1}/\Lambda_1) = V_1 + V_2 := \sum_{i = j_k\in\JJ, i<n} \left(\frac{\lambda_{j_{k+1}}}{\lambda_{j_k}}\right)^{\frac{1}{m}} + \sum_{i\notin\JJ, i<n} \left(\frac{\lambda_{i}}{\lambda_{j_{k_{\JJ}[i]}}}\right)^{\frac{1}{m}} \ \in \ R(\wT).
\]
More precisely, below are the tangent $\wT$-characters at $E_{\JJ}$:
\begin{align*}
& \left(\frac{\lambda_{j_{k+1}}}{\lambda_{j_k}}\right)^{\frac{1}{m}} = \prod_{a=j_k+1}^{j_{k+1}}\left(\frac{\Lambda_{a,1}}{\Lambda_{a,v_1+1}}\right)^{\frac{1}{m}}, & \text{    for } 1\leq i\leq n-1 \text{ and } i = j_k\in {\mathfrak{J}};\\
& \left(\frac{\lambda_{i}}{\lambda_{j_{k_{\JJ}[i]}}}\right)^{\frac{1}{m}} = \prod_{a=i+1}^{j_{k_{\JJ}[i]}}\left(\frac{\Lambda_{a,v_1+1}}{\Lambda_{a,1}}\right)^{\frac{1}{m}}, & \text{    for } 1\leq i\leq n-1 \text{ and } i\notin {\mathfrak{J}}.
\end{align*}
\end{lemma}

\par Furthermore, with the above notation of truncation, we may obtain a clearer description of the broken orbit from $A$ to $E_{\JJ}$ as well. In fact, it must take the form
\[
A = E_{\phi} = E_{\JJ[0]}\rightarrow E_{\JJ[1]}\rightarrow \cdots\rightarrow E_{\JJ[K-1]}\rightarrow E_{\JJ[K]} = E_{\JJ}
\]
where $K = |\JJ|$. The component $E_{\JJ[k-1]}E_{\JJ[k]}$ has homological degree $\sum_{i=j_{k-1}+1}^{j_k}\one_{i1}$, and its tangent $\wT$-character at $E_{\JJ[k-1]}$ is exactly $\lambda_{j_k}$ . In particular, the tangent character of the last component is always $\lambda_{n} = \Lambda_{n+1,v_1+1}/\Lambda_{n+1,1} \in R(\wT) \rightarrow \Lambda_{v_1+1}/\Lambda_1 \in R(T)$, independent of $\JJ$.

\par The contribution to relevant poles of $J^{tw,Y}_A(q)$ along the broken orbit from $A$ to $E_{\JJ}$ is, by definition, 
\[
S_{\JJ} = \frac{1}{1-q\lambda_{j_1}^{-\frac{1}{m}}}\cdot \prod_{k=1}^{|\JJ| - 1} \frac{1}{1-\lambda_{j_k}^{\frac{1}{m}}\lambda_{j_{k+1}}^{-\frac{1}{m}}} \cdot \frac{\prod_{i=1}^nQ_{i1}^m}{m^{|\JJ|}} \cdot \prod_{k=0}^{|\JJ|-1}\coeff^{tw,Y}_{E_{\JJ[k]}E_{\JJ[k+1]}}(m) \cdot J^{tw,Y}_{E_{\JJ}}\left(\lambda_{n}^{\frac{1}{m}}\right).
\]
Recall that here the recursion coefficient along any 1-dim $\wT$-orbit $ab$ takes the form
\[
\coeff^{tw,Y}_{PQ}(m) = \frac{\Eu(TY)|_a}{\Eu((TY)_{0,2,mD} - 1)|_\phi} \cdot \frac{\Eu((y^{-1}\mathfrak{g}/\mathfrak{s})_{0,2,mD})|_\phi}{\Eu(y^{-1}\mathfrak{g}/\mathfrak{s})|_a}
\]
as we discussed in Section \ref{IsoRec}, where the former is contribution from $TY$ and the latter is that from $y^{-1}\mathfrak{g}/\mathfrak{s}$. Here $D$ is the homological degree of $ab$, $\phi: \C P^1\rightarrow ab$ is the $m$-sheet covering ramified at $a$ and $b$, and the $(-1)$ term in denominator comes from the 1-dim automorphism group of $\C P^1$ with two marked points. Below we re-write the goal of \textbf{Step I}.
\begin{proposition} \label{prop8}
There exists a bundle $\mathcal{F} = \mathcal{F}(P_1,\cdots,P_{n-1})\in K_{\wT}(\M_A(\Lambda_{v_1+1}/\Lambda_1))$, well-defined under the specialization as in Proposition \ref{nonisoA}, such that 
\[
S_{\JJ} = \frac{(1-y) \cdot \mathcal{F}(P_1,\cdots,P_{n-1})|_{E_{\JJ}}}{\Eu(N_{E_{\JJ}}\M_A(\Lambda_{v_1+1}/\Lambda_1))}.
\]
\end{proposition}
The proposition indicates that we may write $\sum_{\JJ} S_{\JJ}$ as the $\wT$-fixed point localization of the Euler characteristic $\chi({\M}_A(\Lambda_{v_1+1}/\Lambda_1);(1-y)\mathcal{F})$, which vanishes as we take $y=1$. It follows then that residues from $S_{\JJ}$ sum up to zero, and thus no genuine pole of $J^X_A$ at $q = (\frac{\Lambda_{v_1+1}}{\Lambda_1})^{1/m}$ arises in this way. 

\par The rest of the section is devoted to the proof of this proposition through analyzing separately the ingredients appearing in $S_{\JJ}$, which boils mostly down to direct computation given all the preparation above.

\subsubsection*{Contribution from $y^{-1}\mathfrak{g}/\mathfrak{s}$}
\par The contribution from the twisting bundle $y^{-1}\mathfrak{g}/\mathfrak{s}$ takes the form 
\[
\prod_{k=0}^{|\JJ|-1} \frac{\Eu((y^{-1}\mathfrak{g}/\mathfrak{s})_{0,2,m\cdot D_{E_{\JJ[k]}E_{\JJ[k+1]}}})|_{\phi_{E_{\JJ[k]}E_{\JJ[k+1]}}}}{\Eu(y^{-1}\mathfrak{g}/\mathfrak{s})|_{E_{\JJ[k]}}}
\]
Recall that we have decomposition
\[
\mathfrak{g}/\mathfrak{s} = \bigoplus_{i=1}^n\bigoplus_{s=1}^{v_i}\left( 
\bigoplus_{r=1,r\neq s}^{v_i}\frac{P_{ir}}{P_{is}} \right),
\]
so we may compute the contribution for each direct summand separately.

\par When $r,s\neq 1$, both $P_{ir}$ and $P_{is}$ are topologically trivial on the entire broken orbit from $A$ to $E_{\JJ}$. It is not hard to see the contribution in these cases is always 1.

\par When $r=1,s\neq 1$, $P_{is}$ still restricts to a topologically trivial bundle. However, $P_{ir} = P_{i1}$ has non-trivial restriction to the component $E_{\JJ[k_{\JJ}[i]-1]}E_{\JJ[k_{\JJ}[i]]}$, and it pulls back to $\mathcal{O}(-m)$ of the domain curve $\mathbb{C}P^{1}$ on this component, with $\widetilde{T}$-characters on endpoints $P_{i1}|_{E_{\JJ[k_{\JJ}[i]-1]}} = P_{i1}|_A$, $P_{i1}|_{E_{\JJ[k_{\JJ}[i]]}} = P_{i1}|_A\cdot(\lambda_{j_{k_{\JJ}[i]}})^{-m/m}$. It then follows from direct computation that the contribution in this case is
\[
\left.\prod_{l=0}^{-m}\left(1-y\frac{P_{is}|_A}{P_{i1}|_A}(\lambda_{j_{k_{\JJ}[i]}})^{\frac{l}{m}}\right) \middle/ \left(1-y\frac{P_{is}|_A}{P_{i1}|_A}\right)\right. .
\]
Note that the product $\prod_{l=0}^{-m}$ in the numerator really means the product $\prod_{l=-m+1}^{-1}$ in the denominator. The contribution above is well-defined under the specialization $\wT\rightarrow T$ and $y=1$, and is restriction of a global bundle on $\M_A(\Lambda_{v_1+1}/\Lambda_1)$ by Corollary \ref{tautocor}. 

\par Similarly when $s=1,r\neq 1$, the contribution is
\[
\prod_{l=1}^m \left(1-y\frac{P_{i1}|_A}{P_{ir}|_A}(\lambda_{j_{k_{\JJ}[i]}})^{\frac{l}{m}}\right).
\]
Once again, it comes from the restriction of a global bundle well-defined as $\widetilde{T}\rightarrow T$ and $y=1$. 

\par More importantly, uniformly for all orbits ${\mathfrak{J}}$, when $i=n,r=v_1+1,s=1$ (which always happens geometrically on the last irreducible component of the broken orbit), the $l=m$ term of the above product gives
\[
1-y \frac{P_{n1}|_A}{P_{nr}|_A}(\lambda_{j_{k_{\JJ}[n]}})^{\frac{m}{m}} = 1 - y \cdot \frac{\Lambda_{n+1,1}}{\Lambda_{n+1,v_1+1}} \cdot \left(\frac{\Lambda_{n+1,v_1+1}}{\Lambda_{n+1,1}}\right)^{\frac{m}{m}} = 1-y.
\]
This is exactly the \textbf{vanishing term} $(1-y)$ that we need in Proposition \ref{prop8}.

\par Note that we required $r\neq s$ in the twisting bundle, so there is no such case for $r=s=1$. Now we have finished the discussion on the twisting bundle.

\subsubsection*{Contribution from $TY$}
\par The contribution of $TY$ takes the form 
\[
\prod_{k=0}^{|\JJ|-1} \frac{\Eu(TY)|_{E_{\JJ[k]}}}{\Eu((TY)_{0,2,m\cdot D_{E_{\JJ[k]}E_{\JJ[k+1]}}} - 1)|_{\phi_{E_{\JJ[k]}E_{\JJ[k+1]}}}}.
\]
Recall that
\[
TY = \bigoplus_{i=1}^n\bigoplus_{s=1}^{v_i}\left(\bigoplus_{r=1}^{v_{i+1}}\frac{P_{i+1,r}\Lambda_{i+1,r}}{P_{is}}\right) \ominus \bigoplus_{i=1}^n\bigoplus_{s=1}^{v_i} 1,
\]
It is not hard to see that the $\oplus_{i,s} 1$ part has trivial total contribution. Moreover, we may discuss the remaining part by analyzing the direct summands separately, and an argument similar to that of the twisting bundle above holds almost verbatim. We omit most of the discussion here, and write out the computation only for the following two cases, for which extra attention is required.

\par One case is when $r=s=1$ and $i\leq n-1$. We still apply the same method here. However, depending on whether $i\in {\mathfrak{J}}$ or $i\notin {\mathfrak{J}}$, $P_{i1}$ and $P_{i+1,1}$ may have nontrivial restriction on two consecutive irreducible components of the broken orbit (as $k_{\JJ}[i+1] = k_{\JJ}[i] + 1$ when $i\in\JJ$), or on the same component. Fortunately, both cases give the same final result
\[
\frac{1}{\prod_{l=1}^{m-1}(1-\lambda_{j_{k_{\JJ}[i]}}^{\frac{l}{m}})\prod_{l=0}^{-m}(1-\lambda_{j_{k_{\JJ}[i]}}^{\frac{m}{m}}\lambda_{j_{k_{\JJ}[i+1]}}^{\frac{l}{m}})}.
\]
It comes from a well-defined global bundle on $\M_A(\Lambda_{v_1+1}/\Lambda_1)$.

\par The other case is when $s = 1,r = v_1+1$ and $i\leq n$. By direct computation, the contribution is
\[
\frac{1}{\prod_{l=1}^m(1-\frac{P_{i1}|_A}{P_{i+1,v_1+1}|_A\Lambda_{i+1,v_1+1}}(\lambda_{j_{k_{\JJ}[i]}})^{\frac{l}{m}})} = \frac{1}{\prod_{l=1}^m (1 - \lambda_{i}^{-1}(\lambda_{j_{k_{\JJ}[i]}})^{\frac{l}{m}})}.
\]
The $l=m$ term of the denominator will vanish under the specialization $\widetilde{T}\rightarrow T$. However, this is exactly what we need:
\begin{itemize}
    \item For $i\in {\mathfrak{J}}$, the $l=m$ term in the denominator gives $(1-1)$. Such terms appear $|\JJ|$ times, and thus cancel completely with the $\Eu(-1)$ terms in the original formula of recursive coefficients which come from the 1-dim automorphisms of double-marked $\C P^1$'s, and appear also $|\JJ|$ times as the broken orbit to $E_{\JJ}$ has $|\JJ|$ irreducible components.
    \item For $i\notin {\mathfrak{J}}$, the product of all such $l=m$ terms give
    \[
    \prod_{i\notin \JJ}\frac{1}{1 - \lambda_{i}^{-1} \lambda_{j_{k_{\JJ}[i]}}} = \sum_{1\leq \alpha_i \leq m, \forall i\notin\JJ} \frac{1}{m^{n-|\JJ|}} \prod_{i\notin \JJ} \frac{1}{1-\zeta^{-\alpha_i}\cdot \lambda_{i}^{-\frac{1}{m}}\lambda_{j_{k_{\JJ}[i]}}^{\frac{1}{m}}},
    \]
    where $\zeta = e^{2\pi i/m}$. The sum over $\alpha_i$ is really over different choices of $m$-th roots of the relevant $\wT$-characters. According to what we have discussed in Section \ref{specialcasen2}, the sum over such choices is always implied, and we will from now assume a fixed choice of $m$-th root for each $\lambda_{i}\ (i\notin \JJ)$ and will disregard the sum over $\alpha_i$. Indeed, the choices of $m$-th root for those $\lambda_{i} \ (i\in\JJ)$ have been implicitly made when we wrote out the formula for $S_{\JJ}$. Hence, we may safely focus on one specific summand 
    \[
    \frac{1}{m^{n-|\JJ|}} \prod_{i\notin \JJ} \frac{1}{1-\lambda_{i}^{-\frac{1}{m}}\lambda_{j_{k_{\JJ}[i]}}^{\frac{1}{m}}},
    \]
    and study its contribution to $S_{\JJ}$. The term $1/m^{n-|\JJ|}$ pairs with the term $1/m^{|\JJ|}$ that appears in the definition of $S_{\JJ}$ to give a globally defined expression independent of $\JJ$, and the product over $i\notin \JJ$ provides exactly the part of $\Eu(T_{E_{\JJ}}\M_A(\Lambda_{v_1+1}/\Lambda_1))$ coming from $V_2$, by Lemma \ref{tangentlem}.
\end{itemize} 
So far we have finished the discussion on the contribution from $TY$.

\subsubsection*{Other terms}
\par Finally, we look at the remaining terms
\[
\prod_{k=1}^{|\JJ| - 1} \frac{1}{1-\lambda_{j_k}^{\frac{1}{m}}\lambda_{j_{k+1}}^{-\frac{1}{m}}} \cdot \frac{1}{1-q\lambda_{j_1}^{-\frac{1}{m}}} \cdot \frac{\prod_{i=1}^nQ_{i1}^m}{m^{|\JJ|}} \cdot J^{tw,Y}_{E_{\JJ}}\left(\lambda_{n}^{\frac{1}{m}}\right).
\]
It is not hard to see that the first factor above gives exactly the part of $\Eu(T_{E_{\JJ}}\M_A(\Lambda_{v_1+1}/\Lambda_1))$ coming from $V_1$ (the part coming from $V_2$ have appeared as contribution from $TY$ above), while all other factors come from the restriction of a global expression over $\M_A(\Lambda_{v_1+1}/\Lambda_1)$, well defined under the specialization described in the Proposition, by arguments similar to those in Section \ref{specialcasen2}.
\begin{remark}
The same algorithm may be applied to (the abelian quotients associated to) more general GIT quotients of vector spaces. Indeed, the only information that we need is
\begin{itemize}
    \item the restriction of $TY$ and $\mathfrak{g}/\mathfrak{s}$ at the $\widetilde{T}$-fixed points on the abelian quotient $Y$;
    \item the structure of the space of balanced broken orbits with given $T$-tangent character and homological degree.
\end{itemize}
\end{remark}

\subsection{General case: Step II} \label{StepII}
\par We prove Proposition \ref{nonisoA} and \ref{nonisoC} in this section, using a re-summation argument similar to that of Section \ref{specialcasen2}. We start by noting that the argument in Section \ref{StepI} for $S_{\JJ}$ can be generalized in two directions
\begin{itemize}
    \item[(i)] We may consider balanced broken orbits starting from and ending at any $E_{\JJ}$, rather than only $A$ and $E_{\JJ}$ with $n\in \JJ$.
    \item[(ii)] We may consider any input $\mathcal{G}$ at the endpoints, regular in a neighborhood of $q = (\Lambda_{v_1+1}/\Lambda_1)^{1/m}$ under the specialization $\wT\rightarrow T, Q_{is}=Q_i, y=1$. One such example is $\overline{G}(q)$, the part of $J^{tw,Y}_{E_1}(q)$ regular at $(\Lambda_{2}/\Lambda_1)^{1/m}$, that we considered in Section \ref{specialcasen2}.
\end{itemize}
For these purposes, we generalize our definition of the recursive contribution $S_{\JJ}$ as follows. For subsets $\II, \JJ \subset \{1,2,\cdots, n\}$ with $\max \II < \min \JJ$, we define
\begin{align*}
S_{\II, \JJ}(\mathcal{G}) := & \frac{1}{1-q\lambda_{j_1}^{-\frac{1}{m}}}\cdot \prod_{k=1}^{|\JJ| - 1} \frac{1}{1-\lambda_{j_k}^{\frac{1}{m}}\lambda_{j_{k+1}}^{-\frac{1}{m}}} \cdot \frac{1}{m^{|\JJ|}} \cdot \prod_{i=1 + \max \II}^{\max \JJ} Q_{i1}^m \\
& \cdot \prod_{k=0}^{|\JJ|-1} \coeff^{tw,Y}_{E_{\II\cup\JJ[k]}E_{\II\cup\JJ[k+1]}}(m) \cdot \mathcal{G}_{E_{\II\cup\JJ}}\left(\lambda_{\max \JJ}^{\frac{1}{m}}\right),
\end{align*}
where $\JJ = \{j_1,\cdots,j_K\}$ with $j_1<\cdots<j_K = \max \JJ$, and $\JJ[k] = \{j_1,\cdots,j_k\}$ is the truncation. 

\par $S_{\II, \JJ}(\mathcal{G})$ may be understood as the recursive contribution to $\mathcal{G}$ along the unique balanced broken orbit from $E_{\II}$ to $E_{\II\cup\JJ}$. In particular, $S_{\phi,\JJ}(J^{tw,Y})$ gives back our $S_{\JJ}$ of the previous section. By definition, $S_{\II, \JJ}(\mathcal{G})$ satisfies the following \textbf{composition law}.
\begin{lemma} \label{compositionlaw}
For $\II, \JJ, \KK\subset \{1,2,\cdots,n\}$ with $\max \II < \min \JJ$ and $\max \JJ < \min \KK$, we have
\[
S_{\II,\JJ}(S_{\II\cup\JJ,\KK}(\mathcal{G})) = S_{\II,\JJ\cup\KK}(\mathcal{G}).
\]
\end{lemma}

\par Now we prove Proposition \ref{nonisoC}. Using the new language introduced above, what we have proved through consideration of $\M_A(\Lambda_{v_1+1}/\Lambda_1)$ in the previous section may be re-written as
\[
\sum_{\JJ, \max\JJ = n} S_{\phi,\JJ}(J^{tw,Y}) = 0
\]
under the specialization described in Proposition \ref{nonisoA}. More generally, we may consider for $k\geq 2$ the space $\M_{C,k}(\Lambda_{v_1+1}/\Lambda_{1})$ of balanced broken orbits starting from $C = E_1$ with tangent character $\Lambda_{v_1+1}/\Lambda_{1}$ and homological degree $\sum_{i=2}^k\one_{i,1}$. The endpoints of such orbits always take the form $E_{\{1\}\cup\JJ}$ where $\min\JJ > 1$ and $\max\JJ = k$. Applying the method of the previous section to $\M_{C,k}(\Lambda_{v_1+1}/\Lambda_{v_1})$, we have
\[
T_k := \sum_{\JJ,\min\JJ>1,\max\JJ = k} S_{\{1\},\JJ}(\mathcal{G}^k) = 0,
\]
under the specialization described in Proposition \ref{nonisoC}, where 
\[
\mathcal{G}^k_{E_{\{1\}\cup\JJ}}(q) = J^{tw,Y}_{E_{\{1\}\cup\JJ}}(q) - \sum_{l=k+1}^n S_{\{1\}\cup\JJ, \{l\}}(J^{tw,Y}).
\]
Here $\mathcal{G}^k_{E_{\{1\}\cup\JJ}}(q)$, an analogue of $\overline{G}(q)$ as in Section \ref{specialcasen2}, is exactly the regular part of $J^{tw,Y}_{E_{\{1\}\cup\JJ}}$, free of poles that would descend to $q = (\Lambda_{v_1+1}/\Lambda_1)^{1/m}$. It is not hard to see that they come from a well-defined global expression $\mathcal{G}^k$ over $\M_{C,k}(\Lambda_{v_1+1}/\Lambda_{v_1})$, as the recursion coefficients depend essentially only on the starting points. It is worth noting that $T_k$ would not have had a well-defined limit as $\wT\rightarrow T$ if we had taken $J^{tw,Y}$ directly as the endpoint input instead of $\mathcal{G}^k$. 

\par Taking the sum of $T_k$ over $2\leq k\leq n$, we will see that only the terms where $|\JJ| = 1$ persist, and thus
\[
\sum_{k=2}^n S_{\{1\}, \{k\}}(J^{tw,Y}) = 0
\]
under the specialization described in Proposition \ref{nonisoC}. Indeed, if $\JJ$ contains at least two elements (and we assume the two largest elements are $j_{-2}$ and $j_{-1}$ respectively with $j_{-2}<j_{-1} = \max\JJ$), then the term $S_{\{1\}, \JJ}$ appears exactly twice in $\sum_{k=2}^n T_k$:
\begin{itemize}
    \item once directly in $T_{j_{-1}}$ with $+$ sign;
    \item once in $T_{j_{-2}}$ by Lemma \ref{compositionlaw} (composition law) with $-$ sign. 
\end{itemize}
The above formula gives us exactly Proposition \ref{nonisoC}. In other words, $\sum_{k} T_k$ provides the correct re-summation of $\sum_k S_{\{1\}, \{k\}}(J^{tw,Y})$ where every summand is well-defined and approaches zero when we take the specialization in Proposition \ref{nonisoC}.

\par Proposition \ref{nonisoA} may be proved following the same idea, where the re-summation gives us
\[
\sum_{k=1}^n S_{\phi,\{k\}}(J^{tw,Y}) = S_{\phi,\{1\}}(\mathcal{G}^1) = \frac{-1}{1-q\lambda_{1}^{-1/m}} \coeff^{tw,Y}_{AC}(m) \mathcal{G}^1_{C}(\lambda_1^{1/m}),
\]
where $\mathcal{G}^1_{C}(q)$ is by definition the regular part of $J^{tw,Y}_C(q)$. By Proposition \ref{nonisoC} which we have just proved, no genuine poles arise in $J^X_C(q)$ from $J^{tw,Y}_C(q)$ under the specialization $\wT\rightarrow T, Q_{is} = Q_i, y=1$, which means that $\mathcal{G}^1_{C}(q)$ descends correctly to $J^X_C(q)$. Hence follows Proposition \ref{nonisoA}.

\subsection{Final conclusions}
We reiterate the two main statements proved in this section as the end. {}
\begin{itemize}
    \item $J^{X}_C$, which is descended from $J^{tw,Y}_C$ as $\widetilde{T}\rightarrow T$, $y=1$ and $Q_{is}=Q_i$, does not have poles at $q = (\Lambda_{v_1+1}/\Lambda_1)^{1/m}$ (although $J^{tw,Y}_C$ has indeed poles at $q = (\lambda_{i,v_1+1}/\lambda_{i,1})^{1/m}$ for $i=2,3,\ldots,n$ under $\widetilde{T}$-action).
    \item $J^{X}_A$, which is descended from $J^{tw,Y}_A$, has only simple pole at $q = (\Lambda_{v_1+1}/\Lambda_1)^{1/m}$, and its residue satisfies the recursive formula
    \[
    \Res_{q = \left(\frac{\Lambda_{v_1+1}}{\Lambda_1}\right)^{\frac{1}{m}}} J^{X}_A(q)\frac{dq}{q} = \frac{\prod_{i=1}^n Q_{i}^{m}}{m}\frac{\Eu(T_AX)}{\Eu(T_{\psi} X_{0,2,mD})}\cdot J^{X}_C\left(\left(\frac{\Lambda_{v_1+1}}{\Lambda_1}\right)^{\frac{1}{m}}\right),
    \]
    where $\psi:\mathbb{C}P^1\rightarrow X$ is the $m$-sheet covering stable map to the 1-dim $T$-invariant orbit $AC$ with tangent $T$-character at $A$ being $\Lambda_{v_1+1}/\Lambda_1$, and $D$ is the homological degree of $AC$.
\end{itemize}
\par The cases of fixed points other than $A$ and tangent characters other than $\Lambda_{v_1+1}/\Lambda_1$ are similar. We have thus proved Theorem \ref{2} (and Corollary \ref{MainCor}, when combined with the following remark).

\subsubsection*{Remarks on $J^X$ being the small $J$-function of $X$}
\par As we mentioned in the Introduction, $J^X$ represents actually the value of $\mathcal{J}(\t;q)$ at $\t = 0$. We prove it by showing that $\mathcal{K}_+$ component of $J^X$ is indeed $1-q$. Since we already know that the $Q^0$-term in $J^X$ is exactly $1-q$, it suffices to show that the $q$-rational function attached to any $Q^{d}$ with $d$ nontrivial has strictly higher $q$-degree in the denominator than in the numerator. Recall that 
\[
J^X = (1-q)\sum_{d\in \mathcal{D}} \prod_{i=1}^n Q_i^{\sum_s d_{is}}\frac{\prod_{i=1}^n\prod_{r\neq s}^{1\leq r,s\leq v_i}\prod_{l=1}^{d_{is}-d_{ir}}(1-\frac{P_{is}}{P_{ir}}q^l)}{\prod_{i=1}^{n-1}\prod_{1\leq r\leq v_{i+1}}^{1\leq s\leq v_i}\prod_{l=1}^{d_{is}-d_{i+1,r}}(1-\frac{P_{is}}{P_{i+1,r}}q^l)\cdot\prod_{1\leq r\leq N}^{1\leq s\leq v_n}\prod_{l=1}^{d_{ns}}(1-\frac{P_{ns}}{\Lambda_{r}}q^l)}.
\]
All terms in either the numerator or the denominator have the form $(1-x\cdot q^l)$, and will contribute by degree $l$ when $l>0$ and by degree $0$ when $l\leq0$. Since $J^X = \sum_{a\in \mathcal{F}}J^X_a\phi^a$, it suffices to prove the statement for all localization coefficients $J^X_a$. Without loss of generality, we consider the localization to the distinguished fixed point $A$, in which case we have
\[
J^X_A = J^X|_{P_{is} = \Lambda_s} = (1-q)\sum_{d\in \mathcal{D}} \prod_{i=1}^n Q_i^{\sum_s d_{is}}\frac{\prod_{i=1}^n\prod_{r\neq s}^{1\leq r,s\leq v_i}\prod_{l=1}^{d_{is}-d_{ir}}(1-\frac{\Lambda_s}{\Lambda_r}q^l)}{\prod_{i=1}^{n}\prod_{1\leq r\leq v_{i+1}}^{1\leq s\leq v_i}\prod_{l=1}^{d_{is}-d_{i+1,r}}(1-\frac{\Lambda_s}{\Lambda_r}q^l)}
\]
where $d_{n+1,r} = 0$ and $v_{n+1}=N$ by convention. The difference of $q$-degrees between the denominator and the numerator is
\[
S = \sum_{i=1}^n\sum_{d_{is}>d_{i+1,r}}^{1\leq s\leq v_i,1\leq r\leq v_{i+1}}\binom{d_{is}-d_{i+1,r}+1}{2} - \sum_{i=1}^n\sum_{d_{is}>d_{ir}}^{1\leq r,s\leq v_i}\binom{d_{is}-d_{ir}+1}{2} - 1.
\]
The $-1$ comes from the extra factor $(1-q)$. We want to show that $S\geq 1$ whenever $d \in \mathcal{D}$ is nontrivial, i.e. whenever at least one of $d_{is}\ (i=1,\cdots,n; s = 1,\cdots,v_i)$ is strictly positive. We will prove the statement by induction on $n$. When $n=1$, this is exactly the grassmannian case and the conclusion is already proved in \cite{Givental-Yan}. Now for larger $n$, if $d_{ns}=0$ for all $s$, then only $i\leq n-1$ terms have non-zero contribution to $S$, so $S\geq 1$ by induction hypothesis of $n-1$. If at least one of $d_{ns}$ is positive, then the $i=n$ terms in $S$ contribute by:
\[
\sum_{d_{ns}>0}^{1\leq s\leq v_n,1\leq r\leq N}\binom{d_{ns}+1}{2} - \sum_{d_{ns}>d_{nr}}^{1\leq r,s\leq v_n}\binom{d_{ns}-d_{nr}+1}{2} - 1,
\] 
which, again by $n=1$ case, is great than or equal to $1$. Then it suffices to show the remaining terms have non-negative contribution. In fact, we may assume $d_{ir}\geq d_{i+1,r}$ for all legitimate values of $i$ and $r$, since otherwise the group of terms $\prod_{l=1}^{d_{is}-d_{i+1,r}}(1-\frac{P_{is}}{P_{i+1,r}}q^l)$ in the denominator with $r=s$ would move to the numerator and give $\prod_{l=d_{ir}-d_{i+1,r}+1}^{0}(1-\frac{\Lambda_r}{\Lambda_r}q^l)$, which contains a factor $1-q^0 = 0$. Hence, $d_{is}>d_{i+1,r}$ whenever $d_{is}>d_{ir}$, which means
\begin{align*}
& \sum_{i=1}^{n-1}\sum_{d_{is}>d_{i+1,r}}^{1\leq s\leq v_i,1\leq r\leq v_{i+1}}\binom{d_{is}-d_{i+1,r}+1}{2} - \sum_{i=1}^{n-1}\sum_{d_{is}>d_{ir}}^{1\leq r,s\leq v_i}\binom{d_{is}-d_{ir}+1}{2} \\
\geq & \sum_{i=1}^{n-1} \sum_{1\leq r,s\leq v_i,d_{is}>d_{ir}} \left[\binom{d_{is}-d_{i+1,r}+1}{2} - \binom{d_{is}-d_{ir}+1}{2}\right] + \sum_{i=1}^{n-1} \sum_{s=1}^{v_i} \binom{d_{is}-d_{i,s+1}+1}{2} \\
\geq & 0.
\end{align*}
This completes the induction argument.

\section{Explicit reconstruction} \label{recon}
\par In this section, we carry out the second half of our proving strategy for Theorem \ref{0}. The enlarged torus action by $\wT$ is no longer necessary, so we focus back onto $\widetilde{J}^{tw,Y}$, the limit of $J^{tw,Y}$ under $\wT\rightarrow T$ (i.e. by taking $\Lambda_{is} = 1$ for $i\leq n$ and $\Lambda_{n+1,s} = \Lambda_s$). We prove that in $T$-equivariant settings, its $\mathcal{P}^W$-orbit recovers the entire $\L^X$ under the specialization $Q_{is} = Q_i$ and $y=1$, when passing from $\L^{tw,Y}$ to $X$. In fact, not even $\widetilde{J}^{tw,Y}$ is crucial here: it may be replaced by any point on $\L^{tw,Y}$ invariant under the action by Weyl group $W$, and we will prove the reconstruction theorem below under these more general settings. 

\par In order to facilitate the proof below, we consider for big $\J$-functions of $Y$ a larger coefficient ring $\Lambda^Y = R[y^{\pm 1},\lambda^{\pm 1}][[Q]]$, where $\lambda = (\Lambda_s)_{s=1}^N$, $Q = (Q_{is})_{i=1,s=1}^{\ n\quad v_i}$, and $R$ is a $\C$-algebra. We require the Adams operations on $\Lambda^Y$ to satisfy $\Psi^k(Q_{is}) = Q_{is}^k$, $\Psi^k(\Lambda_{s}^{\pm 1}) = \Lambda_{s}^{\pm k}$ and $\Psi^k(y^{\pm 1}) = y^{\pm k}$. We assume that there exists an ideal $\Lambda^Y_+\subset\Lambda^Y$ stable under the Adams operation and containing the Novikov variables $Q$, and that $\Lambda^Y$ is complete with respect to the adic topology associated to $\Lambda^Y_+$. Taking the specialization $Q_{is} = Q_i$ and $y=1$, we obtain the coefficient ring $\Lambda^X$ with ideal $\Lambda^X_+$ for big $\J$-functions of $X$, which will not really appear very often in this section. We remove the superscripts $X$ and $Y$ when it is clear by context.

\par For simplicity of notations, we denote by $K(Y)$ the $T$-equivariant K-ring of $Y$, and take an additive basis $\{P^a\}_{a\in A}$ of $K(Y)$ over $R(T) = \C[\lambda^{\pm 1}]$ consisting of monomials of ($T$-equivariant) tautological bundles $P = (P_{is})_{i=1,s=1}^{\ n\quad v_i}$ of $Y$. Here $A\subset \prod_{i=1}^n\mathbb{Z}^{v_i}$ is a finite set of integer-valued vectors $a = (a_{is})_{i=1,s=1}^{\ n\quad v_i}$, and the existence (and finiteness) of such $A$ follows from the structure of $Y$. $K(Y)$ has natural Adams operations $\Psi^k(P_{is}^{\pm 1}) = P_{is}^{\pm 1}$ and $\Psi^k(\Lambda_{s}^{\pm 1}) = \Lambda_{s}^{\pm k}$.

\par We denote by $\Lambda^Y\otimes K(Y)$ the tensor product over $R(T)$. It admits Adams operations $\Psi^k$ inherited from those on the two factors, and is naturally endowed with an action by the Weyl group $W = \prod_{i=1}^n S_{v_i}$, where $\sigma = (\sigma_i)_{i=1}^n \in W$ acts by $\sigma(P_{is}) = P_{i,\sigma_i(s)}$, $\sigma(Q_{is}) = Q_{i,\sigma_i(s)}$ and trivially on $y$ and $\lambda = (\Lambda_{s})_{s=1}^N$. The Weyl-group-invariant part $(\Lambda^Y\otimes K(Y))^W$, under $Q_{is} = Q_i$ and $y=1$, descends surjectively to $\Lambda^X\otimes K(X)$.

\par The main theorem of this section is formulated as below.
\begin{theorem} \label{Reconstruction}
\par Let $I = \sum_{d\in \mathcal{D}} I_d Q^d$ be any $W$-invariant point on $\mathcal{L}^{tw,Y}$ that is $\Lambda^Y_+$-close to $(1-q)$. Then, the family 
\[
\mathcal{F} = \bigcup_{\epsilon_a\in\Lambda^Y_+, C_a(q)\in\Lambda^Y_+[q,q^{-1}]} \ 
\sum_{d\in \mathcal{D}} 
I_d Q^d\cdot \exp{\sum_{k>0} \sum_{a\in A} \frac{\Psi^k(\epsilon_a)P^{ka}q^{k\langle a,d\rangle}}{k(1-q^k)}} \cdot \exp{\sum_{a\in A} C_a(q)P^aq^{\langle a,d\rangle}},
\]
where $Q^d=\prod_{i,s}Q_{is}^{d_{is}}$ and $P^a q^{\langle a,d\rangle} = \prod_{i,s} P_{is}^{a_{is}}q^{a_{is}d_{is}}$, lies entirely in $\mathcal{L}^{tw,Y}$. Moreover, when
\[
\epsilon = \sum_{a\in A}\epsilon_a P^a \in \Lambda^Y\otimes K(Y) \quad\text{ and }\quad C(q) = \sum_{a\in A} C_a(q) P^a \in \Lambda^Y\otimes K(Y)\ [q,q^{-1}]
\]
are both $W$-invariant, the above family lies in $(\mathcal{L}^{tw,Y})^W$, and covers the entire ($\Lambda^X_+$-small germ at $(1-q)$ of) $\mathcal{L}^X$ under the specialization $y = 1$ and $Q_{is} = Q_i\ (\forall 1\leq i\leq n, 1\leq s\leq v_i)$.
\end{theorem}
Geometrically, the $\epsilon$-part sweeps over different ruling spaces of $\mathcal{L}^X$, and the $C(q)$-part exhausts each fixed ruling space. Taking $I = \widetilde{J}^{tw,Y}$ as the starting point recovers Theorem \ref{0}. 

\par The $W$-invariance of $\epsilon$ and $C(q)$ can be translated as follows. Let $A^{\leq} = \{a=(a_{js})\in A|a_{j,s_1}\leq a_{j,s_2} \text{ if } s_1<s_2\}$ be the subset consisting of ``increasing'' indices $a$. Then $\{P^a\}_{a\in A^{\leq}}$ as well as their images under $W$-action form an additive basis of $K(Y)$, and any $W$-invariant $C(q)$ can thus be written as the symmetrization
\[
C(q) = \sum_{\sigma \in W}\left( \sum_{k\in\mathbb{Z}} \sum_{a\in A^{\leq}} \sum_{\delta \in \mathcal{D}} C_{k,a,\delta} \cdot q^k \cdot \prod_{i,s}Q_{i,\sigma_i(s)}^{\delta_{is}} \cdot \prod_{i,s}P_{i,\sigma_i(s)}^{a_{is}}\right)
\]
for $C_{k,a,\delta} \in R$. Similarly, any $W$-invariant $\epsilon$ can be written as the symmetrization
\[
\epsilon = \sum_{\sigma \in W}\left( \sum_{a\in A^{\leq}}\sum_{\delta\in \mathcal{D}} \epsilon_{a,\delta} \cdot \prod_{i,s}Q_{i,\sigma_i(s)}^{\delta_{is}} \cdot \prod_{i,s}P_{i,\sigma_i(s)}^{a_{is}}\right)
\]
for $\epsilon_{a,\delta} \in R$. The $\Lambda^Y_+$-small condition on $\epsilon$ and $C(q)$ requires that $C_{k,a,\mathbf{0}}\in\Lambda^Y_+\cap R$ and $\epsilon_{a,\mathbf{0}}\in\Lambda^Y_+\cap R$.

\par That the entire family lies in $\mathcal{L}^{tw,Y}$ follows from Lemma \ref{pfdpreserve} through an argument similar to that appearing in \cite{Givental:perm7}. Applying to $I = \sum_{d\in \mathcal{D}} I_d\prod_{j,s}Q_{js}^{d_{js}}$ the PFD operators $O_1$ and $O_2$, where
\[
O_1 = \exp{\sum_{k>0}\sum_a\frac{\Psi^k(\epsilon'_a)P^{ka}q^{kaQ\partial_Q}}{k(1-q^k)}}, \quad O_2 = \exp{\sum_a C_a(q)P^aq^{aQ\partial_Q}}.
\]
with $P^aq^{aQ\partial_Q} = \prod_{i,s}(P_{is}q^{Q_{is}\partial_{Q_{is}}})^{a_{is}}$, we obtain the family $\mathcal{F}$ in the theorem but with each $\epsilon_a = \epsilon_a(Q)$ replaced with the independent variables $\epsilon'_a$. By Lemma \ref{pfdpreserve}, such new family $\mathcal{F}'$ obtained from $O_1\circ O_2(I)$ lies entirely in the cone $\mathcal{L}^{tw,Y}$, but only with the \textbf{extended coefficient ring} $(\Lambda^Y)' = \Lambda[\epsilon']$ where $\epsilon' = (\epsilon'_a)_{a\in A}$. Then, we obtain $\mathcal{F}$ back by restricting the coefficients back to $\Lambda^Y$ through the map sending $\epsilon'_a$ to $\epsilon_a$. Such restriction of coefficients $(\Lambda^Y)'\rightarrow \Lambda^Y$ is legitimate here because it is a ring homomorphism which preserves also the Adams operations. Moreover, when $\epsilon$ and $C(q)$ are both $W$-invariant, the resulting family resides entirely in $(\mathcal{L}^{tw,Y})^W$, because the restriction of coefficients is naturally $W$-equivariant as well. The above argument is necessary because $\epsilon_a$ depends \emph{a priori} on the Novikov variables $Q$, and thus may be acted upon non-trivially by the $q$-difference operators $P^aq^{aQ\partial_Q}$.

\par It remains only to prove that $\mathcal{F}$ covers the entire $\mathcal{L}^X$ under the specialization of coefficients $y=1$ and $Q_{is}=Q_i$. In fact, such surjectivity may be further reduced to surjectivity to $(\mathcal{L}^{tw,Y})^W$ (before the specialization), as it is not hard to see that the specialization $y=1$ and $Q_{is}=Q_i$ gives by definition a surjective map in between the domains of big $\J$-functions
\[
(\mathcal{K}_+^Y)^W = \left(\Lambda^Y\otimes K(Y)\right)^W[q,q^{-1}] \ \longrightarrow \ \mathcal{K}_+^X = \Lambda^X\otimes K(X)\ [q,q^{-1}],
\]
where $P_{is}$ are sent to K-theoretic Chern roots of $V_i$ for any $1\leq i\leq n$ and $1\leq v_i\leq n$, and thus a surjective map from the image $(\mathcal{L}^{tw,Y})^W = \mathcal{J}^{tw,Y}((\mathcal{K}_+^Y)^W)$ to $\L^X = \J^X(\mathcal{K}_+^X)$. 

\par In other words, it suffices now to prove that for any $W$-invariant input $\t = \t(q) \in (\mathcal{K}_+^Y)^W$, $\mathcal{J}^{tw,Y}(\t)$ can indeed be represented by some point of $\mathcal{F}$. We will first prove the claim \textit{modulo the Novikov variables} $Q$. As $Q_{is} = 0$, the family becomes 
\[
\overline{\mathcal{F}} = I_0\cdot \exp{\left(\sum_{k>0}\sum_{a\in A}\frac{\Psi^k(\epsilon_{a,\mathbf{0}})P^{ak}}{k(1-q^k)}\right)}\cdot \exp{\left(\sum_{k\in\mathbb{Z}}\sum_{a\in A} C_{k,a,\mathbf{0}} P^a \cdot q^k\right)}.
\]
The $W$-action now involves only $P_{is}$. Since $I = \sum_d I_dQ^d$ is $W$-invariant, so is $I_0$; since $\epsilon$ and $C(q)$ are $W$-invariant, $\epsilon_{a,\mathbf{0}}$ and $C_{k,a,\mathbf{0}}$ are invariant under $W$-action on $a$. Therefore, the above family may be re-written as
\[
\overline{\mathcal{F}} = I_0\cdot \exp{\left(\sum_{k>0}\sum_{a\in A^{\leq}}\frac{\Psi^k(\epsilon_{a,\mathbf{0}}\sum_{\sigma\in W}P^{\sigma(a)})}{k(1-q^k)}\right)}\cdot \exp{\left(\sum_{k\in\mathbb{Z}}\sum_{a\in A^{\leq}} (C_{k,a,\mathbf{0}} \sum_{\sigma\in W}P^{\sigma(a)}) \cdot q^k\right)}.
\]
This is now actually a question in the quantum K-theory of a point: does $\overline{\mathcal{F}}$ cover (the germ at $(1-q)$ of) $\L^{pt}$ under coefficient ring $\overline{\Lambda} = \left(\Lambda^Y/(Q)\otimes K(Y)\right)^W = \Lambda^Y/(Q)\otimes K^W(Y)$, as $\epsilon_{a,\mathbf{0}}$ and $C_{k,a,\mathbf{0}}$ vary over $\Lambda^Y_+\cap R$? By deduction in \cite{Givental:perm7}, the $\overline{\Lambda}_+$-germ at $(1-q)$ of $\mathcal{L}^{pt}$ is covered by the ruling spaces parameterized by $t$
\[
\mathcal{L}^{pt} = \bigcup_{t\in\overline{\Lambda}_+} (1-q) \cdot \exp{\left(\sum_{k>0}\Psi^k(t)/k(1-q^k)\right)} \cdot \left(1 + \overline{\Lambda}_+\cdot \mathcal{K}_+\right),
\]
where $\overline{\Lambda}_+ = \Lambda^Y_+/(Q)\otimes K^W(Y)$ and $\mathcal{K}_+$ is the Lagrangian subspace consisting of Laurent polynomials in $q$ as appearing in the usual loop space formalism. Let 
\[
I_0 = (1-q) \cdot \exp{\left(\sum_{k>0}\Psi^k(t^0)/k(1-q^k)\right)} \cdot (1+f^0(q)),
\]
where $t^0\in\overline{\Lambda}_+$ and $f^0(q)\in \overline{\Lambda}_+[q,q^{-1}]$. Then, for any other point 
\[
(1-q) \cdot \exp{\left(\sum_{k>0}\Psi^k(t)/k(1-q^k)\right)} \cdot (1+f(q))
\]
on $\mathcal{L}^{pt}$, it is covered by $\overline{\mathcal{F}}$ as we may simply take
\begin{align*}
\epsilon_{a,\mathbf{0}} \cdot \sum_{\sigma\in W}P^{\sigma(a)} & = t - t^0 \quad \in \overline{\Lambda}_+ = \Lambda^Y_+/(Q)\otimes K^W(Y), \\
\sum_{k\in\mathbb{Z}}\sum_{a\in A^{\leq}} (C_{k,a,\mathbf{0}} \sum_{\sigma\in W}P^{\sigma(a)}) \cdot q^k & = \ln \left( (1+f^0(q))^{-1}\cdot (1 + f(q))\right) \quad \in 1 + \overline{\Lambda}_+\cdot \mathcal{K}_+.
\end{align*}
Here the inverse of $1+f^0(q)$ exists in $1 + \overline{\Lambda}_+\cdot \mathcal{K}_+$ thanks to our assumption that $\Lambda^Y$ is complete with respect to the adic topology determined by $\Lambda^Y_+$, which means that we may use the formal Taylor expansion $(1+x)^{-1} = \sum_{n\geq 0} (-x)^n$; for the same reason, $\ln(1+x) = \sum_{n\geq 1} (-x)^{n+1}/n$ is well-defined on $1 + \overline{\Lambda}_+\cdot \mathcal{K}_+$.
\begin{remark} \label{uniqueruling}
It is implied from the above argument that the ruling spaces of $\L^{pt}$ do not intersect. In other words, the decomposition of $I\in\L^{pt}$ in terms of $t$ and $f(q)$ is always unique. Indeed, if a point lies in such intersection, then there exists $t^1\neq t^2 \in \overline{\Lambda}_+$ and $f^1(q), f^2(q)\in \overline{\Lambda}_+[q,q^{-1}]$ such that
\[
\exp{\left(\sum_{k>0}\Psi^k(t^1-t^2)/k(1-q^k)\right)} = (1+f^1(q))^{-1}\cdot (1 + f^2(q)).
\]
Apparently, LHS lives in $1 + \mathcal{K}_-$ while RHS lives in  $1 + \mathcal{K}_+$ (see Section \ref{Jfunc}), but $\mathcal{K}_-$ and $\mathcal{K}_+$ have only zero intersection.
\end{remark}

\par The case involving the Novikov variables is not substantially different. In fact, the surjectivity modulo $Q$ gives naturally the surjectivity of tangent maps at all points (modulo $Q$), and a formal \textbf{implicit function theorem} would then imply the surjectivity with $Q$. Since we are interested in $W$-invariant points on the image cone, we need an equivariant version of implicit function theorem.
\begin{lemma}[$W$-Equivariant Formal Implicit Function Theorem] \label{formalIFT}
Let $F = F(x,y)$ be a formal function (i.e. power series in terms of the inputs) satisfying $F(0,0) = 0$. 
\begin{itemize}
    \item If the tangent map $A = \frac{\partial F}{\partial y}|_{(0,0)}$ is invertible, then there exists a unique formal function $y=y(x)$ such that $y(0) = 0$ and $F(x,y(x))=0$. 
    \item If further $W$ acts linearly on both the domain and the co-domain, and $F$ is $W$-equivariant in the sense that $w\circ F(w^{-1}x,w^{-1}y) = F(x,y), \ \forall w\in W$, then $y=y(x)$ is also $W$-equivariant.
\end{itemize} 
\end{lemma}
\begin{proof}
The proof of this lemma is rather direct. Let $G(x,y) = y - A^{-1}F(x,y)$, then the lemma is equivalent to find a function $y = y(x)$ satisfying $y(x) = G(x,y(x))$. We solve for $y$ by iteration. Let $y_0(x) = 0$ and $y_{k+1}(x) = G(x,y_k(x))$, then it can be shown that 
\[
y_{k+1}(x)\equiv y_{k}(x)\ (\mod\ x^{k+1})\quad \text{ as long as }\quad\quad y_{k}(x)\equiv y_{k-1}(x)\ (\mod\ x^{k}).
\]
To start with, $y_1(x) = - A^{-1}F(x,0)$ does not have constant term, so $y_1(x)\equiv y_0(x)\equiv 0\ (\mod\ x)$. The limit $y(x) = y_\infty(x)$ of iteration then exists, and is the unique solution satisfying $y(0)=0$. Moreover, by definition of $y_{k+1}(x)$,
\[
w\circ y_{k+1}\circ w^{-1}(x) = w\circ y_k\circ w^{-1}(x) - w\circ A^{-1}\circ w^{-1}\circ w\circ F(y_k(w^{-1}(x)), w^{-1}x). 
\]
Thanks to the $W$-equivariance of both $F$ and $A$, $y_{k+1}(x)$ is $W$-equivariant whenever $y_k(x)$ is.
\end{proof}

\par Recall that our goal is to show that for any given $\t(q)\in (\Lambda^Y_+\otimes K(Y))^W[q,q^{-1}]$, there exists $\epsilon$ and $C(q)$ such that
\[
[\mathcal{F}]_+ = 1 - q + \t(q),
\]
where $[\cdot]_+$ means projection along $\mathcal{K} = \mathcal{K}_+ \oplus \mathcal{K}_-$ to $\mathcal{K}_+$. Now for a fixed $\t$, we take $x = Q$ and $y = (\epsilon,C(q))$, and $F(x,y) = F_{\t}(x,y) = [\mathcal{F}]_+ - (1 - q + \t(q))$. More precisely, the inputs of $F$ are the vector of Novikov variables and the vector of coefficients $\epsilon_a$ and $C_{k,a}$, as appearing in the expansion of $\epsilon$ and $C(q)$, while the outputs are the expansion coefficients $F_{a,k}(x,y)$, as appearing in $F(x,y) = \sum_{k\in\mathbb{Z},a\in A} F_{a,k}(x,y) P^a q^k$. Such formal function $F$ has coefficients in $\Lambda^Y_+/(Q)$ which depend on $I$ and $\t$. Under this interpretation, our goal is to solve for formal functions $\epsilon_a = \epsilon_a(Q)$ and $C_{a,k} = C_{a,k}(Q)$ such that $F(Q,\epsilon(Q),C(Q,q)) = 0$. Note that the solutions $\epsilon_a(Q)$ and $C_{a,k}(Q)$ live naturally in $\Lambda^Y_+/(Q)[[Q]]$, which may be identified with $\Lambda^Y_+$ by construction of $\Lambda^Y$. 

\par We attempt to use Lemma \ref{formalIFT}. First, by what we have proved in the $\mod Q$ case, there exists $\epsilon_{a,\mathbf{0}}$ and $C_{a,k,\mathbf{0}}$, both free of Novikov variables $Q$, that solves $F(0,(\epsilon_{a,\mathbf{0}})_{a\in A}, (C_{a,k,\mathbf{0}})_{a\in A,k\in\mathbb{Z}}) = 0$. Therefore, we may assume $F(0,0) = 0$ under translation of $\epsilon_a$ and $C_{a,k}$ by $\epsilon_{a,\mathbf{0}}$ and $C_{a,k,\mathbf{0}}$ respectively.

\par Second, that $\partial F/\partial y$ is invertible when $x=y=0$ may also be attributed to what we have proved in the $\mod Q$ case. In fact, it is not hard to see that the invertibility of $\partial F/\partial y|_{(0,0)}$ follows directly from the unique decomposition result in Remark \ref{uniqueruling}. Alternatively, consider the base point where $I = I_0 = 1-q$, $Q = 0$, $\epsilon=0$ and $C(q) = 0$. $\partial F/\partial y$ obtains a rather explicit expression in this case
\[
\frac{\partial F}{\partial\epsilon_a} = P^a, \frac{\partial F}{\partial C_{a,k}} = (1-q) \cdot q^k P^a.
\]
Any tangent vector $f(q)\in\Lambda_+\otimes K(Y)[q,q^{-1}]$ of the codomain thus decomposes uniquely into the effect of $\epsilon$ and $C(q)$ by
\[
f(q) = f(1) + (1-q)\cdot \frac{f(q)-f(1)}{1-q},
\]
as the first term on RHS resides in $\Lambda^Y_+/(Q)\otimes K(Y)$ while the second term in $(1-q)\cdot\Lambda^Y_+/(Q)\otimes K(Y)[q,q^{-1}]$. $\partial F/\partial y$ is therefore invertible in this special case. In reality, $\partial F/\partial y|_{(0,0)}$ describes more generally the tangent map at
\[
I_0\cdot \exp{\left(\sum_{k>0}\sum_{a\in A}\frac{\Psi^k(\epsilon_{a,\mathbf{0}})P^{ak}}{k(1-q^k)}\right)}\cdot \exp{\left(\sum_{k\in\mathbb{Z}}\sum_{a\in A} C_{k,a,\mathbf{0}} P^a \cdot q^k\right)} \in \mathcal{L}^{pt},
\]
where we still have $Q=0$ but possibly $I_0\neq 1-q$, $\epsilon\neq 0$ and $C(q)\neq 0$. According to our argument before Remark \ref{uniqueruling}, such point differs from $(1-q)$ by the multiplication of an invertible element, and thus $\partial F/\partial y|_{(0,0)}$ is still invertible in this more general case.

\par Finally, all remains to check is the $W$-equivariance of $F$. This follows from the assumption that our starting point $I = \sum_{d} I_d Q^d \in \mathcal{L}^{tw,Y}$ is $W$-invariant. For permutation $\sigma = (\sigma_1, \cdots, \sigma_n)\in W$ and index $a = (a_{is}) \in A$, define $\sigma(a)\in A$ by $\sigma(a)_{i,\sigma_i(s)} = a_{is}$. Then, 
\begin{itemize}
    \item the action $\sigma\circ F$ replaces each $P_{is}$ in the formula of the family $\mathcal{F}$ and the intended input $\t$ by $P_{i,\sigma_i(s)}$ and thus $P^a = \prod_{i,s}P_{is}^{a_{is}}$ by $P^{\sigma(a)}$;
    \item the action $F\circ \sigma^{-1}$ effectively replaces $Q_{is}$ by $Q_{i,\sigma_i(s)}$ (thus $Q^d $ by $Q^{\sigma(d)}$), $\epsilon_{a}$ by $\epsilon_{\sigma(a)}$, and $C_{a,k}$ by $C_{\sigma(a),k}$.
\end{itemize}
Note that at the current moment, $\epsilon_a$ and $C_{a,k}$ are regarded as independent variables and not yet functions of the Novikov variables. With these changes incorporated, the family $\mathcal{F}$ becomes
\begin{align*}
& \sum_{d\in\mathcal{D}} I_d|_{P_{is}\mapsto P_{i,\sigma_i(s)}} \cdot Q^{\sigma(d)} \cdot \exp{\left(\sum_{k>0,a\in A} \frac{\Psi^k(\epsilon_{\sigma(a)})P^{k\cdot \sigma(a)}q^{k\langle a,d\rangle}}{k(1-q^k)}\right)} \cdot \exp{\left(\sum_{k\in\mathbb{Z},a\in A} C_{\sigma(a),k} P^{\sigma(a)} q^{\langle a,d\rangle + k}\right)} \\
& = O_1\cdot O_2\cdot \sum_{d\in\mathcal{D}} I_d|_{P_{is}\mapsto P_{i,\sigma_i(s)}} \cdot Q^{\sigma(d)} = O_1\cdot O_2\cdot I,
\end{align*}
which is still itself. Here $O_1$ and $O_2$ are the two PFD operators that appeared at the beginning of this proof. The first equality holds because $\langle \sigma(a),\sigma(d) \rangle = \langle a,d \rangle = \sum_{i,s}a_{is}d_{is}$; the second holds because $I$ is $W$-invariant. In this way, we have proved that $\mathcal{F}$ is $W$-invariant. On the other hand, the intended input $1-q+\t$ is $W$-invariant by our assumption. Therefore, the formal function 
\[
F(Q,\epsilon,C(q)) = [\mathcal{F}(Q,\epsilon,C(q))]_+ - (1-q + t(Q,q))
\]
satisfies $\sigma \circ F \circ \sigma^{-1} = F$. 

\par We may now apply Lemma \ref{formalIFT} to $F$. In conclusion, for any given intended input $\t(q) \in (\Lambda^Y_+\otimes K(Y))^W[q,q^{-1}]$, there exist solutions $\epsilon_a = \epsilon_a(Q) \in \Lambda^Y_+/(Q)[[Q]] = \Lambda^Y_+ \ (a\in A)$ and $C_{k,a} = C_{k,a}(Q)\in \Lambda^Y_+/(Q)[[Q]][q,q^{-1}] = \Lambda^Y_+[q,q^{-1}] \ (k\in\mathbb{Z},a\in A)$, such that the $\mathcal{K}_+$-part of $\mathcal{F}(\epsilon,C(q))$ is $1-q+\t(q)$. In this way, the family $\mathcal{F}$ covers the entire $W$-invariant part of $\L^{tw,Y}$, and thus the entire $\L^X$ under the specialization $y = 1$ and $Q_{is} = Q_i \ (\forall i,s)$. This completes our proof of Theorem \ref{Reconstruction}.

\section{Twisted theories} \label{twist}
\par In this section, we consider different twisted quantum K-theories of $X$ as applications of the non-isolated recursion techniques developed in Section \ref{NonisoRec}.

\subsection{Euler-type twistings and non-abelian quantum Lefschetz} \label{eulertwisting}

\par We will start with the Euler-type twisting defined already in Section \ref{Jfunc}. Let $E$ be a vector bundle over $X = \text{Flag}(v_1,\cdots, v_n; N)$. The $(\Eu,E)$-twisted big $\mathcal{J}$-function is the generating function of correlators defined with respect to the modified virtual structure sheaves
\[
\mathcal{O}^{\virt}_{0,m,d} \otimes \Eu(\ft_*\ev^* E),
\]
and the correspondingly modified Poincar\'e pairings on $K(X)$. We denote its image cone by $\mathcal{L}^{X,(\Eu,E)}$. Since $E$ may be written as a symmetric Laurent polynomial of the K-theoretic Chern roots $P_{is}$ of the tautological bundles $V_i$ of $X$, the same expression in $P_{is}$ gives rise to a well-defined bundle over $Y$, which we still call $E$. We denote by $\mathcal{J}^{tw,Y,(\Eu,E)}$ the aforementioned $(\Eu,y^{-1}\mathfrak{g}/\mathfrak{s})$-twisted big $\J$-function of $Y$ further twisted by $(\Eu,E)$, and its image cone by $\mathcal{L}^{tw,Y,(\Eu,E)}$.

\par We consider first the case of $E = \mu^{-1}V_j^\vee = \mu^{-1}\sum_{s=1}^{v_j} P_{js}^{-1}$ for $1\leq j\leq n$ as an example so that our notation is not too cumbersome. The general case is completely parallel. Here $\mu$ is once again the equivariant parameter of an auxiliary fiber-wise $\mathbb{C}^\times$-action, and we take it into consideration to ensure convergence. To be more precise, we enlarge the coefficient ring into $\widetilde{\Lambda}^Y = \Lambda^Y[[\mu]]$ and take $\widetilde{\Lambda}^Y_+ = \Lambda^Y_+ + (\mu)$. Since the Euler-type twisting can be written into exponential form
\[
\Eu(\ft_*\ev^* E) = \prod_{\text{Chern roots } L}(1-L^{-1}) = \exp{\sum_{\text{Chern roots } L\ } \sum_{k<0} \frac{L^k}{k}} = \exp{\sum_{k<0}\frac{\Psi^k(E_{0,m,d})}{k}},
\]
we may apply (the genus-zero version of) the quantum Adams-Riemann-Roch (qARR) theorem \cite{Givental:perm11} which gives
\[
\mathcal{L}^{tw,Y} = \Box^Y_j\cdot\mathcal{L}^{tw,Y,(\Eu,\mu^{-1}V_j^\vee)}, \quad\text{where } \Box^Y_j = \exp{\sum_{k<0}\frac{\Psi^k(\mu^{-1}E)}{k(1-q^k)}} = \exp{\sum_{k>0}\sum_{s=1}^{v_j}\frac{\mu^kP_{js}^kq^k}{k(1-q^k)}}.
\]
Following the idea of non-abelian localization, we may compose $\Box^Y_j$ with PFD operators to form 
\[
\bigtriangleup^Y_j = \exp{\sum_{k>0}\sum_{s=1}^{v_j}\frac{\mu^kP_{js}^kq^k(1-q^{kQ_{js}\partial_{Q_{js}}})}{k(1-q^k)}},
\]
which not only still maps $\mathcal{L}^{tw,Y,(\Eu,\mu^{-1}V_j^\vee)}$ to $\mathcal{L}^{tw,Y}$, but the dilaton shift $(1-q)$ to itself as well. $\bigtriangleup^Y_j$ is well-defined as the exponential of a ``convergent'' sequence here since the $\widetilde{\Lambda}^Y_+$-adic topology is assumed. Moreover, as we have seen earlier (e.g. Section \ref{StructureProofThm2}), this new operator can be regarded as the asymptotic expression of a ratio of infinite products, and thus has even nicer effect on J-functions.
\[
(\bigtriangleup^Y_j)^{-1}\cdot Q^d = \prod_{s=1}^{v_j}\prod_{l=1}^{d_{js}}(1-\mu P_{js}q^l) \cdot Q^d, \quad\quad\text{for } Q^d = \prod_{i=1}^n\prod_{s=1}^{v_i}Q_{is}^{d_{is}}.
\]
Applying $(\bigtriangleup^Y_j)^{-1}$ to $J^{tw,Y} = \sum_{d\in\mathcal{D}}Q^d J^{tw,Y}_d$, we obtain a point on $\mathcal{L}^{tw,Y,(\Eu,\mu^{-1}V_j^\vee)}$, $\widetilde{\Lambda}^Y_+$-close to $(1-q)$:
\[
J^{tw,Y,(\Eu,\mu^{-1}V_j^\vee)} = (1-q)\sum_{d\in \mathcal{D}}\ Q^d \cdot J^{tw,Y}_d \cdot \prod_{s=1}^{v_j}\prod_{l=1}^{d_{js}}(1-\mu P_{js}q^l).
\]
It can also be checked directly that it satisfies the properly twisted recursive relations on $Y$. 

\par We will show that this function descends to a point on $\mathcal{L}^{X,(\Eu,\mu^{-1}V_j^\vee)}$ as $\widetilde{T}\rightarrow T, y\rightarrow 1$ and $Q_{is}\rightarrow Q_i$. The proof comprises two aspects. First, this new function does not contain new poles. It looks like a coincidence in our case here that the extra factors all appear on the numerator instead of the denominator, but this is not the essential reason. In fact, the appearance of $\mu$ in the twisting terms indicates that even if they do appear in denominator (as happens in some other cases), we should regard them rather as their $q$-(or more precisely, $\mu$-)power series expansions. Second, at the existing poles of $J^{tw,Y,(\Eu,\mu^{-1}V_j^\vee)}$ we should show that under the specialization $\wT\rightarrow T, Q_{is}=Q_i$ and $y=1$, the residues satisfy the properly twisted recursive relations on $X$. Note that the \emph{isolated} double-twisted (i.e. by $(\Eu,y^{-1}\mathfrak{g}/\mathfrak{s})$ and then by $(\Eu,\mu^{-1}V_j^\vee)$) recursion coefficients between non-degenerate $\widetilde{T}$-fixed points on $Y$ descend correctly to the $(\Eu,\mu^{-1}V_j^\vee)$-twisted recursion coefficients between $T$-fixed points on $X$, which is rather direct from what we have proved in Section \ref{IsoRec}. In fact, compared to Section \ref{IsoRec}, the only modification necessary here is to take into consideration the changes coming from the $(\Eu,E)$-twisting in the two recursion coefficients, where $E = \mu^{-1}V_j^\vee$ for $X$ and $E = \mu^{-1}\sum_{s=1}^{v_j} P_{js}^{-1}$ for $Y$. Along the 1-dim orbit $AB$, both changes take the form 
\[
\frac{\Eu(E_{0,2,mD})|_\phi}{\Eu(E)|_A},
\]
where $\phi$ is as always the ramified $m$-sheet covering to $AB$. When $A$ and $B$ are both non-degenerate $T$-fixed points, the bundle $\mu^{-1}V_j^\vee$ of $X$ may genuinely be identified with $\mu^{-1}\sum_{s=1}^{v_j} P_{js}^{-1}$ on $Y$ over the entire $T$-orbit $AB$ (not only in the sense of Proposition \ref{descend}), so the two recursion coefficients accord when we take $\widetilde{T}\rightarrow T, Q_{is}\rightarrow Q_i$ and $y = 1$.

\par Yet once again, we are not finished until we prove that the contribution from the non-isolated orbits vanishes. The strategy in Section \ref{NonisoRec} carries almost entirely over, and the only missing part is to show that over any given balanced broken $\wT$-orbit, for example from $E_{\II}$ to $E_{\II\cup\JJ}$, the change on $S_{\II,\JJ}(\mathcal{G})$ due to the $(\Eu,\mu^{-1}V_j^\vee)$-twisting comes from the restriction (to $\wT$-fixed points) of a well-defined global expression on the moduli space of all balanced broken $T$-orbits from $E_{\II}$ with the same homological degree. Without loss of generality, we consider as example the (unique) broken orbit from $A = E_\phi$ to $E_{\JJ}$ on $Y$, where $\{n\}\subset \JJ\subset\{1,2,\cdots,n\}$, which is an element in $\M_A(\Lambda_{v_1+1}/\Lambda_1)$ (for notation see Section \ref{StepI}). Recall that the orbit takes the form
\[
A = E_{\phi} = E_{\JJ[0]}\rightarrow E_{\JJ[1]}\rightarrow \cdots\rightarrow E_{\JJ[K-1]}\rightarrow E_{\JJ[K]} = E_{\JJ},
\]
and thus the change imposed by the $(\Eu,\mu^{-1}V_j^\vee)$-twisting on its total recursion coefficient $S_{\JJ}$ is
\[
\mathcal{T}_{\JJ}(\mu^{-1}V_j^\vee) := \prod_{k=0}^{|\JJ|-1} \frac{\Eu((\mu^{-1}V_j^\vee)_{0,2,m\cdot D_{E_{\JJ[k]}E_{\JJ[k+1]}}})|_{\phi_{E_{\JJ[k]}E_{\JJ[k+1]}}}}{\Eu(\mu^{-1}V_j^\vee)|_{E_{\JJ[k]}}} = \prod_{s=1}^{v_j} \mathcal{T}_{\JJ}(\mu^{-1}P_{js}^{-1}).
\]
We will prove that $\mathcal{T}_{\JJ}(\mu^{-1}V_j^\vee)$ is equal to the restriction at $\JJ$ of certain common expression $\mathcal{T} = \mathcal{T}(P_1,\cdots,P_{n-1}) \in K_{\wT}(\M_A(\Lambda_{v_1+1}/\Lambda_1))$, well-defined under the specialization $\wT\rightarrow T, Q_{is} = Q_i$ and $y=1$. Here $\{P_i\}_{i=1}^{n-1}$ are the ($\wT$-equivariant) tautological bundles on $\M_A(\Lambda_{v_1+1}/\Lambda_1)$ that we defined in Section \ref{StepI}.

\par It suffices to prove it for each factor $\mathcal{T}_{\JJ}(\mu^{-1}P_{js}^{-1})$. In fact, the K-theoretic Chern roots $\mu^{-1}P_{js}^{-1}$ are not special here, and we may simply consider the general case of $\mathcal{T}_{\JJ}(\mu^{-1}L^{-1})$ for any $\wT$-equivariant line bundle $L$ over $Y$. For simplicity of notation, denote by $\alpha_i = -\langle c_1(L), \mathbf{1}_{i1} \rangle\ (1\leq i\leq n)$, and by $\nu = L|_A\in R(\wT)$. Through explicit computation, 
\begin{align*}
\mathcal{T}_{\JJ}(\mu^{-1}L^{-1}) & = \prod_{k=0}^{|\JJ|-1} \frac{\Eu((\mu^{-1}L^{-1})_{0,2,m\cdot D_{E_{\JJ[k]}E_{\JJ[k+1]}}})|_{\phi_{E_{\JJ[k]}E_{\JJ[k+1]}}}}{\Eu(\mu^{-1}L^{-1})|_{E_{\JJ[k]}}} \\
& = \prod_{k=0}^{|\JJ|-1} \ \prod_{l=1}^{-m \cdot\langle c_1(L), D_{E_{\JJ[k]}E_{\JJ[k+1]}} \rangle} \left(1 - \mu L|_{E_{\JJ[k]}}\cdot \lambda_{j_{k+1}}^{\frac{l}{m}}\right) \\
& = \prod_{i=1}^n \prod_{l=1}^{m\alpha_i} \left( 1 - \mu\nu\prod_{\iota=1}^{i-1}\lambda_{j_{k_{\JJ}[\iota]}}^{\alpha_{\iota}} \cdot \lambda_{j_{k_{\JJ}[i]}}^{\frac{l}{m}} \right),
\end{align*}
since the homological degree of the irreducible component $E_{\JJ[k]}E_{\JJ[k+1]}$ is $\sum_{a=j_k+1}^{j_{k+1}} \mathbf{1}_{a1}$ for $\JJ = \{j_1,\cdots,j_{|\JJ|}\}$. Here $j_{k_{\JJ}[i]}$ is defined by $j_{k_{\JJ}[i]-1}< i \leq j_{k_{\JJ}[i]}$ as in Section \ref{StepI}, and $\lambda_{j_{k_{\JJ}[i]}}^{1/m}$ is restriction to $\JJ$ of a global bundle by Corollary \ref{tautocor}. Therefore, $\mathcal{T}_{\JJ}(\mu^{-1}L^{-1})$ is indeed restriction to $\JJ$ of a global bundle $\mathcal{T}(\mu^{-1}L^{-1})$ on $\M_A(\Lambda_{v_1+1}/\Lambda_1)$ as well. Due to the existence of $\mu$, the factors of $\mathcal{T}(\mu^{-1}L^{-1})$ are always invertible, even as we take $\Lambda_{is}=1 \ (i\leq n)$ and $\Lambda_{n+1,s} = \Lambda_s$, which means $\mathcal{T}(\mu^{-1}L^{-1})$ is well-defined under the specialization $\wT\rightarrow T, Q_{is}=Q_i$ and $y=1$.

\par This completes the proof of our claim that $J^{tw,Y,(\Eu,\mu^{-1}V_j^\vee)}$ descends to a point on $\mathcal{L}^{X,(\Eu,\mu^{-1}V_j^\vee)}$. The point can be explicitly written as
\[
J^{X,(\Eu,\mu^{-1}V_j^\vee)} = (1-q)\sum_{d\in \mathcal{D}} \ \prod_{i=1}^n Q_i^{\sum_{s=1}^{v_i} d_{is}} \cdot J^X_d \cdot \prod_{s=1}^{v_j}\prod_{l=1}^{d_{js}}(1-\mu P_{js}q^l).
\]
Here $J^X_d$ is the coefficient of $\prod_i Q_i^{\sum_s d_{is}}$ in the summand of $J^X$ (see Corollary \ref{MainCor}) associated to $d\in\mathcal{D}$. Alternatively, a less ambiguous way of defining $J^X_d$ is as the limit of $J^{tw,Y}_d$ under the specialization $\wT\rightarrow T, Q_{is}=Q_i$ and $y=1$, for any $d\in\mathcal{D}$.

\par It is not hard to see that the argument above works for any given $E$ in $K(X)$. Besides, the starting point $J^{tw,Y}$ is non-essential either and may be replaced by any point $\L^{tw,Y}$ invariant under the action of the Weyl group $W$. Expressing $E$ in terms of $P = (P_{is})_{i=1,s=1}^{n \quad v_i}$ and the $T$-equivariant parameters $\lambda = (\Lambda_s)_{s=1}^N$, we may decompose it formally into a ($W$-invariant) linear combination of monomials $L_k^{-1}$
\[
E = \mu^{-1} \sum_{k=1}^r \pm L_k^{-1}(P,\lambda).
\]
We formulate the general result as the non-abelian quantum Lefschetz theorem below. 
\begin{theorem}[Non-Abelian Quantum Lefschetz Theorem]
Suppose
\[
I = \sum_{d\in\mathcal{D}} Q^d \cdot I_d(P,\Lambda,y;q) 
\]
is a $W$-invariant point on the $\wT$-equivariant image cone $\L^{tw,Y}$, where $P = (P_{is})_{i=1,s=1}^{n \quad v_i}$ and $\Lambda = (\Lambda_{is})_{i=2,s=1}^{n+1 \ v_i}$. Then,
\[
I^{(\Eu,E)} = \sum_{d\in\mathcal{D}} \ \prod_{i=1}^n Q_{i}^{\sum_{s=1}^{v_i}d_{is}} \cdot I_d(P,\lambda,1;q) \cdot \prod_{k=1}^r \prod_{l=1}^{ - \langle c_1(L_k),d \rangle} (1- \mu L_k(P,\lambda)\cdot q^l)^{\pm 1}
\]
locates on the $T$-equivariant image cone $\L^{X,(\Eu,E)}$ under the specialization $\Lambda_{is}=1\ (i\leq n)$, $\Lambda_{n+1,s} = \Lambda_s$, $Q_{is} = Q_i$ and $y=1$.
\end{theorem}
$I^{(\Eu,E)}$ is invariant under permutation of $s$ in $P_{is}$, and is thus a well-defined $q$-rational function with coefficients in $K_T(X)[[Q_1,\cdots,Q_n,\mu]]$.
\begin{remark}
One may derive an explicit reconstruction theorem (like Theorem \ref{Reconstruction}) for $\L^{X,(\Eu,E)}$ starting from any $W$-invariant point on $\L^{tw,Y,(\Eu,E)}$ as well. The same surjectivity argument as in Section \ref{recon} may be applied without change.
\end{remark}

\subsection{More examples}
\par More generally, the same argument works for any exponential-type twisting, which takes the form (see \cite{Givental:perm11})
\[
\mathcal{O}^{\text{virt}}_{0,m,d}\longmapsto \mathcal{O}^{\text{virt}}_{0,m,d}\otimes \exp{\sum_{k\neq 0}\frac{\Psi^{k}(E^{(k)}_{0,m,d})}{k}},
\]
where $E^{(k)}$ are vector bundles on $X$ and $(\cdot)_{g,m,d}$ stands for $\ft_*\ev^*(\cdot)$ (see Section \ref{Jfunc}). One special case is
\[
\mathcal{O}^{\text{virt}}_{0,m,d}\longmapsto \mathcal{O}^{\text{virt}}_{0,m,d}\otimes \Eu^{-1}((\hbar T^*X)_{0,m,d}^\vee)
\]
where $\hbar$ is the equivariant parameter of a fiber-wise $\mathbb{C}^\times$-action on $T^*X$. Such twisting is of exponential-type because we have the expansion
\[
\Eu^{-1}((\hbar T^*X)_{0,m,d}^\vee) = \exp{\sum_{k>0} \frac{\Psi^k((\hbar T^*X)_{0,m,d})}{k}},
\]
convergent in $\hbar$-adic topology. We denote by $\mathcal{L}^\mathcal{E}$ the image cone of the big $\J$-function so twisted. By an argument similar to the one above, one may find that the point $J^\mathcal{E} = (1-q)\sum_d Q^d J^\mathcal{E}_d$ resides in $\mathcal{L}^\mathcal{E}$ where
\[
J^\mathcal{E}_d = J^X_d \cdot \frac{\prod_{i=1}^n\prod_{1\leq r\leq v_{i+1}}^{1\leq s\leq v_i}\prod_{l=0}^{d_{is}-d_{i+1,r}-1}(1-\hbar\frac{P_{is}}{P_{i+1,r}}q^l)}{\prod_{i=1}^n\prod_{r\neq s}^{1\leq r,s\leq v_i}\prod_{l=0}^{d_{is}-d_{ir}-1}(1-\hbar\frac{P_{is}}{P_{ir}}q^l)}.
\] 
Generating functions twisted in this way are important as they are very often ``balanced'', an essential property for applying rigidity arguments \cite{Okounkov}. More precisely, terms in the numerator and denominator come in pairs in the form of
\[
\frac{1-\hbar\cdot x\cdot q^l}{1 - q\cdot x\cdot q^l},
\]
and thus the limit exists as $x^{\pm 1}\rightarrow\infty$. Liu \cite{Liu} has a deeper discussion on when the property holds. 

\par The expression matches the \textbf{vertex functions} appearing in \cite{PSZ} for the special case of grassmannians and in \cite{KPSZ} for the case of (type-A) flag varieties, up to a power of $q$ and $\hbar$ and a sign which can all absorbed by variable change of the Novikov variables. Note that the settings of the above two papers are dual to those of this present paper in the definition of grassmannians and flag varieties.

\par Following the argument after Theorem 4 of \cite{Givental-Yan} and dividing $J^\mathcal{E}$ by $\Eu(\hbar^{-1}TX)$, we will obtain a point on the cone $\mathcal{L}^{X,(\Eu,\hbar^{-1}TX)}$ with Euler-type twisting. As $\hbar\rightarrow 1$, it reduces to the image cone of the (untwisted) big $\J$-function of $Z\subset X$, where $Z$ is the zero locus in $X$ of a generic section of $TX$, i.e. the disjoint union of $\chi(X)$ isolated points. 

\par A variant of the quantum K-theory considered in this paper is the quantum Hirzebruch K-theory. Correlators of the latter are defined with the virtual structure sheaves of the moduli spaces modified by certain ($y$-equivariant) Euler classes of the virtual tangent bundles. Here $y$ comes from the Hirzebruch $\chi_{-y}$-genus and means differently from the auxiliary equivariant parameter that we have used in previous sections. By \cite{Irit}, dividing the above point on $\mathcal{L}^{X,(\Eu,TX)}$ further by $1-yq$, we obtain a value of the big $\mathcal{J}$-function of the quantum Hirzebruch K-theory.

\subsection{Level structures}
\par Given a vector bundle $E$ on $X$, the \textbf{level-$l$ structure of $E$} gives the following modification to the virtual structure sheaves
\[
\mathcal{O}^{\text{virt}}_{g,m,d}\longmapsto \mathcal{O}^{\text{virt}}_{g,m,d}\otimes \det^{-l}(E_{g,m,d}).
\]
We denote by $\J^{X,(E,l)}$ the associated big $\J$-function, and by $\L^{X,(E,l)}$ its image cone. 

\par As usual, writing $E$ in terms of the K-theoretic Chern roots $P_{is}$, we obtain a well-defined element in $K(Y)$, which we also call $E$ through abuse of notation. We denote by $\J^{tw,Y,(E,l)}$ the big $\J$-function of $Y$ with the usual Euler-type twisting $(\Eu,y^{-1}\mathfrak{g}/\mathfrak{s})$ and then the level structure $(E,l)$, and by $\L^{tw,Y,(E,l)}$ its image cone.

\par In this section, we illustrate only through the example $E = V_j$ for some specific $j$ to avoid over-complicated notation. We locate one point on $\L^{X,(V_j,l)}$ and thus prove Theorem \ref{MainTheoremLevel}. As may be seen from the proof, other cases are no different. Moreover, an argument almost identical to that in Section \ref{recon} will provide an explicit reconstruction of the entire $\L^{X,(V_j,l)}$. We omit the details here.

\par The recursive formulae on $Y$ associated to $\L^{tw,Y,(V_j,l)}$ takes the explicit form
\[
\Res_{q = \widetilde{\lambda}^{1/m}}\f_a(q)\frac{dq}{q} = \coeff^{tw,Y}_{ab}(m)\cdot \frac{\det^{-l}((V_j)_{0,2,mD})|_\phi}{\det^{-l}(V_j)|_a} \cdot \f_b(\widetilde{\lambda}^{1/m})
\]
along any 1-dim $\wT$-orbit $ab$ with tangent $\wT$-character $\widetilde{\lambda}$ at $a$ and homological degree $D$. It is not hard to show that the following $q$-rational function satisfies such recursive formulae
\[
J^{tw,Y,(V_j,l)} = \sum_{d\in\mathcal{D}} \prod_{i,s} Q_{is}^{d_{is}} \cdot J^{tw,Y}_d \cdot \left[\prod_{s=1}^{v_j} P_{js}^{d_{js}}q^{\frac{d_{js}(d_{js}-1)}{2}}\right]^l.
\]
It may be checked through direct computation. The intuition of such determinantal modification to the $J$-function comes from the existing formulae in \cite{R-Z} and in \cite{Wen}.

\par We expect that setting $\wT\rightarrow T$, $Q_{is} = Q_i$ and $y=1$ in $J^{tw,Y,(V_j,l)}$ will give us a point on $\L^{X,(V_j,l)}$, and this is indeed true. The argument for Euler-type twistings in Section \ref{eulertwisting} may be transplanted here for level structures without too much pain. Once again, the only difficulty is to show that the contribution from the non-isolated recursion coefficients, i.e. the contribution from balanced broken orbits, eventually vanishes. Once again, it amounts to prove that in the twisted recursion coefficient associated to broken orbit $\JJ$, the extra terms due to the level structures
\[
\mathcal{R}_{\JJ} := \prod_{k=0}^{|\JJ|-1} \frac{\det^{-l}((V_j)_{0,2,mD_{E_{\JJ[k]}E_{\JJ[k+1]}}})|_{\phi_{E_{\JJ[k]}E_{\JJ[k+1]}}}}{\det^{-l}(V_j)|_{E_{\JJ[k]}}},
\]
come from the restriction to $\JJ$ of a global expression on the moduli $\M_A(\Lambda_{v_1+1}/\Lambda_1)$ of all balanced broken orbits with the same homological degree as $\JJ$. One may be tempted to check it directly by hand as before, but there is a faster way. In fact, it suffices to observe that for any vector space $W$,
\[
\det^{-l} W = \left(\frac{\Eu(\mu W)}{\Eu(\mu^{-1} W^\vee)}\right)^l \cdot (-1)^{l\dim W}\quad\quad\quad \text{ at } \mu=1.
\]
By applying the above formula to $W = V_j|_{E_{\JJ[k]}}$ and $W = (V_j)_{0,2,mD_{E_{\JJ[k]}E_{\JJ[k+1]}}}|_{\phi_{E_{\JJ[k]}E_{\JJ[k+1]}}}$, we may express $\mathcal{R}_{\JJ}$, up to signs, as the ratio of (the $l$-th powers of) Euler-type twistings, which have been proven to come from global expressions on $\M_A(\Lambda_{v_1+1}/\Lambda_1)$. The remaining extra signs, which on orbit $ab$ with homological degree $D = (D_{is})$ amount to
\[
(-1)^{l\cdot \rank (V_j)_{0,2,mD} - l\cdot \rank V_j} = (-1)^{l\cdot \langle c_1(V_j), D\rangle} =  (-1)^{l\cdot \sum_{s=1}^{v_j}D_{js}},
\]
depend linearly on $D$ and come thus from a (constant) global expression on $\M_A(\Lambda_{v_1+1}/\Lambda_1)$ as well. In conclusion, 
\[
J^{X,(V_j,l)} = (1-q)\sum_{d\in \mathcal{D}} \ \prod_{i=1}^n Q_i^{\sum_s d_{is}} \cdot J^X_d \cdot \left[\prod_{s=1}^{v_j} P_{js}^{d_{js}}q^{\frac{d_{js}(d_{js}-1)}{2}}\right]^l,
\]
which is the outcome of setting $\wT\rightarrow T$, $Q_{is} = Q_i$ and $y=1$ in $J^{tw,Y,(V_j,l)}$, gives a point on $\L^{X,(V_j,l)}$.
\begin{remark}
The observation above seems to provide a way of systematically expressing level structures in terms of the more classical Euler-type twistings up to signs. However, the concurrent appearance of both $\mu$ and $\mu^{-1}$ implies that the technical issue of convergence may arise in certain cases. In fact, the Euler-type twisting $(\Eu,\mu^{-1}E)$ is usually well-defined only when we add $\mu$ into the base coefficient ring $\Lambda^X$ and consider the adic topology associated to $\mu$. Therefore, one has to treat the limit at $\mu = 1$ with extra care. Our argument above is legitimate as we do only algebraic manipulation of the recursion coefficients, which are always well-defined even when $\mu = 1$.
\end{remark}

\par $J^{X,(V_j,l)}$ above is actually the small $J$-function (i.e. with zero input) of the corresponding theory with level structures when $l \leq \left[\prod_{i=j}^n(v_{i+1}-v_i)\right]/(n-j+1)$, which may be proved by estimating the degrees of $q$ in the numerators and denominators for each $Q^d$. The proof is purely technical and highly resemblant to that in the grassmannian case \cite{Givental-Yan}, so we omit the details here.

\par A correspondence between big $\mathcal{J}$-functions with level structures of dual flag varieties still holds, generalizing the one for dual grassmannians discovered in \cite{Dong-Wen}. To be more precise, there is a $T$-equivariant isomorphism between the flag variety
\[
X = \text{Flag}(v_1,v_2,\cdots,v_n;N)
\]
and
\[
X' = \text{Flag}(N-v_n,N-v_{n-1},\cdots,N-v_1;N),
\]
where $T$ acts on the ambient $\mathbb{C}^N$ of the former with characters $\Lambda_1,\cdots,\Lambda_N$ but on that of the latter with characters all inverted, and the isomorphism is given explicitly by
\[
0\subset V_1\subset V_2\subset\ldots\subset V_n\subset\mathbb{C}^N \longmapsto 0\subset (V_n)^{\perp}\subset (V_{n-1})^{\perp}\subset\ldots\subset (V_1)^\perp \subset (\mathbb{C}^N)^*.
\]
Both have $n$ tautological bundles, and we name them $V_i$ and $V_i'$ respectively. Note that $\rank V_i = v_i$ and $\rank V_i' = N-v_{n-i+1}$. Pulling back $V_i'$ to $X$ through the isomorphism described above, it becomes the dual quotient bundle $(\mathbb{C}^N/V_i)^\vee$. Under such identification, we formulate the so-called \textbf{level correspondence} between dual flag varieties as below.
\begin{theorem}[Level Correspondence] \label{levelcorrespondence}
\[
\mathcal{L}^{X,(V_i,l)} = \mathcal{L}^{X',((V_i')^\vee,-l)}, \quad \forall 1\leq i\leq n.
\]
\end{theorem}
This is because
\[
\det(\ft_*\ev^*V_i)\ \otimes\ \det(\ft_*\ev^*(V_i')^\vee) = \det(\ft_*\ev^* (V_i\ \oplus\ \mathbb{C}^N/V_i)) = \Lambda_1\cdots\Lambda_N,
\]
and the constant scaling here is compensated by the discrepancy of Poincar\'e pairings. As a corollary, when $|l|$ is small (so that the explicit description of small $J$-function given earlier holds for both sides), the two small $J$-functions coincide.

\section{K-theoretic Mirrors} \label{chap5}

\par The goal of this section is to represent (components of) the small $J$-function $J^X$ as Jackson-type integrals (i.e. $q$-integrals) over lattices on torus. Throughout the section we assume $|q|>1$. We start by reviewing the properties of some model $q$-integrals which will be used in an essential way when we build up the mirrors. 

\subsection{Two model integrals}

\par Consider the $q$-function
\[
f(x) = x^{\frac{\ln \lambda}{\ln q}} \prod_{m=1}^{\infty}(1-x/q^m)
\]
where $\lambda$ is regarded as a parameter for now. Since we have the asymptotic expression near the unit circle
\[
\prod_{m=1}^{\infty}(1-x/q^m) \sim e^{\sum_{k>0}x^k/k(1-q^k)},
\] 
in many cases we will abuse the notations and denote $f(x)$ also by
\[
f(x) = x^{\frac{\ln \lambda}{\ln q}} e^{\sum_{k>0}x^k/k(1-q^k)}
\]
to maintain visual resemblance with the complex oscillating integrals which appear in quantum cohomology. It is not hard to check that 
\[
(1-q^{x\partial_x}/\lambda) f(x) = x \cdot f(x) = q^{\lambda\partial_{\lambda}} f(x).
\]
Consider the following improper $q$-integral for $f(x)$
\[
F(\lambda) = \int_0^{\infty}f(x) x^{-1}d_qx := \sum_{d\in\mathbb{Z}} f(q^{-d}).
\]
The integral is actually computed on the ``lattice'' $\{q^{-d}\}_{d\in\mathbb{Z}}$. Since the lattice is invariant under the multiplicative $q$-translation $q^{x\partial_x}$, we have
\[
(1-\lambda^{-1}) F(\lambda) = q^{\lambda\partial_{\lambda}} F(\lambda).
\]
Therefore, when $|\lambda|>1$, we have
\[
F(\lambda) = c\cdot \sum_{d\geq 0}\frac{\lambda^{-d}}{(1-q^{-1})(1-q^{-2})\cdots (1-q^{-d})} = \frac{c}{\prod_{m=0}^{\infty}(1-q^{-m}/\lambda)}
\]
where $c$ is the limit of $F$ as $\lambda$ approaches infinity. As the rightmost expression is an analytic function for $\lambda\neq q^{-m}\ (m\geq 0)$, we may extend the domain of definition of $F(\lambda)$ to $\mathbb{C}-\{q^{-m}|m\geq 0, m\in\mathbb{Z}\}$ by analytic continuation. 
\begin{remark}
In order to compute $c$, note that the integrand vanishes at $x = q^{-d}$ for $d>0$. Therefore,
\begin{align*}
F(\lambda) & = \sum_{d\geq 0}\lambda^{-d}\prod_{m=d+1}^{\infty}(1-q^{-m}) \\
& = \prod_{m=1}^{\infty}(1-q^{-m}) \cdot \sum_{d\geq 0}\frac{\lambda^{-d}}{(1-q^{-1})(1-q^{-2})\cdots (1-q^{-d})}.
\end{align*}
In other words, $c = \prod_{m=1}^{\infty}(1-q^{-m})$.
\end{remark}

\par Similarly, consider the $q$-function
\[
g(y)=\frac{y^{1 + \frac{\ln \lambda}{\ln q}}}{\prod_{m=0}^{\infty}(1-y/q^m)}.
\]
It has the asymptotic expression 
\[
g(y) \sim y^{1 + \frac{\ln \lambda}{\ln q}} \cdot e^{-\sum_{k>0}q^ky^k/k(1-q^k)},
\]
and satisfies the $q$-difference equations
\[
(1-\lambda q\cdot q^{-y\partial_y}) g(y) = y \cdot g(y) = q^{\lambda\partial_{\lambda}}g(y).
\]
Because of the factors $(1-y/q^m)$ in the denominator, the $q$-integral above is not defined for $g(y)$. Hence we consider instead the shifted improper $q$-integral
\[
G(\lambda) = \int_{0}^{-\infty} g(y) y^{-1} d_qy := \sum_{d\in\mathbb{Z}} g(-q^{-d}).
\]
In other words, the integral is taken actually over the lattice $\{-q^{-d}\}_{d\in\mathbb{Z}}$. Similar integrals may be defined over shifted lattices of the form $\{q^{-d}/A\}_{d\in\mathbb{Z}}$  for any $A\neq 1$, and the results share similar properties. Here the choice $A = -1$ is not special and is purely for notational simplicity. 

\par From the $q$-difference equations of $g$, we have
\[
(1-\lambda q) G(\lambda) = q^{\lambda\partial_{\lambda}} G(\lambda).
\]
When $|\lambda|<1$,
\[
G(\lambda) = c \cdot \prod_{m=-\infty}^0 (1-\lambda q^m) = c\cdot \sum_{d\geq 0} \frac{\lambda^d q^d}{(1-q)\cdots(1-q^d)}.
\]
As $|q|>1$, the rightmost expression is actually an entire function of $\lambda$, so the domain of definition of $G(\lambda)$ may be extended to the entire $\mathbb{C}$ by analytic continuation.

\par The two integrals above are closely related to the $q$-gamma functions. See \cite{DK} for a modern treatment of it.

\subsection{Construction}

\par Now we start to construct the mirrors. Recall that the small $J$-function $J^X$ of the flag variety $X = \text{Flag}(v_1,\cdots,v_n;N)$ is obtained from a point $J^{tw,Y}$ on the twisted big $\mathcal{J}$-function of the abelian quotient $Y$, which itself is a toric variety, by taking certain limits of the Novikov variables and equivariant parameters. Hence we first build up the mirror for the twisted theory of $Y$. Recall that
\[
J^{tw,Y} = (1-q)\sum_{d\in \mathcal{D}} \prod_{i,s} Q_{is}^{d_{is}}\frac{\prod_{i=1}^n\prod_{r\neq s}^{1\leq r,s\leq v_i}\prod_{l=1}^{d_{is}-d_{ir}}(1-y\frac{P_{is}}{P_{ir}}q^l)}{\prod_{i=1}^n\prod_{1\leq r\leq v_{i+1}}^{1\leq s\leq v_i}\prod_{l=1}^{d_{is}-d_{i+1,r}}(1-\frac{P_{is}}{P_{i+1,r}\Lambda_{i+1,r}}q^l)},
\]
where $\Lambda_{i+1,r}$ are equivariant parameters for the action on $Y$ by the enlarged torus $\widetilde{T}$.

\par It should be made clear that when we take limit of the Novikov variables and the equivariant parameters, the resulting expressions may no longer have the nice properties that we need for $\J$-functions. For example, even in the case of grassmannians, when we take $y=1$ but not $Q_{is}=Q_i$, the resulting expression will have poles at unexpected positions like $q = (\frac{\Lambda_1}{\Lambda_2})^{1/m}$, or non-simple poles at legitimate positions like $q = (\frac{\Lambda_2}{\Lambda_1})^{1/m}$, which disqualifies it from residing on the image of any big $\mathcal{J}$-function (twisted or untwisted) of any target variety. These poles will only cancel out when we take the limit of Novikov variables as well. Yet, these expressions are always well-defined as $q$-functions, and still satisfy their own finite-difference equations. In fact, these equations come exactly from taking the same limit of coefficients in the finite-difference equations satisfied by $J^{tw,Y}$. Therefore, we may dive freely into the algebraic constructions in this section, without worrying about the same issues as in previous sections when, for instance, we compute the recursion coefficients which should receive special consideration only because they are not well-defined at the limit we take.

\par Now, through direct computation, we write out the $q$-difference equations satisfied by $J^{tw,Y}$. Fix any $1\leq i\leq n$ and $1\leq s\leq v_i$, we have 
\begin{align*}
& \prod_{1\leq r\leq v_{i+1}}(1-\frac{P_{is}q^{Q_{is}\partial_{Q_{is}}}}{P_{i+1,r}q^{Q_{i+1,r}\partial_{Q_{i+1,r}}}}\cdot\frac{1}{\Lambda_{i+1,r}}) \prod_{1\leq r\leq v_i}^{r\neq s}(1-y\cdot q\cdot\frac{P_{ir}q^{Q_{ir}\partial_{Q_{ir}}}}{P_{is}q^{Q_{is}\partial_{Q_{is}}}}) J^{tw,Y} \\
= & Q_{is} \prod_{1\leq r\leq v_{i-1}}(1-\frac{P_{i-1,r}q^{Q_{i-1,r}\partial_{Q_{i-1,r}}}}{P_{is}q^{Q_{is}\partial_{Q_{is}}}}\cdot\frac{1}{\Lambda_{is}}) \prod_{1\leq r\leq v_i}^{r\neq s}(1-y\cdot q\cdot\frac{P_{ir}q^{Q_{is}\partial_{Q_{is}}}}{P_{is}q^{Q_{ir}\partial_{Q_{ir}}}}) J^{tw,Y}.
\end{align*}
\par As K-theoretic mirrors, we will construct $q$-integrals over certain multiplicative $q$-lattices on a family of tori, which satisfy the same $q$-difference equations as above. The family of tori is parameterized by the Novikov variables $Q = (Q_{is})_{i=1,s=1}^{n\quad v_i}$, and takes the form
\[
\mathcal{X}_{Q} = \left\{(X,Y)\in\mathcal{X}|\frac{\prod_{r=1}^{v_{i+1}}X_{isr}}{\prod_{r=1}^{v_{i-1}}X_{i-1,rs}} = Q_{is}\frac{\prod_{1\leq s'\leq v_i}^{s'\neq s}Y_{iss'}}{\prod_{1\leq s'\leq v_i}^{s'\neq s}Y_{is's}},\ \ \forall i,s\right\},
\]
where $\mathcal{X}$ is the ambient complex torus with multiplicative coordinates $X_{isr} (1\leq i\leq n, 1\leq s\leq v_i, 1\leq r\leq v_{i+1})$ and $Y_{iss'} (1\leq i\leq n, 1\leq s,s'\leq v_i,s\neq s')$ of dimension
\[
\dim \mathcal{X} =  \sum_{i} v_i(v_{i+1}+v_i-1).
\]
Note that we define $v_0 = 0$ for convenience of notations. Fixing the Novikov variables, each torus in the family has dimension
\[
\dim \mathcal{X}_{Q} = \sum_{i} v_i(v_{i+1}+v_i-2).
\]
For simplicity of notations, we name the two index sets for the variables $X_{isr}$ and the variables $Y_{iss'}$ as
\begin{align*}
\mathfrak{I}_X & := \{(i,s,r)| 1\leq i\leq n, 1\leq s\leq v_i, 1\leq r\leq v_{i+1} \},\\
\mathfrak{I}_Y & := \{(i,s,s')| 1\leq i\leq n, 1\leq s,s'\leq v_i, s\neq s'\}.\\
\end{align*}
Over the family of tori $\mathcal{X}_{Q}$, we consider the $q$-integral
\begin{align*}
\mathcal{I}^{\widetilde{T}} = \int_{\Gamma\subset\mathcal{X}_{Q}} & \prod_{i,s,r} X_{isr}^{\ln\Lambda_{i+1,r}/\ln q} \cdot \prod_{i,s,s'} Y_{iss'}^{1+\ln y/\ln q} \cdot e^{\sum_{k>0}\frac{\sum_{i,s,r} X_{isr}^k - q^k\sum_{i,s,s'} Y_{iss'}^k}{k(1-q^k)}}\\
& \cdot \frac{(\bigwedge_{i,s,r}d_q\ln X_{isr})\wedge(\bigwedge_{i,s,s'}d_q\ln Y_{iss'})}{\bigwedge_{1\leq i\leq n, 1\leq s\leq v_i} (\sum_{r=1}^{v_{i+1}}d_q\ln X_{isr} - \sum_{r=1}^{v_{i-1}}d_q\ln X_{i-1,rs} - \sum_{1\leq s'\leq v_i}^{s'\neq s}d_q\ln\frac{Y_{iss'}}{Y_{is's}})},
\end{align*}
where the products and sums with omitted index sets are taken either over $(i,s,r)\in\mathfrak{I}_X$ or $(i,s,s')\in\mathfrak{I}_Y$. Here the cycle $\Gamma$ of integration is a $q$-lattice in $\mathcal{X}_{Q}$ of rank $\sum_{i} v_i(v_{i+1}+v_i-2)$, or more precisely a linear combination of such cycles. The choices of such cycles in order to recover the components of $J^{tw,Y}$ will be specified as we proceed through the computations below. The ``$q$-volume form'' is invariant under under multiplicative $q$-translation. In other words, when any of $X_{isr}$ or $Y_{iss'}$ is multiplied by a power of $q$, the volume form remains the same. Besides, one note is that in spite of the complicated expression of the denominator, one may think simply of $\bigwedge_{i,s} d_q\ln Q_{is}$.

\par $\mathcal{I}^{\widetilde{T}}$ satisfies a similar set of $q$-difference equations as $J^{tw,Y}$ does, which we state as the following theorem.
\begin{theorem}
Fix any $1\leq i\leq n$ and $1\leq s\leq v_i$,
\begin{align*}
& \prod_{1\leq r\leq v_{i+1}}(1-\frac{q^{Q_{is}\partial_{Q_{is}}}}{q^{Q_{i+1,r}\partial_{Q_{i+1,r}}}}\cdot\frac{1}{\Lambda_{i+1,r}}) \prod_{1\leq s'\leq v_i}^{s'\neq s}(1-y\cdot q\cdot\frac{q^{Q_{is'}\partial_{Q_{is'}}}}{q^{Q_{is}\partial_{Q_{is}}}}) \ \mathcal{I}^{\widetilde{T}} \\
= & Q_{is} \prod_{1\leq r\leq v_{i-1}}(1-\frac{q^{Q_{i-1,r}\partial_{Q_{i-1,r}}}}{q^{Q_{is}\partial_{Q_{is}}}}\cdot\frac{1}{\Lambda_{is}}) \prod_{1\leq s'\leq v_i}^{s'\neq s}(1-y\cdot q\cdot\frac{q^{Q_{is}\partial_{Q_{is}}}}{q^{Q_{is'}\partial_{Q_{is'}}}}) \ \mathcal{I}^{\widetilde{T}}.
\end{align*}
\end{theorem}
In fact, the $q$-difference equations in this theorem may be obtained by removing the tautological bundles which appear in the $q$-difference equations satisfied by $J^{tw,Y}$ and replacing $s'$ for some of the indices $r$.

\par In order to prove the theorem, we start by analyzing the effect of $q^{X_{isr}\partial_{X_{isr}}}$ on the integrand of $\mathcal{I}^{\widetilde{T}}$. The only factors related to $X_{isr}$ are
\[
X_{isr}^{\ln\Lambda_{i+1,r}/\ln q} \cdot e^{\sum_{k>0}\frac{X_{isr}^k}{k(1-q^k)}},
\]
for which we have 
\begin{align*}
& q^{X_{isr}\partial_{X_{isr}}} \cdot (X_{isr}^{\ln\Lambda_{i+1,r}/\ln q} \cdot e^{\sum_{k>0}\frac{X_{isr}^k}{k(1-q^k)}})\\
= & (qX_{isr})^{\ln\Lambda_{i+1,r}/\ln q} \cdot e^{\sum_{k>0}\frac{q^kX_{isr}^k}{k(1-q^k)}}\\
= & q^{\ln\Lambda_{i+1,r}/\ln q} \cdot X_{isr}^{\ln\Lambda_{i+1,r}/\ln q} \cdot e^{-\sum_{k>0}\frac{x^k}{k} + \frac{X_{isr}^k}{k(1-q^k)}}\\
= & \Lambda_{i+1,r}(1-X_{isr}) \cdot (X_{isr}^{\ln\Lambda_{i+1,r}/\ln q} \cdot e^{\sum_{k>0}\frac{X_{isr}^k}{k(1-q^k)}}).
\end{align*}
On the other hand, for any integrand $h = h(X,Y)$ where $X=(X_{isr})$ and $Y=(Y_{iss'})$, due to the expression of $\mathcal{X}_{Q}$, the effect of $q^{X_{isr}\partial_{X_{isr}}}$ may be interpreted alternatively in the following way
\begin{align*}
& \int_{\Gamma\subset\mathcal{X}_{Q}} q^{X_{isr}\partial_{X_{isr}}} \cdot h(X,Y) \ \frac{d_q\ln X\wedge d_q\ln Y}{d_q\ln Q}\\
= & \int_{\Gamma\subset\mathcal{X}_{Q}} h(\cdots,q X_{isr},\cdots) \ \frac{d_q\ln X\wedge d_q\ln Y}{d_q\ln Q}\\
= & \int_{\Gamma'\subset\mathcal{X}_{\cdots qQ_{is}\cdots q^{-1}Q_{i+1,r}\cdots}} h(\cdots, X_{isr}',\cdots) \ \frac{d_q\ln X'\wedge d_q\ln Y}{d_q\ln Q'} \quad (\text{ take }X_{isr}' = qX_{isr})\\
= & q^{Q_{is}\partial_{Q_{is}} - Q_{i+1,r}\partial_{Q_{i+1,r}}} \int_{\Gamma\subset\mathcal{X}_{Q}} h(X',Y) \ \frac{d_q\ln X'\wedge d_q\ln Y}{d_q\ln Q'}\\
\end{align*}
Here $d_q\ln X\wedge d_q\ln Y/d_q\ln Q$ stands for the $q$-translation-invariant $q$-volume form in our integral and $\Gamma'$ for the image of $\Gamma$ under the change of variable. Note that the Novikov variables in the denominator of the $q$-volume form are abused notation. They are the corresponding expressions purely in terms of $X$ and $Y$, and thus are not affected by operators like $q^{Q_{is}\partial_{Q_{is}}- Q_{i+1,r}\partial_{Q_{i+1,r}}}$. Combining the two results above, we have derived that the operator
\[
(1-\frac{q^{Q_{is}\partial_{Q_{is}}}}{q^{Q_{i+1,r}\partial_{Q_{i+1,r}}}}\cdot\frac{1}{\Lambda_{i+1,r}})
\]
applied to $\mathcal{I}^{\widetilde{T}}$ has the effect of multiplying the integrand of $\mathcal{I}^{\widetilde{T}}$ by the factor $X_{isr}$. Similarly, the terms in the integrand of $\mathcal{I}^{\widetilde{T}}$ related to $Y_{iss'}$ are 
\[
Y_{iss'}^{1+\ln y/\ln q} \cdot e^{-\sum_{k>0}\frac{q^k Y_{iss'}^k}{k(1-q^k)}},
\]
for which we have
\begin{align*}
& q^{-Y_{iss'}\partial_{Y_{iss'}}} \cdot (Y_{iss'}^{1+\ln y/\ln q} \cdot e^{-\sum_{k>0}\frac{q^k Y_{iss'}^k}{k(1-q^k)}})\\
= & y^{-1}q^{-1} (1-Y_{iss'}) \cdot (Y_{iss'}^{1+\ln y/\ln q} \cdot e^{-\sum_{k>0}\frac{q^k Y_{iss'}^k}{k(1-q^k)}}).
\end{align*}
As a result, we may show that the operator 
\[
(1-y\cdot q\cdot\frac{q^{Q_{is}\partial_{Q_{is}}}}{q^{Q_{is'}\partial_{Q_{is'}}}})
\]
applied to $\mathcal{I}^{\widetilde{T}}$ gives rise to the factor $Y_{iss'}$. 

\par Now, we are ready to look back at the $q$-difference equations in the theorem. With $i$ and $s$ fixed, operators on LHS and RHS applied to $\mathcal{I}^{\widetilde{T}}$ produce the factors
\[
\prod_{1\leq r\leq v_{i+1}} X_{isr} \prod_{1\leq s'\leq v_i}^{s'\neq s}Y_{is's}, \ \text{  and  } \ Q_{is} \prod_{1\leq r\leq v_{i-1}} X_{i-1,rs} \prod_{1\leq s'\leq v_i}^{s'\neq s} Y_{iss'}
\]
respectively, which are equal by definition of $\mathcal{X}_{Q}$. This completes our proof of the theorem.

\par The similarity between the $q$-difference equations satisfied by $J^{tw,Y}$ and $\mathcal{I}^{\widetilde{T}}$ suggests a more direct connection between the two $q$-functions. This is achieved by careful selections of the integration cycle $\Gamma$. We start with description of such cycles. Consider a sequence of injective maps $S = \{S_{i}\}_{i=1}^n$ where
\[
S_i: \{1,2,\cdots, v_i\}\longrightarrow \{1,2,\cdots, v_{i+1}\}.
\]
The defining equations of $\mathcal{X}_{Q}$ allows us to express the variables $X_{i,s,S_i(s)}$, which are determined by the sequence of maps $S$, in terms of the Novikov variables $Q_{is}$ and the remaining ``free'' variables $X_{isr}\ (r\neq S_i(s))$ and $Y_{iss'}$. For instance, taking $S_i = \text{id}$ for all $i$, we may express $X_{iss}$ as
\[
X_{iss} = Q_{is} \cdot \prod_{1\leq s'\leq v_i}^{s'\neq s}\frac{Y_{iss'}}{Y_{is's}} \cdot \prod_{1\leq r\leq v_{i+1}}^{r\neq s}\frac{1}{X_{isr}} \cdot \prod_{r=1}^{v_{i-1}}X_{i-1,rs}.
\]
Note that not all $X_{i-1,rs}$ appearing in the formula are free variables when $s\leq v_{i-1}$, in which case we have to further reduce the expression. Assuming that $v_{j-1}+1\leq s \leq v_j$ for some $1\leq j\leq i$ (i.e. $j$ is the smallest index such that $Q_{js}$ is defined), we would eventually reach
\begin{align*}
X_{iss} = & \prod_{k=j+1}^i \left( Q_{ks}\cdot \prod_{1\leq s'\leq v_k}^{s'\neq s}\frac{Y_{kss'}}{Y_{ks's}} \cdot \prod_{1\leq r\leq v_{k+1}}^{r\neq s}\frac{1}{X_{ksr}} \cdot \prod_{1\leq r\leq v_{k-1}}^{r\neq s}X_{k-1,rs} \right)\\
& \ \ \cdot Q_{js}\cdot \prod_{1\leq s'\leq v_j}^{s'\neq s}\frac{Y_{jss'}}{Y_{js's}} \cdot \prod_{1\leq r\leq v_{j+1}}^{r\neq s}\frac{1}{X_{jsr}} \cdot \prod_{1\leq r\leq v_{j-1}}X_{j-1,rs}.
\end{align*}
In this way, we may lift any $q$-lattice in the space of the free variables to one on the torus $\mathcal{X}_{Q}$, and we define the integration cycle $\Gamma_S$ associated to $S$ to be the lift of
\[
L_S = \prod_{(i,s,r)\in \mathfrak{I}_X}^{r\neq S_i(s)}\{X_{isr}=q^{-d} | d\in\mathbb{Z}\} \ \times \ \prod_{(i,s,s')\in \mathfrak{I}_Y}\{Y_{iss'}=-q^d | d\in\mathbb{Z}\}.
\]
We denote by $\mathcal{I}^{\widetilde{T}}_S$ the integral over $\Gamma_S$.

\par Each such sequence of maps $S$ naturally gives rise to a $T$-fixed point on $Y$, and the injectivity of $S_i$ guarantees that the fixed point descends to some point on $X$. More precisely, recall that any point on $Y$ may be represented by a sequence of matrices in $R = M_{v_2\times v_1}(\mathbb{C})\times\cdots\times M_{N\times v_n}(\mathbb{C})$, then the fixed point to which $S$ corresponds is represented by 
\[
\left( \sum_{i=1}^{v_1}E_{S_1(i),i}, \cdots, \sum_{i=1}^{v_n}E_{S_n(i),i} \right) \in R = M_{v_2\times v_1}(\mathbb{C})\times\cdots\times M_{N\times v_n}(\mathbb{C}),
\]
where $E_{ab}$ refers to the elementary matrix (of proper size) with the $(a,b)$-entry being $1$ and all other entries vanishing. 

\par Eventually we will prove that $\mathcal{I}^{\widetilde{T}}$ integrated over the cycle $\Gamma_S$ is simply a multiple of $J^{tw,Y}$ localized at the corresponding fixed point, and thus, under the appropriate limit of the equivariant parameters and the Novikov variables, a multiple of $J^X$ localized at the corresponding fixed point. We still take the case where $S_i = \text{id}$ for all $i$ as an example. For simplicity of notation, we denote it by $S = \text{Id} = \{S_i = \text{id}\}_{i=1}^n$. Note its corresponding fixed point on $Y$ is exactly our old friend $A$ (see Section \ref{NonisoRec}). Down to the flag variety $X$, $S = \text{Id}$ gives rise to the $T$-invariant coordinate flag
\[
\Span\{\e_1,\cdots,\e_{v_1}\}\subset \Span\{\e_1,\cdots,\e_{v_2}\} \subset\cdots\subset \Span\{\e_1,\cdots,\e_{v_n}\} \subset \mathbb{C}^N
\]
which we also denoted by $A$ by abuse of notation.

\par In this case, the factors associated to $X_{iss}$ in the integrand are
\[
X_{iss}^{\ln\Lambda_{i+1,s}/\ln q} \cdot e^{\sum_{k>0}\frac{X_{iss}^k}{k(1-q^k)}}.
\]
Replacing each $X_{iss}$ with the formula above in terms of the free variables, and rewriting the $q$-exponential functions through
\[
e^{\sum_{k>0}\frac{X_{iss}^k}{k(1-q^k)}} = \sum_{d_{is}\geq 0}\frac{X^{d_{is}}_{iss}}{(1-q)(1-q^2)\cdots(1-q^{d_{is}})},
\]
we expand $\mathcal{I}^{\widetilde{T}}_{\Id}$ into an infinite sum indexed by powers of Novikov variables
\[
\mathcal{I}^{\widetilde{T}}_{\Id} = \sum_{d_{is}\geq 0} \ \prod_{i=1}^{n}\prod_{s=1}^{v_i}\frac{Q_{is}^{\sum_{k=i}^{n}d_{is}+\frac{\ln\Lambda_{i+1,s}}{\ln q}}}{(1-q)(1-q^2)\cdots(1-q^{d_{is}})} \cdot \mathcal{I}_{\Id}^{\{d_{is}\}},
\]
where $\mathcal{I}_{\Id}^{\{d_{is}\}}$ is independent of the Novikov variables. The term $Q_{is}^{\sum_{k=i}^{n}d_{is}+\frac{\ln\Lambda_{i+1,s}}{\ln q}}$ appears because $Q_{is}$ arises (and only arises) in the expansion of $X_{kss}$ with $k\geq i$. Applying the change of variables
\[
D_{is} = \sum_{k=i}^n d_{is}, \quad (\forall \ 1\leq i\leq n, 1\leq s\leq v_i)
\]
or equivalently
\[
d_{is} = D_{is}-D_{i+1,s}, \quad (\forall \ 1\leq i\leq n, 1\leq s\leq v_i)
\]
we may re-write the sum above as
\[
\mathcal{I}^{\widetilde{T}}_{\Id} = \sum_{D_{is}\geq 0}\prod_{i=1}^{n}\prod_{s=1}^{v_i}\frac{Q_{is}^{D_{is}+\frac{\ln\prod_{k=i}^{n}\Lambda_{i+1,s}}{\ln q}}}{\prod_{l=1}^{D_{is}-D_{i+1,s}}(1-q^l)} \cdot \mathcal{I}_{\Id}^{\{D_{is}\}},
\]
where, by explicit computation, $\mathcal{I}_{\Id}^{\{D_{is}\}} = $
\begin{align*}
\pm \int_{L_{\Id}} & \prod_{(i,s,r)\in \mathfrak{I}_X}^{r\neq s}X_{isr}^{\frac{\ln \Lambda_{i+1,r}}{\ln q}}X_{isr}^{-D_{is}+D_{i+1,r}+ \frac{\ln \frac{\prod_{k=i+2}^{n+1}\Lambda_{kr}}{\prod_{k=i+1}^{n+1}\Lambda_{ks}}}{\ln q}} \cdot 
\prod_{(i,s,s')\in \mathfrak{I}_Y}Y_{iss'}^{1+\frac{\ln y}{\ln q}}Y_{iss'}^{D_{is}-D_{is'}+ \frac{\ln \frac{\prod_{k=i+i}^{n+1}\Lambda_{ks}}{\prod_{k=i+1}^{n+1}\Lambda_{ks'}}}{\ln q}} \cdot \\
& \exp{\left(\sum_{k>0} \frac{\sum_{(i,s,r)\in \mathfrak{I}_X}^{r\neq s} X_{isr}^k - q^k\sum_{i,s,s'\in\mathfrak{I}_Y}Y_{iss'}^k}{k(1-q^k)}\right)} \cdot \bigwedge_{(i,s,r)\in \mathfrak{I}_X}^{r\neq s}d_q \ln X_{isr} \wedge \bigwedge_{(i,s,s')\in \mathfrak{I}_Y} d_q \ln Y_{iss'}.
\end{align*}
Here $L_{\Id}$ is the $q$-lattice in the space of free variables whose lift is $\Gamma_{\Id}$. By definition of $D_{is}$ and the requirement that $d_{is} = D_{is}-D_{i+1,s}\geq 0$, the term associated to a specific set of $D_{is}$ in $\mathcal{I}^{\widetilde{T}}_{\Id}$ should only exist if $D_{is}\geq D_{i+1,s}$ for all possible $i$ and $s$. In this sense, our notation for the denominator
\[
\prod_{l=1}^{D_{is}-D_{i+1,s}}(1-q^l) = \frac{\prod_{l=-\infty}^{D_{is}-D_{i+1,s}}(1-q^l)}{\prod_{l=-\infty}^{0}(1-q^l)}
\] 
is convenient since it will give a vanishing factor $(1-q^0) = 0$ in the numerator when $D_{is}-D_{i+1,s}<0$. 

\par Moreover, in order to understand the expression of $\mathcal{I}_{\Id}^{\{D_{is}\}}$, note that
\begin{itemize}
	\item for a fixed 3-tuple $(i,s,r)\in\mathfrak{I}_X$ with $r\neq s$, $X_{isr}$ appears in the expansion of $X_{i+1, rr}, X_{i+2,rr},$ $\cdots,X_{nrr}$, while $X_{isr}^{-1}$ appears in the expansion of $X_{iss}, X_{i+1,ss}, \cdots, X_{nss}$;
	\item for a fixed 3-tuple $(i,s,s')\in\mathfrak{I}_Y$, $Y_{iss'}$ appears in the expansion of $X_{iss}, X_{i+1,ss}, \cdots, X_{nss}$, while $Y_{iss'}^{-1}$ appears in the expansion of $X_{is's'}, X_{i+1.s's'},\cdots, X_{ns's'}$.
\end{itemize}
Recall that tautological bundles of the abelian quotient $Y$ are denoted by $P_{is}\ (1\leq i\leq n, 1\leq s\leq v_i)$. Localized at the $\widetilde{T}$-fixed point $A$ of $Y$ corresponding to $S=\Id$ as previously explained, these bundles become
\[
P_{is}|_A = \prod_{k=i+1}^{n+1}\Lambda_{ks}.
\]
Therefore, the expression may be simplified as $\mathcal{I}_{\Id}^{\{D_{is}\}} = $
\begin{align*}
\pm \int_{L_{\Id}} & \prod_{(i,s,r)\in \mathfrak{I}_X}^{r\neq s}X_{isr}^{-D_{is}+D_{i+1,r}+ \frac{\ln \frac{P_{i+1,r}\Lambda_{i+1,r}}{P_{is}}|_A}{\ln q}} \cdot 
\prod_{(i,s,s')\in \mathfrak{I}_Y} Y_{iss'}^{D_{is}-D_{is'} + 1 + \frac{\ln y\frac{P_{is}}{P_{is'}}|_A}{\ln q}} \cdot \\
& \exp{\left(\sum_{k>0} \frac{\sum_{(i,s,r)\in \mathfrak{I}_X}^{r\neq s} X_{isr}^k - q^k\sum_{i,s,s'\in\mathfrak{I}_Y}Y_{iss'}^k}{k(1-q^k)}\right)} \cdot \bigwedge_{(i,s,r)\in \mathfrak{I}_X}^{r\neq s}d_q \ln X_{isr} \wedge \bigwedge_{(i,s,s')\in \mathfrak{I}_Y} d_q \ln Y_{iss'}.
\end{align*}

\par Finally, we remark that the sign in the front of $\mathcal{I}_{\Id}^{\{D_{is}\}}$ emerges as we descend the $q$-volume form on $\Gamma_{\Id}$ to the one above on $L_{\Id}$. We complete our definition of the integration cycle $\Gamma_{\Id}$ here by endowing it with the orientation such that the sign in $\mathcal{I}_{\Id}^{\{D_{is}\}}$ is $+$. (And the orientation of other integration cycles $\Gamma_S$ may be chosen following the same idea.) Consequently, assembling the terms in $\mathcal{I}_{\Id}^{\{D_{is}\}}$ associated to each $X_{isr}$ and $Y_{iss'}$, we obtain
\[
\mathcal{I}_{\Id}^{\{D_{is}\}} = \prod_{(i,s,r)\in\mathfrak{I}_X}^{r\neq s} I^{(D_{is}-D_{i+1,r})}_{+,X_{isr}}\left(\frac{P_{i+1,r}\Lambda_{i+1,r}}{P_{is}}\middle|_A\right) \cdot \prod_{(i,s,s')\in\mathfrak{I}_Y} I^{(D_{is'}-D_{is})}_{-,Y_{iss'}}\left(y\frac{P_{is}}{P_{is'}}\middle|_A\right),
\]
where
\begin{align*}
I^{(d)}_{+,x}(\lambda) := \int_{0}^{\infty} \gamma^{(d)}_{+}(x;\lambda)\ d_q x & := \int_{0}^{\infty} e^{\sum_{k>0}\frac{x^k}{k(1-q^k)}} \cdot x^{-d+\frac{\ln\lambda}{\ln q}}\ d_q\ln x;\\
I^{(d)}_{-,y}(\lambda) := \int_{0}^{-\infty} \gamma^{(d)}_{-}(y;\lambda)\ d_q y & := \int_{0}^{-\infty} e^{-\sum_{k>0}\frac{q^kx^k}{k(1-q^k)}} \cdot y^{1-d+\frac{\ln\lambda}{\ln q}}\ d_q\ln x.
\end{align*}
It may be checked through direct computation that
\begin{align*}
(1-q^dq^{x\partial_x}\lambda^{-1})\cdot \gamma^{(d)}_{+}(x;\lambda) & = \gamma^{(d-1)}_{+}(x;\lambda);\\
(1-q^{1-d}q^{-y\partial_y}\lambda)\cdot \gamma^{(d)}_{-}(y;\lambda) & = \gamma^{(d-1)}_{-}(y;\lambda).
\end{align*}
Since the $q$-translation operators $q^{x\partial_x}$ and $q^{-y\partial_y}$ do not change the integrals, we obtain by induction that
\begin{align*}
I^{(d)}_{+,x}(\lambda) & = \frac{I^{(0)}_{+,x}(\lambda)}{\prod_{l=1}^d(1-\lambda^{-1}q^l)};\\
I^{(d)}_{-,y}(\lambda) & = I^{(0)}_{-,y}(\lambda) \cdot \prod_{l=1}^{-d}(1-\lambda q^l).
\end{align*}
Substituting back into $\mathcal{I}^{\widetilde{T}}_{\Id}$, we recover a multiple of $J^{tw,Y}|_A$. More precisely,
\[
\mathcal{I}^{\widetilde{T}}_{\Id} = \frac{\mathcal{I}^{\{0\}}_{\Id}}{1-q} \cdot \prod_{i=1}^n\prod_{s=1}^{v_i}Q_{is}^{\frac{\ln P_{is}|_A}{\ln q}} \cdot J^{tw,Y}|_A.
\]
Here $\mathcal{I}^{\{0\}}_{\Id}$ means the value of $\mathcal{I}_{\Id}^{\{D_{is}\}}$ when all $D_{is} = 0$. It is a $q$-function with coefficients in the representation ring of $\widetilde{T}$ but independent of the Novikov variables. It is now clear why $\mathcal{I}^{\widetilde{T}}$ satisfies a similar set of $q$-difference equations as $J^{tw,Y}$ does. In fact, $\mathcal{I}^{\{0\}}_{\Id}/(1-q)$ can entirely be regarded as a ``constant'', while $Q_{is}^{\frac{\ln P_{is}|_A}{\ln q}}$ conjugates $P_{is}|_A q^{Q_{is}\partial_{Q_{is}}}$ into $q^{Q_{is}\partial_{Q_{is}}}$. More precisely, given any function $h$ of the Novikov variables
\[
q^{Q_{is}\partial_{Q_{is}}} \cdot \left( Q_{is}^{\frac{\ln P_{is}|_A}{\ln q}} \cdot h(Q) \right) = Q_{is}^{\frac{\ln P_{is}|_A}{\ln q}} \cdot P_{is}|_A q^{Q_{is}\partial_{Q_{is}}} \cdot \left( h(Q) \right).
\]
The same argument works for the integrals $\mathcal{I}^{\widetilde{T}}_{S}$ over all other cycles $\Gamma_S$, and the tautological bundles that appear in the argument should then be localized to, instead of $A$, the $\widetilde{T}$-fixed point of $Y$ corresponding to $S$.

\par Therefore, under the specialization
\[
Q_{is}\rightarrow Q_i, \quad \Lambda_{is}\rightarrow 1 \ (1\leq i\leq n, 1\leq s\leq v_i), \quad \Lambda_{n+1,s}\rightarrow \Lambda_s, \quad y\rightarrow 1,
\]
we obtain the following theorem.
\begin{theorem}
The $q$-integral
\begin{align*}
\mathcal{I}^{T} = \int_{\Gamma\subset\mathcal{X}_{(Q_{i})}} & \ \prod_{s=1}^{v_n}\prod_{r=1}^{N} X_{nsr}^{\ln\Lambda_{r}/\ln q} \cdot \prod_{(i,s,s')\in\mathfrak{I}_Y} Y_{iss'} \cdot \exp{\sum_{k>0}\frac{\sum_{(i,s,r)\in\mathfrak{I}_X} X_{isr}^k - q^k\sum_{(i,s,s')\in\mathfrak{I}_Y} Y_{iss'}^k}{k(1-q^k)}}\\
& \cdot \frac{\left(\bigwedge_{(i,s,r)\in\mathfrak{I}_X} d_q\ln X_{isr}\right) \wedge \left(\bigwedge_{(i,s,s')\in\mathfrak{I}_Y} d_q\ln Y_{iss'}\right)}{\bigwedge_{1\leq i\leq n, 1\leq s\leq v_i} \left(\sum_{r=1}^{v_{i+1}}d_q\ln X_{isr} - \sum_{r=1}^{v_{i-1}}d_q\ln X_{i-1,rs} - \sum_{1\leq s'\leq v_i}^{s'\neq s}d_q\ln\frac{Y_{iss'}}{Y_{is's}}\right)},
\end{align*}
with suitable choices of (linear combinations of) $q$-lattices $\Gamma$ on the torus
\[
\mathcal{X}_{(Q_{i})} = \left\{(X,Y)\in\mathcal{X}\ \middle| \ \frac{\prod_{r=1}^{v_{i+1}}X_{isr}}{\prod_{r=1}^{v_{i-1}}X_{i-1,rs}} = Q_{i}\frac{\prod_{1\leq s'\leq v_i}^{s'\neq s}Y_{iss'}}{\prod_{1\leq s'\leq v_i}^{s'\neq s}Y_{is's}},\ \ \forall i,s\right\},
\]
represents components of the $T$-equivariant small $J$-function $J^X$ of $X = \text{Flag}(v_1,\cdots, v_n; N)$.
\end{theorem}
By ``components'', we mean the coefficients when the $K_T(X)$-valued $q$-function $J^X$ is expressed in terms of a basis $\{\phi_\alpha\}$ of $K_T(X)$. When the basis is taken to be the $T$-fixed point classes, the coefficients are exactly the localization of $J^X$ to the fixed points (by replacing each $P_{is}$ by its restriction).
\begin{remark}
Taking the specialization
\[
Q_{is}\rightarrow Q_i, \quad \Lambda_{is}\rightarrow 1 \ (1\leq i\leq n, 1\leq s\leq v_i), \quad \Lambda_{n+1,s}\rightarrow \Lambda_s, \quad y\rightarrow 1
\]
is a well-defined operation on $\mathcal{I}^{\widetilde{T}}$. In fact, over each fixed integration cycle $\Gamma_S$, $\mathcal{I}^{\widetilde{T}}_{\Gamma_S}$ is a product of terms of the form
\[
I^{(D_{is}-D_{i+1,r})}_{+,X_{isr}} \left(\frac{P_{i+1,r}\Lambda_{i+1,r}}{P_{is}}\middle|_S\right) \ \text{ and }\ I^{(D_{is'}-D_{is})}_{-,Y_{iss'}} \left(y\frac{P_{is}}{P_{is'}}\middle|_S\right)
\]
where $S$ denotes the $\widetilde{T}$-fixed point on $Y$ corresponding to $\Gamma_S$. It suffices to show that these terms, which become
\[
I^{(D_{is}-D_{i+1,r})}_{+,X_{isr}}\left(\frac{P_{i+1,r}}{P_{is}}\middle|_S\right) \ \text{ and }\ I^{(D_{is'}-D_{is})}_{-,Y_{iss'}}\left(\frac{P_{is}}{P_{is'}}\middle|_S\right)
\]
under the specialization, are still well-defined. By our analysis of the integrals of $f$ and $g$ at the beginning of the section, $I^{(0)}_{+,x}(\lambda)$ and $I^{(0)}_{-,y}(\lambda)$ are analytic functions of $\lambda$ away from possible poles at non-positive powers of $q$. Since $\frac{P_{i+1,r}}{P_{is}}|_S$ and $\frac{P_{is}}{P_{is'}}|_S$ are always non-trivial monomials of equivariant parameters even after the specialization by our choice of $S$, and thus never powers of $q$, $I^{(0)}_{+,X_{isr}}\left(\frac{P_{i+1,r}}{P_{is}}|_S\right)$ and $I^{(0)}_{-,Y_{iss'}}\left(\frac{P_{is}}{P_{is'}}|_S\right)$ are always well-defined, so are all factors appearing in $\mathcal{I}^{T}_{S}$.
\end{remark}

\par Note that if one modifies any of the map $S_i$ in $S$ by pre-composing it by a permutation of the domain $\{1, 2, \cdots, v_i\}$, the resulting $S$ will still give rise to the same fixed point on $X$. Correspondingly, the integral over $\Gamma_S$ is also invariant under such permutations of $S$ under the above limit of the equivariant parameters and the Novikov variables.

\section*{Acknowledgement}
This material is based upon work supported by the National Science Foundation under Grant DMS-190632. The author would like to thank his PhD advisor Alexander Givental for helpful discussions, valuable advice and constant support. The author would also like to thank Yaoxiong Wen and Rachel Webb for discussions on quasimaps and the abelian/non-abelian correspondence.

\bibliographystyle{abbrv}
\bibliography{notes}

$\,$\
\noindent
\textsc{Sorbonne Université and Université Paris Cité, CNRS, IMJ-PRG, F-75005 Paris, France}

\textit{e-mail address:} \href{mailto:xiaohanyan@imj-prg.fr}{xiaohanyan@imj-prg.fr}

\end{document}

%% file: OrbitsY.pdf_tex
\begingroup%
  \makeatletter%
  \providecommand\color[2][]{%
    \errmessage{(Inkscape) Color is used for the text in Inkscape, but the package 'color.sty' is not loaded}%
    \renewcommand\color[2][]{}%
  }%
  \providecommand\transparent[1]{%
    \errmessage{(Inkscape) Transparency is used (non-zero) for the text in Inkscape, but the package 'transparent.sty' is not loaded}%
    \renewcommand\transparent[1]{}%
  }%
  \providecommand\rotatebox[2]{#2}%
  \newcommand*\fsize{\dimexpr\f@size pt\relax}%
  \newcommand*\lineheight[1]{\fontsize{\fsize}{#1\fsize}\selectfont}%
  \ifx\svgwidth\undefined%
    \setlength{\unitlength}{383.55974813bp}%
    \ifx\svgscale\undefined%
      \relax%
    \else%
      \setlength{\unitlength}{\unitlength * \real{\svgscale}}%
    \fi%
  \else%
    \setlength{\unitlength}{\svgwidth}%
  \fi%
  \global\let\svgwidth\undefined%
  \global\let\svgscale\undefined%
  \makeatother%
  \begin{picture}(1,0.44036062)%
    \lineheight{1}%
    \setlength\tabcolsep{0pt}%
    \put(0,0){\includegraphics[width=\unitlength,page=1]{OrbitsY.pdf}}%
    \put(0.10290662,0.40329566){\color[rgb]{0,0,0}\makebox(0,0)[lt]{\lineheight{1.25}\smash{\begin{tabular}[t]{l}$A$\end{tabular}}}}%
    \put(-0.00140635,0.03229076){\color[rgb]{0,0,0}\makebox(0,0)[lt]{\lineheight{1.25}\smash{\begin{tabular}[t]{l}$C = E_1$\end{tabular}}}}%
    \put(0.87160753,0.39769061){\color[rgb]{0,0,0}\makebox(0,0)[lt]{\lineheight{1.25}\smash{\begin{tabular}[t]{l}$E_2$\end{tabular}}}}%
    \put(0.87350936,0.0249459){\color[rgb]{0,0,0}\makebox(0,0)[lt]{\lineheight{1.25}\smash{\begin{tabular}[t]{l}$E_{1,2}$\end{tabular}}}}%
    \put(0.86780403,0.21131828){\color[rgb]{0,0,0}\makebox(0,0)[lt]{\lineheight{1.25}\smash{\begin{tabular}[t]{l}$\Omega$\end{tabular}}}}%
    \put(0.3980696,0.24174641){\color[rgb]{0,0,0}\makebox(0,0)[lt]{\lineheight{1.25}\smash{\begin{tabular}[t]{l}$T$-orbits\end{tabular}}}}%
  \end{picture}%
\endgroup%